\documentclass{article}

\usepackage[utf8]{inputenc}
\usepackage[british]{babel}
\usepackage{verbatim}
\usepackage{hyperref}
\usepackage{amsmath}
\usepackage{stmaryrd} 
\usepackage{amsfonts} 
\usepackage{amsthm}
\usepackage{xcolor}
\usepackage{amsmath}  
\usepackage{amsfonts}
\usepackage{amssymb}
\usepackage{mathtools}
\usepackage{graphicx}
\usepackage{subfig}
\usepackage[autostyle,italian=guillemets]{csquotes}
\usepackage{caption} 
\usepackage{booktabs}  
\addcontentsline{toc}{chapter}{Chapter1}
\newtheorem{theorem}{Theorem}[section]

\newtheorem{lemma}[theorem]{Lemma}

\newtheorem{thm}{Theorem}[section]
\newtheorem{defn}{Definition}[section]

\newtheorem{oss}{Remark}[section]

\newcommand{\refcite}{\cite}

\newcommand{\norm}[1]{{\left\Vert #1 
		\right\Vert}}
\newcommand{\R} {\mathbb{R}}

\newcommand\blfootnote[1]{%
  \begingroup
  \renewcommand\thefootnote{}\footnote{#1}%
  \addtocounter{footnote}{-1}%
  \endgroup
}

\begin{document}

\markboth{M.~Grasselli et al.}{Non-isothermal non-Newtonian fluids: the stationary case}

%
\title{Non-isothermal non-Newtonian fluids: the stationary case}

\date{}

\author
{Maurizio Grasselli, \, Nicola Parolini, \, Andrea Poiatti, \, Marco Verani \thanks{Dipartimento di Matematica,
Politecnico di Milano,
Milano I-20133, Italy. E-mail adresses: \textit{mau\-ri\-zio.grasselli@polimi.it}, \textit{nicola.parolini@polimi.it}, \textit{andrea.poiatti@polimi.it}, \textit{marco.verani@polimi.it}}}
\maketitle

\begin{abstract}
The stationary Navier-Stokes equations for a non-Newtonian incompressible fluid are coupled with the stationary heat equation and subject to Dirichlet type boundary conditions. The viscosity is supposed to depend on the temperature and the stress depends on the strain through a suitable power law depending on $p\in (1,2)$ (shear thinning case). For this problem we establish the existence of a weak solution as well as we prove some regularity results both for the Navier-Stokes and the Stokes cases. Then, the latter case with the Carreau power law is approximated through a FEM scheme and some error estimates are obtained. Such estimates are then validated through some two-dimensional numerical experiments.
\end{abstract}

\medskip


\blfootnote{The paper has been published in \textit{Mathematical Models and Methods in Applied Sciences}, 33(9) pp. 1747-1801, 2023, DOI: 10.1142/S0218202523500410}

\medskip

\section{Introduction}\setcounter{equation}{0}
In the last decades, the theoretical and numerical analysis of non-Newtonian fluids has seen a renew of interest, stimulated by the wide spectrum of applied problems, ranging from biological applications to industrial processes.
In particular, concerning the latter, we point out that many processes deal with materials exhibiting non-Newtonian behaviors. For instance, food processing, polymer manufacturing, tribology, injection molding of foams, rubber extrusion. Moreover, most of them are carried out under non-isothermal conditions. Motivated by these important applications, in the present paper we focus on a class of non-isothermal non-Newtonian fluid models (see, for instance, Refs. \refcite{Ahmed,Alsteens,Chiruvella:1996,farid,Ghoreishy,Ishikawa,Helmig} and references therein).
We recall that the description of a non-Newtonian fluid behavior is based on a power-law ansatz. We refer the reader to Refs. \refcite{Batchelor4} and \refcite{Lady} for a general continuum mechanical background
and to Refs. \refcite{M3,Bird6,23H,35malek,50Rue,44Ra} for a detailed discussion of several models for non-Newtonian fluids. In particular, for a large class of non-Newtonian fluids, the dominant departure
from a Newtonian behavior is that, in a simple shear flow, the viscosity and the
shear rate are not proportional. These are the so-called fluids with shear-dependent
viscosity. Here we are interested in the theoretical and numerical analysis of steady flows of such fluids accounting for the presence of a non-negligible temperature field. In particular, we consider the case where  the viscosity depends on the temperature. Very similar models, in the Newtonian case, have been analyzed, e.g., in Ref. \refcite{Bernardi-et-al:2015}. A slightly more complex model has been derived and investigated in Ref.  \refcite{C2} (see also the references therein). In that case, the authors consider a non-stationary model for incompressible homogeneous Newtonian fluids in a fixed bounded three-dimensional domain. 

Let us now describe the problem we want to analyze. We consider an incompressible fluid which occupies a bounded domain $\Omega\subset \mathbb{R}^d,$ $d=2,3,$ with a sufficiently smooth boundary. We denote by $\mathbf{u}:\Omega \to \mathbb{R}^d$ its velocity field and by $\Theta:\Omega \to \mathbb{R}$ its temperature.
We suppose that $(\mathbf{u},\Theta)$ satisfies the following equations
\begin{align}
-&\mbox{div}\left[\nu(\Theta) \mathbf{\tau}(x,\mathbf{\varepsilon}(\mathbf{u}))\right] +(\mathbf{u}\cdot\nabla)\mathbf{u}+\nabla\pi= \mathbf{f} \qquad \hbox{in }\Omega
\label{Sy01}\\
&\mbox{div }\mathbf{u}=0\qquad \hbox{in }\Omega \label{Sy02}\\
-&\kappa\Delta \Theta + \mathbf{u} \cdot \nabla\Theta = g \qquad \hbox{in }\Omega\label{Sy03}
\end{align}
endowed with the boundary conditions
\begin{equation}
\mathbf{u} = \mathbf{0}, \quad \Theta = \theta  \qquad \hbox{on }\partial \Omega.
\label{bc}
\end{equation}
Here  $\nu>0$ is the viscosity coefficient, $\tau$ denotes the stress tensor which is a suitable function of the strain rate tensor $\varepsilon(\mathbf{u})$ defined by $\varepsilon_{ij}(\mathbf{u}):=\frac{1}{2}(\partial_{x_i} u_j + \partial_{x_j}u_i)$. Moreover, $\pi$ is the fluid pressure, $\mathbf{f}$ is a given body force, $\kappa>0$ is the heat diffusion coefficient, $g$ is a known heat source, and $\theta$ is a given boundary temperature. {The Dirichlet boundary condition for the temperature is also assumed in Ref. \refcite{Bernardi-et-al:2015}. Neumann or Robin boundary conditions can also be considered though mixed ones can be problematic whenever regularity is required}.

Concerning the constitutive law which defines $\tau$, we observe that most real non-Newtonian fluids that can be modeled by a constitutive law such that
\begin{align}
\vert \tau(x, \mathbf{B})\vert \leq \tau_2(1 +  \vert \mathbf{B}\vert)^{p-1},\qquad x\in\Omega \nonumber
\end{align}
with $\mathbf{B}\in \mathbb{R}^{d \times d}$ symmetric tensor and  $\tau_2$ positive constant, are shear thinning fluids, namely, the shear exponent $p$, is ``small'', i.e., $p \in (1, 2)$ (cf. Ref. \refcite{36m} for a discussion of such models and further
references). This is the case we consider in the present contribution.

The mathematical analysis of the Navier-Stokes problem for non-Newtonian fluids started with
the work of Ladyžhenskaya.\cite{Lady} After the fundamental works of the ``Prague school'' led by Ne\v{c}as et. al. (see Refs. \refcite{Male,MNRR} and the references therein), the problem has been intensively studied and various existence and regularity
properties have been proved in the last years. The literature on this subject is rather vast. Thus we only mention some of the main   contributions which are mostly related to the stationary case. In the isothermal case, there are several results
on the existence of weak solutions,\cite{Stein} interior regularity\cite{Seregin} and regularity up-to-the boundary for the Dirichlet problem.\cite{Beirao2,Beirao1,Berselli,Crispo,Ruzi,Shilkin} Moreover, we refer to Ref.  \refcite{Malek} for some $C^{1,\alpha}$-regularity results. 
However, there are not so many contribution on the non-isothermal case. In Ref. \refcite{Rubi}, the author obtains the existence of a distributional solution to a steady-state system of equations for incompressible, possibly non-Newtonian of the $p$-power type, viscous flow coupled with the heat equation in a smooth
bounded domain of $\R^d$, $d=2,3$. Notice that in this model the fluid viscosity is considered to be independent of the temperature. In Ref. \refcite{Quasi}, the authors analyze a system of equations describing the stationary flow of a quasi-Newtonian
fluid, with temperature dependent viscosity and with a viscous heating. Existence of at least one appropriate weak solution is proved. 
In Ref. \refcite{Consiglieri} the   existence of weak solutions to the coupled system of stationary equations for a class of general non-Newtonian fluids with energy transfer is established.

Moreover, in the aforementioned Ref. \refcite{C2}, the authors establish the large-data and long-time existence of a suitable class of weak solutions to the corresponding problem regarding unsteady flows of incompressible homogeneous Newtonian fluids. In Sec.2 of Ref. \refcite{Bernardi-et-al:2015} the authors prove the existence and (conditional) uniqueness of weak solutions, together with some regularity results, to problem \eqref{Sy01}-\eqref{bc} in the case of Newtonian fluids, which is a simplification of the model treated in Ref. \refcite{C2}. In particular, it describes the stationary flow of a viscous incompressible Newtonian fluid, in the case where the viscosity of the fluid depends on the temperature. The mathematical analysis of a very similar problem in the case of a non-Newtonian fluid, with the exponent $p$ depending on the temperature, can be found in Ref. \refcite{C4}. There, in particular, the existence of a weak solution is obtained, taking $p$ constant, for $p \in (3d/(d+2),2]$. Also, a conditional uniqueness for close solutions is shown.

On the numerical side, some pioneering results on steady isothermal non-Newtonian incompressible Stokes fluids, were given in Refs. \refcite{Baranger-Najib:1990,Barrett-Liu:1993,Barrett-Liu:1994,Glowinski-Rappaz:2003,Sandri:1993}. More recently, in Refs. \refcite{Belenki-et-al:2012} and \refcite{Hirn:2013},  optimal error estimates for shear thinning non-Newtonian fluids have been addressed, whereas in Ref. \refcite{Sandri:2014}  a finite element method based on a four-field formulation of the nonlinear Stokes equations has been considered. 
Moreover, without aiming at completeness, we refer to the book Ref. \refcite{handbook:2011} and the references therein, and to Refs. \refcite{Diening-Kreuzer-Suli:2013,Ghattas:2015,Kreuzer-Suli:2016,Eckstein:2018,Ko-Suli:2019} as recent relevant contributions on the numerical discretization of generalized Stokes problems. We also point out that the approximation of steady isothermal non-Newtonian incompressible Stokes fluid problems on polytopal meshes have been addressed in Ref. \refcite{Botti_et_al:2021}.
Concerning non-isothermal Newtonian fluid problems, we refer to Ref. \refcite{Bernardi-et-al:2015}, where spectral elements have been employed, and to Ref. \refcite{Djoko-et-al:2020} which considers  the finite element approximation of the heat equation coupled with Stokes equations under threshold type boundary condition. It is worth mentioning that a very recent contribution is devoted to the numerical approximation of the steady state problem for an isothermal incompressible heat-conducting fluid with implicit non-Newtonian rheology. \cite{Farrell-Suli-et-al:2021} There, it is proven that the sequence of numerical approximations  converge to a weak solution to the original problem. Also, a block preconditioner is introduced to efficiently solve the resulting system of nonlinear equations. 

\medskip

The goal of this paper is twofold: theoretical and numerical. Regarding the former, we establish first the existence of a solution of \eqref{Sy01}-\eqref{Sy03} under rather general assumptions on the stress-strain relationship in the shear thinning case. Then we obtain  some regularity results. In particular, we extend the regularity results of Ref. \refcite{BerselliRou} for the 2D Stokes problem, and the ones of Ref. \refcite{Beirao2} for the 3D Navier-Stokes problem, to the non-isothermal case. Concerning the numerical approximation of \eqref{Sy01}-\eqref{Sy03}, we consider the Stokes approximation and we choose the Carreau law as constitutive stress-strain law.  This choice is motivated by relevant applications in the context of mixing and extrusion processes of polymeric materials. In this kind of problems, the inertia terms in the Navier-Stokes equations can often be neglected, due to the high effective viscosity leading to Reynolds numbers lower than
unity.\cite{Chiruvella:1996,Fard:2013,Guzman:2003} On account of the existence and regularity results obtained for the continuous problem, we perform an \textit{a priori} error analysis of a FEM approximation scheme of the problem and then we verify the convergence results by means of some numerical experiments.  

The outline of the paper is the following. In Section \ref{sec:notation} we introduce the main assumptions and the functional framework.   Section \ref{sec::mainTheoretical} contains the notion of weak solution together with the main analytical results of the paper. Then, in Section \ref{sec:approx}, we study an approximating problem whose results are exploited in the proof of some analytical results and in the numerical analysis part. Section \ref{sec::proofsTheoretical} is devoted to the proofs of the theoretical results. The numerical aspects can be found in Section \ref{numerical}, namely, the analysis of the discrete problem along with \textit{a priori} estimates for the error as well as some numerical experiments. {Finally, an appendix contains some technical tools which are used in the proof of the main theorems.}
\section{Notation and functional setting}\setcounter{equation}{0}
\label{sec:notation}
Our basic assumptions are the following:
\begin{description}
\item[(H1)] $\nu\in W^{1,\infty}(\mathbb{R})$
and there exists $\nu_1>0$ such that
\begin{align}
\nu_1\leq \nu(s),\qquad\forall s\in\mathbb{R};\nonumber
\end{align}

\item[(H2)] $\tau : \Omega \times \mathbb{R}_{sym}^{d \times d} \to \mathbb{R}_{sym}^{d \times d}$ is a Carath\'{e}odory function;

\item[(H3)] there exist $\tau_1,\tau_2>0$ such that, for all $\mathbf{B} \in \mathbb{R}_{sym}^{d \times d}$ and almost 
any $x\in \Omega$,
\begin{align}
& \tau(x, \mathbf{B})\cdot\mathbf{B}\geq \tau_1 (\vert \mathbf{B}\vert^p -1), \nonumber \\
& \vert \tau(x, \mathbf{B})\vert \leq \tau_2(1 +  \vert \mathbf{B}\vert)^{p-1}, \nonumber
\end{align}
where $p\in (1,2)$ (shear thinning case);

\item[(H4)] for all $\mathbf{B}_1, \mathbf{B}_2 \in \mathbb{R}_{sym}^{d \times d}$ such that $\mathbf{B}_1 \not=\mathbf{B}_2$ and almost any $x\in \Omega$, we have (strict monotonicity)
$$(\tau(x,\mathbf{B}_1) - \tau(x,\mathbf{B}_2)) \cdot (\mathbf{B}_1 - \mathbf{B}_2) > 0. $$

\end{description}
Here $\cdot$ and $\vert \cdot \vert$ denote the scalar product and the Euclidean norm in $\mathbb{R}_{sym}^{d \times d}$ (or in $\mathbb{R}^d$), respectively.

We then set $H:=L^2(\Omega)$, $V:=H^1(\Omega)$, $V_0: =H^1_0(\Omega)$, and
\begin{align}
&\mathcal{V}:= \left\{\mathbf{v} \in \textbf{C}^\infty_c(\Omega)\,:\, \mbox{div }\mathbf{u}=0 \; \hbox{ in } \Omega\right\}, \nonumber \\
& \textbf{L}^r_{div}:= \overline{\mathcal{V}}^{\textbf{L}^r(\Omega)-norm}, \nonumber \\
& \textbf{V}^r_{div}:= \overline{\mathcal{V}}^{\textbf{L}^r(\Omega)-gradient\, norm}, \nonumber
\end{align}
where $r\in [1,\infty)$.

We denote by $\|\cdot\|$ and $(\cdot,\cdot)$ the norm and the
scalar product on both $H$ or $\textbf{H}$. In particular, we set $\Vert v \Vert_{V_0} := \Vert \nabla v \Vert$.
If $X$ is a (real) Banach space
then $X'$ will denote its dual
and $\langle\cdot,\cdot\rangle$ will stand
for the duality pairing between $X^\prime$ and $X$. Moreover, {we consider the Stokes operator $\textbf{A}=-P\Delta$, where $P$ is the Leray orthogonal projector onto ${\textbf{L}}^2_{div}$}.
{ We also need to introduce the standard Bogovskii operator $B$ (see, e.g., Ref. \refcite{Bogo}). Denoting by $W^{1,q}_0(\Omega)$ the closure of $C^\infty_c (\Omega)$ in $W^{1,q}(\Omega)$), $1<q<\infty$, we know that $B: L^q_{(0)}(\Omega)\to W^{1,q}_0(\Omega)$ satisfies the equation $\text{div }Bf=f$, where $(0)$ stands for the zero integral mean, and we recall that, for any $q\in (1,\infty)$, it holds
	\begin{equation}
	\Vert Bf\Vert_{W^{1,q}(\Omega)}\leq C\Vert f\Vert_{L^q(\Omega)}.
	\label{p1} 
	\end{equation} 
	for some constant $C>0$ depending on $q$. Besides, if $f=\text{div }\textbf{g}$, where $\textbf{g}\in \textbf{L}^q(\Omega)$ is such that $\text{div }\textbf{g}\in L^q(\Omega)$ and $\textbf{g}\cdot \textbf{n}=0$ on $\partial\Omega$, then we have 
	 \begin{equation}
	 	\Vert Bf\Vert_{L^q(\Omega)}\leq C\Vert \textbf{g}\Vert_{\textbf{L}^q(\Omega)}.
	 \label{p2}
	 \end{equation}
  }

Concerning the data, the basic hypotheses are 

\begin{description}
\item[(H5)] $\mathbf{f} \in \textbf{W}^{-1,p^\prime}(\Omega)$;

\item[(H6)] $g \in V^\prime$;

\item[(H7)] $\theta \in H^{1/2}(\partial\Omega)$.
\end{description}
Here $p^\prime$ denotes the conjugate index of $p$.

Denoting by $\Theta_0\in V$ the Dirichlet lift of $\theta$, namely the (weak) solution to the Dirichlet problem
\begin{equation}
\begin{aligned}
-\kappa&\Delta \Theta_0 =0 \qquad \hbox{in }\Omega \label{eq:Theta0}\\
& \Theta_0 = \theta \qquad \hbox{on }\partial\Omega 
\end{aligned}
\end{equation}
and setting $\vartheta=\Theta - \Theta_0$, we can rewrite \eqref{Sy01}-\eqref{bc} as follows

\begin{align}
-&\mbox{div}(\nu(\vartheta + \Theta_0) \mathbf{\tau}(\mathbf{\varepsilon}(\mathbf{u})) +(\mathbf{u}\cdot\nabla)\mathbf{u}+\nabla\pi= \mathbf{f} \qquad \hbox{in }\Omega
\label{Sy04}\\
&\mbox{div }\mathbf{u}=0\qquad \hbox{in }\Omega \label{Sy05}\\
-&\kappa\Delta \vartheta + \mathbf{u} \cdot \nabla\vartheta = g - \mathbf{u} \cdot \nabla\Theta_0 \qquad \hbox{in }\Omega\label{Sy06}\\
&\mathbf{u} = \mathbf{0}, \quad \vartheta = 0 \qquad \hbox{on }\partial \Omega.
\label{Sy07}
\end{align}

We now introduce the definition of weak solution, namely,

\begin{defn}
\label{weaksol}
A pair $(\mathbf{u},\vartheta)\in \textbf{V}^p_{div} \times V_0$ is a weak solution to \eqref{Sy01}-\eqref{bc} if
\begin{align}
&\int_\Omega \nu(\vartheta + \Theta_0) \tau_{ij}(x,\varepsilon(\mathbf{u}))\varepsilon_{ij}(\mathbf{v})dx
-\int_\Omega u_j u_i \partial_{x_j} v_i dx = \langle \mathbf{f}, \mathbf{v} \rangle \qquad \forall\, \mathbf{v}\in \mathcal{V}, \label{fluid}\\
&\int_\Omega \kappa \nabla \vartheta \cdot \nabla \phi - \int_\Omega u_j\vartheta \partial_{x_j}\phi = \langle g,\phi \rangle
+ \int_\Omega u_j\Theta_0 \partial_{x_j}\phi \qquad \forall\,\phi\in V_0\cap {W}^{1,d}(\Omega). \label{temperature}
\end{align}
\end{defn}
\begin{oss}
Notice that the definition is well posed, since, when $d=3$, $\phi\in V_0\cap W^{1,3}(\Omega)$ ensures that the convective term in \eqref{temperature} is finite, by the embeddings $\textbf{V}^p_{div}\hookrightarrow \textbf{L}^q(\Omega)$, for some $q>2$, and $V_0\hookrightarrow L^6(\Omega)$. On the contrary, in the case $d=2$ it is enough to consider $\phi\in V_0$, by the embeddings $\textbf{V}^p_{div}\hookrightarrow \textbf{L}^q(\Omega)$, for some $q>2$, and $V_0\hookrightarrow L^r(\Omega)$ for any $r\in[2,\infty)$. { Formulation \eqref{temperature} does not allow to consider $\vartheta$ as a test function, in the case $d=3$. We will take care of this problem whenever necessary (see in particular the conditional uniqueness result of Theorem \ref{uniquebis} for the Stokes problem). Moreover, when $p>\frac{3d}{d+2}$, we can deduce by density that \eqref{fluid} also holds for any $\mathbf{v}\in \textbf{V}_{div}^p$. In conclusion, when $d=3$, if we assume $p>\frac{3}{2}$, then we can extend by density the validity of \eqref{temperature} for any $\phi\in V_0$, thus retrieving the standard definition of a weak solution in which the test space (i.e. $V_0$) is the same as the solution space.}  
\label{test}
\end{oss}
\section{Main theoretical results}
\label{sec::mainTheoretical}
We present here our main theoretical results. The existence of a weak solution is given by
\begin{thm}
\label{exist} Let assumptions (H1)-(H7) hold. If $p\in (2d/(d+2),2)$ then there exists a weak solution in the sense of Definition \ref{weaksol}.
\end{thm}

\begin{oss}
\label{pressure}
The pressure $\pi$ (up to a constant) can be recovered from \eqref{fluid} by means of a suitable version of de Rham's Theorem (see, e.g., Thm. 2.2.10 in Ref. \refcite{Br}).
More precisely, we can find a unique {$\pi\in L^{s}_{(0)}(\Omega)$, where $s=\min\{\frac{dp}{2(d-p)},p'\}$}, such that
\begin{equation*}
\int_\Omega \nu(\vartheta + \Theta_0) \tau_{ij}(x,\varepsilon(\mathbf{u}))\varepsilon_{ij}(\mathbf{v})dx
-\int_\Omega u_j u_i \partial_{x_j} v_i dx = \langle \mathbf{f}, \mathbf{v} \rangle +\int_\Omega \pi \, \mbox{\rm div }\mathbf{v} dx
\end{equation*}
for all $\mathbf{v}\in \textbf{C}^\infty_c(\Omega)$. {Notice that the fact that we need $\pi\in L^s_{(0)}(\Omega)$, with $s\leq p'$ is due to the convective term. Indeed, one obtains that $\mathbf{u}\otimes \mathbf{u}\in \textbf{L}^{\frac{dp}{2(d-p)}}(\Omega)$. Therefore, in the Stokes problem we get $\pi\in L^{p'}_{(0)}(\Omega)$.}
\end{oss}
\begin{oss}
Observe that Theorem \ref{exist} is still valid if we neglect the convective term $(\textbf{u}\cdot\nabla)\textbf{u}$, i.e., if we consider the Stokes problem. {In particular, in this case, the range of $p$ for which Theorem \ref{exist} holds for any $p\in(1,2)$. Indeed, it suffices to use the compact embedding $\textbf{V}^p_{div}\hookrightarrow \textbf{L}^{q}(\Omega)$, for any $q<\frac{dp}{d-p}$, to handle the convective term in the equation for the temperature (see \eqref{convu}-\eqref{conv1b} below for further comments on this issue). As a consequence, the weak formulation \eqref{fluid} holds, by density, for any $\mathbf{v}\in \textbf{V}^p_{div}$.} 
\label{existence1}
\end{oss}
We can then establish some regularity properties of a weak solution provided that the stress-strain relationship 
has the form
\begin{equation}
\tau(x,\textbf{B})=\tau(\textbf{B})=(1+\vert\textbf{B}\vert)^{p-2}\textbf{B}   
\label{Carreau}   
\end{equation}
and the data are more regular. Actually, we will also consider the so-called Carreau law  (see Section \ref{numerical})
\begin{equation}
\tau(x,\textbf{B})=\tau(\textbf{B})=(1+\vert\textbf{B}\vert^2)^{\frac{p-2}{2}}\textbf{B}.    
\label{Carreau2}   
\end{equation}
For an extensive analysis of power-law fluids, the reader is referred, for instance, to Example 1.73 in Ref. \refcite{MNRR}.  
We need to distinguish the original Navier-Stokes problem from the Stokes case, namely, when the convective term $(\textbf{u}\cdot\nabla)\textbf{u}$ is neglected. 

In the case $d=3$, we suppose
\begin{equation}
    \delta=\frac{6p-9}{3-p}>0,
\label{delta}
\end{equation}
and assume
\begin{align}
\textbf{f}\in \textbf{L}^{p^\prime}(\Omega),\qquad g\in {L}^{3+\delta}(\Omega),\qquad \theta\in {W}^{2-\frac{1}{3+\delta}, 3+\delta}(\partial\Omega)
\label{extra}
\end{align}
or, possibly, the stronger ones 
\begin{align}
\textbf{f}\in \textbf{L}^{p^\prime}(\Omega),\qquad g\in W^{1,3+\delta}(\Omega),\qquad \theta\in W^{3-\frac{1}{3+\delta}, 3+\delta}(\partial\Omega).
\label{extra2}
\end{align}
{Moreover, we fix the following quantities:
\begin{align}
\overline{q}:=4p-2,\quad l:=\frac{4p-2}{p+1}.
\label{qbar}
\end{align}
}
Then we can prove the following
\begin{thm}
	\label{existregol} Let $d=3$ and let assumptions (H1)-(H2), \eqref{Carreau}, \eqref{delta}, and \eqref{extra} hold. If $p\in (p_0,2)$, with $p_0=20/11$, then there exists a weak solution in the sense of Definition \ref{weaksol} which enjoys the following additional regularity:
	\begin{itemize}
		\item $\textbf{u}\in \textbf{W}^{1,\overline{q}}(\Omega)\cap \textbf{W}^{2,l}(\Omega),\qquad \nabla\pi\in \textbf{L}^l(\Omega),$ \\ 
		\item $\Theta=\vartheta+\Theta_0\in W^{2,3+\delta}(\Omega)$,
	\end{itemize}
where $\overline{q}$ and $l$ are defined in \eqref{qbar}.
For the Stokes problem the same regularity holds for any {$p\in(3/2,2)$}. 
Also, in both cases, if extra-assumptions \eqref{extra2} hold then we have
	\begin{itemize}
	\item $\Theta=\vartheta+\Theta_0\in W^{3,3+\delta}(\Omega)$. 
\end{itemize}
\end{thm}
\begin{oss}
{
Concerning the minimal exponent $p_0=20/11$, this result coincides with the one proposed in Ref. \refcite{Beirao2}, Theorem 2.3. Indeed, here we adopt the same technique which consists in finding a result for the Stokes problem first (exploiting Ref. \refcite{Beirao2}, Theorem 2.2), then extending it to the Navier-Stokes problem by looking at the additional convective term as a perturbation of the Stokes problem. As far as we know this value of $p_0$ is the best result in literature concerning the regularity of stationary Navier-Stokes problem in the shear thinning setting. Indeed, other results (for flat domains) can be found in Ref. \refcite{Beirao3}, where the lower bound is $p_1=15/8>p_0$ and in Ref. \refcite{Berselli}, where the minimal exponent is improved, reaching $p_1=\frac{7+\sqrt{35}}{7}$, which is still greater than $p_0$. Moreover, we point out that Ref. \refcite{BerselliRou}, Theorem 2.29, allows to improve the lower bound $p=3/2$ for the Stokes problem given in 
Ref. \refcite{Beirao2}, Theorem 2.2. Indeed, it is shown that the same results hold for the Stokes problem for any $p\in(1,2)$. Nevertheless we still need $p\in(3/2,2)$ due to the coupling with the temperature.} \label{validity} 
\end{oss}
We also consider the two-dimensional case, since the numerical experiments will be carried on in this setting. We assume 
\begin{equation}
\delta=\frac{4(p-1)}{2-p}>0
\label{a1}
\end{equation}
and  
\begin{align}
\textbf{f}\in \textbf{L}^{p^\prime}(\Omega),\qquad g\in L^{2+\delta}(\Omega),\qquad \theta\in W^{2-\frac{1}{2+\delta}, 2+\delta}(\partial\Omega)
\label{extraB1}
\end{align}
or, possibly, the stronger ones 
{
\begin{align}
\textbf{f}\in \textbf{L}^{p^\prime}(\Omega),\quad g\in W^{1,\overline{q}}(\Omega),\quad \theta\in W^{2-\frac{1}{2+\delta}, 2+\delta}(\partial\Omega)\cap W^{3-\frac{1}{\overline{q}},\overline{q}}(\partial\Omega),
\label{extraB2}
\end{align}
where $\overline{q}$ is defined in \eqref{qbar}.

Then the following result holds
\begin{thm}
	\label{existregol2} Let $d=2$ and let assumptions (H1)-(H2), \eqref{Carreau} (or \eqref{Carreau2}), \eqref{a1}, and \eqref{extraB1} hold. If ${p\in (1,2)}$, then there exists a weak solution to the Stokes problem, in the sense of Definition \ref{weaksol}, which enjoys the following additional regularity:
{	\begin{itemize}
		\item $\textbf{u}\in \textbf{W}^{1,\overline{q}}(\Omega)\cap\textbf{W}^{2,l}(\Omega),\qquad \pi\in W^{1,l}(\Omega),$
	\end{itemize}
	where $\overline{q}$ and $l$ are defined in \eqref{qbar}. 

	We also obtain 
	\begin{itemize}
	 \item $\Theta=\vartheta+\Theta_0\in W^{2,2+\delta}(\Omega)$. 
	\end{itemize}
Moreover, if assumptions \eqref{extraB2} hold then we have
	\begin{itemize}
	\item {$\Theta=\vartheta+\Theta_0\in W^{3,\overline{q}}(\Omega)$. }
\end{itemize}
}
\end{thm}
\begin{oss}
{Differently from the case $d=3$, for the Stokes problem with $d=2$ we can assume $p\in(1,2)$, thanks to the results in Ref. \refcite{BerselliRou} and the regularity of $\vartheta$.}
\label{lowerbound}
\end{oss}
In conclusion, concerning the Stokes problem, i.e., neglecting the term $(\textbf{u}\cdot\nabla)\textbf{u}$, we can obtain a conditional uniqueness result for the weak solutions according to Definition \ref{weaksol}.
\begin{thm}
\label{uniquebis} Let assumptions (H1)-(H7) hold. For $d=2,3$, if we consider $\tau$ as in \eqref{Carreau2}, assuming that, for $d=2$, $\textbf{u}\in \textbf{W}^{1,q}(\Omega)$, with $q>p$ and $p\in(1,2)$, whereas, for $d=3$, $\textbf{u}\in \textbf{W}^{1,\frac{6p(p-1)}{5p-6}}(\Omega)$ and $p\in[3/2,2)$, if
\begin{align}
\frac{M_3}{2}>\frac{M_4}{\kappa^3} \left(\Vert g\Vert_{V^\prime} + {\Vert \Theta_0\Vert_{V}}\right)^2\left(1+\Vert\textbf{u}\Vert^{2p-2}_{\textbf{W}^{1,\alpha(p-1)}(\Omega)}\right),
    \label{cond}
\end{align}
with $\alpha=\frac{q}{p-1}$ for $d=2$, and $\alpha=\frac{6p}{5p-6}$ for $d=3$, then the weak solution to the Stokes problem according to Definition \ref{weaksol} is unique. Here $M_3$ and $M_4$ are positive constants, depending on $\Omega$, $\tau_1$, $\tau_2$, $\nu_1$ and $\Vert\mathbf{f}\Vert_{\textbf{W}^{-1,p^\prime}(\Omega)}$.
\end{thm}
\begin{oss}
{Compared with the result of Ref. \refcite{uniq}, Corollary 2.2, here we need more a priori regularity for $\mathbf{u}$. As it will be clear in the proof, this is due to the dependence of the viscosity $\nu$ on $\vartheta$ (see \eqref{term1bis} below).}
\end{oss}
In Sections \ref{pp1}-\ref{pp2} we give the proofs of the above results. In particular, Theorem \ref{exist} is proven through an approximating problem which is analyzed in the next section. This problem will be useful in Section \ref{numerical} as well. Then, the proof of Theorem \ref{uniquebis} will be given in Section \ref{uniqueness}, together with the proof of some conditional uniqueness results concerning the approximating problem.
\section{The approximating problem}
\label{sec:approx}
Fix $\sigma>0$, $r\in [2, \infty)$ and $p\in(1,2)$. Then consider the problem of finding $(\mathbf{u},\pi,\vartheta)$ such that
\begin{align}
-&\mbox{div}\left[\nu(\vartheta + \Theta_0) \left(\mathbf{\tau}(x,\mathbf{\varepsilon}(\mathbf{u}))
+ \sigma \vert \varepsilon(\mathbf{u}) \vert^{r-2} \varepsilon(\mathbf{u})\right)\right]
+(\mathbf{u}\cdot\nabla)\mathbf{u}+\nabla\pi= \mathbf{f} \qquad \hbox{in }\Omega
\label{Sy08}\\
&\mbox{div }\mathbf{u}=0\qquad \hbox{in }\Omega \label{Sy09}\\
-&\kappa\Delta \vartheta + \mathbf{u} \cdot \nabla\vartheta = g - \mathbf{u} \cdot \nabla\Theta_0 \qquad \hbox{in }\Omega\label{Sy10}\\
&\mathbf{u} = \mathbf{0}, \quad \vartheta = 0 \qquad \hbox{on }\partial \Omega.
\label{Sy11}
\end{align}

In this case the definition of weak solution reads

\begin{defn}
\label{weaksolreg}
A pair $(\mathbf{u},\vartheta)\in \textbf{V}^r_{div} \times V_0$ is a weak solution if
\begin{align}
&\int_\Omega \left[\nu(\vartheta + \Theta_0)\left( \tau_{ij}(x,\varepsilon(\mathbf{u})) +\sigma \vert \varepsilon(\mathbf{u}))\vert^{r-2} \varepsilon_{ij}(\mathbf{u})\right)\right]\varepsilon_{ij}(\mathbf{v})dx \nonumber\\
&-\int_\Omega u_j u_i \partial_{x_j} v_i dx = \langle \mathbf{f}, \mathbf{v} \rangle \qquad \forall\, \mathbf{v}\in \textbf{V}^r_{div} \label{fluidapp}\\
&\int_\Omega \kappa \nabla \vartheta \cdot \nabla \phi + \int_\Omega u_j\vartheta \partial_{x_j}\phi = \langle g,\phi \rangle
+ \int_\Omega u_j\Theta_0 \partial_{x_j}\phi \qquad \forall\,\phi\in V_0. \label{tempapp}
\end{align}
\end{defn}
We now state some existence and conditional uniqueness results, whose proofs are postponed to Section \ref{appro}. 
We establish first the existence of a weak solution to the approximating problem.
\begin{thm}
\label{existreg} Let assumptions (H1)-(H7) hold. Then there exists a weak solution $(\mathbf{u}_\sigma,\vartheta_\sigma)\in \textbf{V}^r_{div} \times V_0 $ in the sense of Definition \ref{weaksolreg}, for any $p\in(1,2)$ and $r\geq 2$.
\end{thm}

\begin{oss}
A unique pressure $\pi\in L^{r^\prime}_{(0)}(\Omega)$ can be recovered (see Remark \ref{pressure}).
\end{oss}

\begin{oss}
{We can also consider the case of non-zero divergence, that is,
\begin{align}
-&\mbox{{\rm div}}\left[(\nu(\vartheta + \Theta_0) \mathbf{\tau}(x,\mathbf{\varepsilon}(\mathbf{u})) + \sigma \vert \varepsilon(\mathbf{u}) \vert^{r-2} \varepsilon(\mathbf{u})\right]
+\nabla\pi= \mathbf{f} \qquad \hbox{in }\Omega
\label{div1}\\
&\mbox{{\rm div} }\mathbf{u}=w\qquad \hbox{in }\Omega \label{div2}\\
-&\kappa\Delta \vartheta + \mathbf{u} \cdot \nabla\vartheta = g - \mathbf{u} \cdot \nabla\Theta_0 \qquad \hbox{in }\Omega\label{div3}\\
&\mathbf{u} = \mathbf{0}, \quad \vartheta = 0 \qquad \hbox{on }\partial \Omega.
\label{div4}
\end{align}
Indeed, we can ``subtract'' $B w$ (see, e.g., Ref. \refcite{So}, Lemma 2.1.1 ).
}


We expect to be able to prove Theorem \ref{existreg} also in this case.
\end{oss}
We can also prove a first conditional uniqueness result, { in the most interesting case, i.e., when $p\in(1,2)$ and $d=3$}. This is given by
\begin{thm}
\label{unique} Let assumptions (H1)-(H7) hold. Suppose $d=3$ and {$p\in (1,2)$}. 
Suppose, in addition,
that $\mathbf{u}\in \textbf{W}^{1,q}(\Omega)$ where $q\geq d$. Set $N=\Vert \nabla \mathbf{u}\Vert_{\textbf{L}^q(\Omega)}$.
There exist two positive constants $M_1$ and $M_2$, depending on $\Omega$, $\tau_1$, $\tau_2$, $\nu_1$ and $\Vert\mathbf{f}\Vert_{\textbf{W}^{-1,p^\prime}(\Omega)}$,
such that if
\begin{equation}
\sigma\nu_1 > M_1\frac{\Vert \nu^\prime\Vert_{L^\infty(\mathbb{R})}}{\kappa^{3/2}}\left(\Vert g\Vert_{V^\prime} + \frac{\Vert \Theta_0\Vert_{V}}{\sqrt{\sigma}}\right)\left(1+{N^{p-1}}+\sigma N\right)
+\frac{M_2}{\sqrt{\sigma}}
\label{uniqcond}
\end{equation}
then the weak solution is unique.
\end{thm}
\begin{oss}
{Observe that, as expected, in the case with $\sigma>0$ and $r\geq2$ the problem resembles the one treated in Ref. \refcite{Bernardi-et-al:2015}, Proposition 2.3. Indeed, in both cases, there is some smallness assumption on the $\textbf{W}^{1,3}(\Omega)$ norm of $\mathbf{u}$.}
\end{oss}
If we assume some additional hypotheses, then we can prove a more refined conditional uniqueness result. In particular, let us suppose 
\begin{equation}
\theta\in H^{3/2}(\partial\Omega),\qquad g\in H.
\label{g}
\end{equation}
{We analyze the problem in Definition \ref{weaksolreg}} in the case with $r=2${, i.e., when the regularization accounts only for the Laplace operator}, and $d=3$. We have
\begin{thm}
	\label{unique2} Let assumptions (H1)-(H7), together with \eqref{g}, hold. Suppose $p\in (1,2)$.
	Then there exist four positive constants $M_1$, $M_2$, $M_3$ and $M_4$, depending on $\Omega$, $\tau_1$, $\tau_2$, $\nu_1$ and $\Vert\mathbf{f}\Vert_{\textbf{W}^{-1,p^\prime}(\Omega)}$,
	such that if
		\begin{align}
&\frac{\sigma\nu_1}{2}<\frac{M_1}{\kappa^2}\Vert \nu^\prime\Vert_{L^\infty(\mathbb{R})}^2\left( 1+\frac{1}{(\sqrt{\sigma})^{p-1}}+ \sqrt{\sigma}\right)^2\nonumber\\&\nonumber\times \left(\frac{1}{\kappa^3\nu_1\sigma}\left(\Vert g\Vert_{V^\prime} + \frac{M_2\Vert \theta\Vert_{H^{1/2}(\partial\Omega)}}{\sqrt{\sigma}}\right)+\frac{1}{\kappa\sqrt{\sigma}}\Vert \theta\Vert_{H^{3/2}(\partial\Omega)}+\frac{1}{\kappa}\Vert g\Vert\right)^2\\&\nonumber
	+ M_3\left(1+\frac{1}{\sqrt{\sigma}}+\sigma \left(\Vert\nu\Vert_{L^\infty(\R)}\right.\right.\\&\left.\left.+\Vert \nu^\prime\Vert_{L^\infty(\mathbb{R})}\left(\left(\frac{1}{\kappa^3\sigma}\left(\Vert g\Vert_{V^\prime} + \frac{M_2\Vert \theta\Vert_{H^{1/2}(\partial\Omega)}}{\sqrt{\sigma}}\right)+\frac{1}{\kappa\sqrt{\sigma}}\Vert \theta\Vert_{H^{3/2}(\partial\Omega)}\right.\right.\right.\right.\nonumber\\&\left.\left.\left.\left.+\frac{1}{\kappa}\Vert g\Vert\right)+\Vert \theta\Vert_{H^{3/2}(\partial\Omega)}\right)\right)\right),
	\label{uniqcond2}
	\end{align}
	then the solution { to \eqref{fluidapp}-\eqref{tempapp}} is unique.
\end{thm}
\begin{oss}
	Notice that condition \eqref{uniqcond2} is better than  \eqref{uniqcond}, since it does not depend on the solution $\textbf{u}$, but {it depends} only on the data and on $\Omega$, { even though in a nontrivial way}.
\end{oss}

\section{Proofs of the main theoretical results}
\label{sec::proofsTheoretical}
\subsection{Proof of Theorem \ref{exist}}
\label{pp1}
Assume for the moment that the Dirichlet lift $\Theta_0$ belongs to $H^2(\Omega)$ (i.e. $\theta \in H^{3/2}(\partial\Omega)$). Then choose $\sigma=\frac{1}{n}$, $n\in \mathbb{N}_0$, in \eqref{fluidapp} and denote by $(\mathbf{u}_n, \vartheta_n)$ a solution to \eqref{fluidapp}-\eqref{tempapp}.
Take $\mathbf{v}=\mathbf{u}_n$ in \eqref{fluidapp} and $\phi=\vartheta_n$ in \eqref{tempapp}. This yields, on account of $\mbox{div }\mathbf{u}_n=0$, the identities
\begin{align}
&\int_\Omega \nu(\vartheta_n + \Theta_0) \left(\tau_{ij}(x,\varepsilon(\mathbf{u}_n))\varepsilon_{ij}(\mathbf{u}_n)
+ \frac{1}{n} \vert \varepsilon(\mathbf{u}_n))\vert^{r}\right)dx = \langle \mathbf{f}, \mathbf{u}_n \rangle
\label{identity1}\\
&\int_\Omega \kappa \vert \nabla \vartheta_n \vert^2 = \langle g,\vartheta_n \rangle + \int_\Omega(u_n)_j\Theta_0 \partial_{x_j}\vartheta_n.
\label{identity2}
\end{align}
From identity \eqref{identity1}, recalling (H1)-(H3) and using Young's and Korn's inequalities, we deduce
\begin{equation}
\Vert \mathbf{u}_n \Vert^p_{\textbf{V}^p_{div}} + \frac{C}{n}\Vert \mathbf{u}_n \Vert^r_{\textbf{V}^r_{div}} \leq C \left(\Vert \mathbf{f} \Vert^{p^\prime}_{(\textbf{V}^p_{div})^\prime} + 1\right).
\label{bound4}
\end{equation}
Here $C>0$ is a positive constant depending at most on $\Omega$, $d$, $\nu_1$, and $\tau_1$. {By (H3), we also have
\begin{align}
\Vert \tau(\cdot,\varepsilon(\mathbf{u}_n))\Vert_{\textbf{L}^{p'}(\Omega)}\leq C.
\label{ta}
\end{align}
}
From identity \eqref{identity2} we infer
\begin{equation}
\Vert \vartheta_n\Vert_{V_0}^2 \leq C\left(\Vert g \Vert^2_{V^\prime} + \Vert \mathbf{u}_n \Vert^2_{\textbf{H}_{div}}\Vert \Theta_0\Vert^2_{H^2(\Omega)}\right).
\label{bound5}
\end{equation}
Here we have used Poincar\'{e}'s and Young's inequalities as well as the embedding $H^2(\Omega) \hookrightarrow C^0(\overline{\Omega})$. The constant $C>0$ depends at most on $\kappa$, $\Omega$, and $d$.

Combining \eqref{bound4} and \eqref{bound5}, we can find a (not relabeled) subsequence $(\mathbf{u}_{n},\vartheta_{n})$ and a pair $(\mathbf{u},\vartheta)\in \textbf{V}^p_{div} \times V_0$ such
that
\begin{align}
&\mathbf{u}_{n} \rightharpoonup \mathbf{u}\;\hbox{in } \textbf{V}^p_{div}, \quad \mathbf{u}_{n} \rightarrow \mathbf{u} \; \hbox{in } \textbf{L}^2_{div}, \label{convu}\\
&\vartheta_{n} \rightharpoonup \vartheta \;\hbox{in } V_0, \quad \vartheta_{n} \rightarrow \vartheta \;\hbox{in } L^{q}(\Omega)\label{convtheta}
\end{align}
for any given $q< \frac{2d}{d-2}$. Note that this is possible due to the compact embedding $\textbf{V}^p_{div}\hookrightarrow \textbf{L}^2(\Omega)$ given by the assumption $p\in(2d/(d+2),2)$. {If we simply assume $p\in(1,2)$ together with the Stokes problem, i.e., neglecting the convective term $(\mathbf{u}\cdot\nabla)\mathbf{u}$, by the compactness of the embedding $\textbf{V}_{div}^p\hookrightarrow \textbf{L}^q(\Omega)$, for any $q<\frac {dp}{d-p}$, we get
\begin{align}
\mathbf{u}_{n}\to \mathbf{u}\quad \text{in }\textbf{L}^q(\Omega),\quad \forall q<\frac {dp}{d-p}.
\label{conv1b}
\end{align}}
{In both the cases} we also have (cf. \eqref{bound4})
\begin{equation}
\frac{1}{n}\Vert \mathbf{u}_n\Vert_{\textbf{V}^r_{div}}^{r} \rightarrow 0. \label{singpert}
\end{equation}
First of all, observe that, on account of \eqref{convu}-\eqref{convtheta}, we can pass to the limit as $n\to\infty$ in
\begin{equation}
\int_\Omega \kappa \nabla \vartheta_{n} \cdot \nabla \phi + \int_\Omega (u_{n})_j\vartheta_{n} \partial_{x_j}\phi = \langle g,\phi \rangle
+ \int_\Omega (u_{n})_j\Theta_0 \partial_{x_j}\phi \qquad \forall\,\phi\in C_c^\infty(\Omega). \label{tempapp2}
\end{equation}
This gives
\begin{equation}
\int_\Omega \kappa \nabla \vartheta\cdot \nabla \phi + \int_\Omega u_j\vartheta \partial_{x_j}\phi = \langle g,\phi \rangle
+ \int_\Omega u_j\Theta_0 \partial_{x_j}\phi \qquad \forall\,\phi\in C_c^\infty(\Omega) \label{temp2}
\end{equation}
and it is easy to realize, by a density argument, that this variational identity also holds if $\Theta_0 \in V$ and the test functions are taken in $V_0 \cap W^{1,d}(\Omega)$, thanks to the embedding $W^{1,d}(\Omega)\hookrightarrow L^q(\Omega)$ for any $q\in(2,\infty)$ (see also Remark \ref{test}). { Notice that actually the convergence \eqref{conv1b}, together with \eqref{convtheta}, is enough to guarantee the passage to the limit in \eqref{tempapp2}, recalling the compact embedding $V_0\hookrightarrow L^s(\Omega)$ for any $2\leq s<\frac{2d}{d-2}$ and the embedding $W^{1,d}(\Omega)\hookrightarrow L^q(\Omega)$ for all $q\in(2,\infty)$.}

{We now deal with the more difficult problem of passing to the limit in the Navier-Stokes (or Stokes) equation. First we notice that in the case when $p\in(\frac{3d}{d+2},2)$ for Navier-Stokes and $p\in(1,2)$ for Stokes problem, we could more easily obtain the result by exploiting the technique first devised in Ref. \refcite{Li}. The crucial point of this argument is the possibility of using  $\mathbf{w}_n:=\mathbf{u}_n-\mathbf{u}$ as a test function in the weak formulation. This is always possible in the case of Stokes problem, while, in the Navier-Stokes case, on account of the convective term, this is allowed only when $p>\frac{3d}{d+2}$. Thus, in order to handle lower values of $p$, we follow the more general approach introduced in Ref. \refcite{Lip}, which is valid for any $p\in(1,2)$ for the Stokes problem and for any $p\in(\frac{2d}{d+2},2)$ for the Navier-Stokes one.
Let us consider the identity
\begin{align}
&\int_\Omega \nu(\vartheta_{n} + \Theta_0) \tau_{ij}(x,\varepsilon(\mathbf{u}_{n}))\varepsilon_{ij}(\mathbf{v})dx
+ \frac{1}{n} \int_\Omega \nu(\theta_n+\Theta_0)\vert \varepsilon(\mathbf{u}_{n})\vert^{r-2} \varepsilon_{ij}(\mathbf{u}_{n})\varepsilon_{ij}(\mathbf{v})dx \nonumber\\
&-\int_\Omega (u_{n})_j (u_{n})_i \partial_{x_j} v_i dx = \langle \mathbf{f}, \mathbf{v} \rangle \qquad \forall\, \mathbf{v}\in \textbf{V}^r_{div}, \label{fluidapp2}
\end{align}
where the convective term is absent in the case of Stokes problem. The sequence $\{\mathbf{w}_n\}$ satisfies the assumptions of Theorem \ref{app1}, thanks to \eqref{convu}, so that we can consider a Lipschitz sequence $\{\mathbf{w}_{n,j}\}$ satisfying the same properties. Notice that this sequence is still not divergence free. Therefore, recalling the Bogovskii operator defined in Section \ref{sec:notation}, we set
$$
\pmb\psi_{n,j}:=B(\text{div }\mathbf{w}_{n,j})=B(\text{div}(\chi_{\{\mathbf{w}_{n,j}\not= \mathbf{w}_n\}}\mathbf{w}_{n,j})),
$$
where we exploited the fact that $\text{div }\mathbf{w}_n=0$. Combining \eqref{p1} with \eqref{cc}, we get
\begin{align}
\Vert \pmb\psi_{n,j}\Vert_{{\textbf{W}}^{1,p}(\Omega)}\leq C\Vert \text{div}(\chi_{\{\mathbf{w}_{n,j}\not= \mathbf{w}_n\}}\mathbf{w}_{n,j})\Vert_{{L}^p(\Omega)}\leq C\frac{\gamma_n}{\theta_n}\mu_{j+1}+C\epsilon_j,
\label{bnd}
\end{align}
which implies 
\begin{align}
\limsup_{n\to \infty}\Vert \pmb\psi_{n,j}\Vert_{{\textbf{W}}^{1,p}(\Omega)}\leq C\epsilon_j,
    \label{limsup}
\end{align}
for any $j\in\mathbb{N}$. Moreover, recalling that $B: L^q_{(0)}(\Omega) \to {W^{1,q}_0(\Omega)}$ is continuous for any $q\in(1,\infty)$, by \eqref{convu} and Theorem \ref{app1} we have 
$$ 
\text{div }\mathbf{w}_{n,j}\rightharpoonup 0\quad \text{ in }L^{q}(\Omega)\quad \forall q\in(1,\infty),
$$
as $n\to\infty$, so that
\begin{align}
\pmb\psi_{n,j}\rightharpoonup 0 \quad \text{ in }\textbf{W}^{1,q}_0(\Omega)\quad \forall q\in(1,\infty).
\label{wkconv}
\end{align}
Thus, by the compact embedding $\textbf{W}^{1,q}(\Omega)\hookrightarrow \textbf{L}^q(\Omega)$, for any $q>1$, we immediately infer 
\begin{align}
\pmb\psi_{n,j}\to 0 \quad \text{ in }\textbf{L}^{q}(\Omega)\quad \forall q\in(1,\infty).
    \label{strconv}
\end{align}
Let us now set
\begin{align*}
\pmb\phi_{n,j}:=\mathbf{w}_{n,j}-\pmb\psi_{n,j}\in \textbf{V}^q_{div},
\end{align*}
for any $q\in(1,\infty)$. From the above results and from Theorem \ref{app1}, it follows  that, for any $j\in\mathbb{N}$, 
\begin{align}
&\pmb\phi_{n,j}\rightharpoonup 0 \quad \text{ in }\textbf{W}^{1,q}_0(\Omega),
\label{wkconv1}
\\&\pmb\phi_{n,j}\to 0 \quad \text{ in }\textbf{L}^{q}(\Omega),
    \label{strconv2}
\end{align}
for all $q\in (1,\infty)$, as $n\to\infty$.
Hence we can consider $\pmb\phi_{n,j}\in \textbf{V}_{div}^r$ as a test function in \eqref{fluidapp2}, rewriting the result as follows:
\begin{align}
&\int_\Omega \left(\nu(\vartheta_{n} + \Theta_0) \tau(x,\varepsilon(\mathbf{u}_{n}))-\nu(\vartheta + \Theta_0) \tau(x,\varepsilon(\mathbf{u}))\right):\varepsilon(\mathbf{w}_{n,j})dx\nonumber
\\&=- \frac{1}{n} \int_\Omega \nu(\theta_n+\Theta_0)\vert \varepsilon(\mathbf{u}_{n})\vert^{r-2} \varepsilon_{ij}(\mathbf{u}_{n})\varepsilon_{ij}(\pmb\phi_{n,j})dx+\langle\textbf{f},\pmb\phi_{n,j}\rangle \nonumber\\
&\nonumber-\int_\Omega (u_{n})_j (u_{n})_i \partial_{x_j} [\pmb\phi_{n,j}]_i dx- \int_\Omega \nu(\vartheta + \Theta_0) \tau(x,\varepsilon(\mathbf{u})):\varepsilon(\mathbf{w}_{n,j})dx\nonumber\\&
+\int_\Omega \nu(\vartheta_{n} + \Theta_0) \tau(x,\varepsilon(\mathbf{u}_{n})):\varepsilon(\pmb\psi_{n,j})dx:=\sum_{k=1}^5 \mathcal{I}_k.
 \label{fluidapp2bb}
\end{align}
Observe that 
\begin{align*}
    \mathcal{I}_1\leq \frac C n \Vert \mathbf{u}_n\Vert_{\textbf{V}_{div}^r}^{r-1}\Vert \pmb\phi_{n,j}\Vert_{\textbf{W}^{1,r}(\Omega)}\to 0,
\end{align*}
as $n\to \infty$, recalling \eqref{singpert} and \eqref{wkconv1}. Furthermore, by \eqref{wkconv}, recalling that $\mathbf{f}\in \textbf{W}^{-1,p'}(\Omega)$, we have
$$
\mathcal{I}_2 \to 0\quad \text{as }n\to \infty.
$$
Then we observe that, being $p>\frac{2d}{d+2}$, it holds $\mathbf{u}_n\to \mathbf{u}$ in $\textbf{L}^2(\Omega)$, so that  
\begin{align*}
&\Vert \mathbf{u}_n\otimes \mathbf{u}_n-\mathbf{u}\otimes \textbf{u}\Vert_{\textbf{L}^q(\Omega)}\\&\leq C(\Vert \mathbf{u}_n\Vert_{\textbf{L}^{\frac{dp}{d-p}}(\Omega)}+\Vert \mathbf{u}\Vert_{\textbf{L}^{\frac{dp}{d-p}}(\Omega)})\Vert \mathbf{u}_n-\mathbf{u}\Vert \to 0, \quad  q= \frac{2dp}{dp+2(d-p)}>1,
\end{align*}
as $n\to \infty$,  recalling \eqref{bound4} and the embedding $\textbf{V}_{div}^p\hookrightarrow \textbf{L}^{\frac{dp}{d-p}}(\Omega)$. This implies that 
$$
\mathbf{u}_n\otimes \mathbf{u}_n\to \mathbf{u}\otimes \mathbf{u}\quad \text{ in } \textbf{L}^{\frac{2dp}{dp+2(d-p)}}(\Omega)
$$
and this result, together with \eqref{wkconv1}, ensures that 
$$
\mathcal{I}_3\to 0\quad \text{as }n\to \infty.
$$
Concerning $\mathcal{I}_4$, we observe that 
\begin{align*}
\nu(\theta+\Theta_0)\tau(\cdot, \varepsilon(\mathbf{u}))\in \textbf{L}^{p'}(\Omega)\hookrightarrow \textbf{W}^{-1,p^\prime}(\Omega),
\end{align*}
so that \eqref{wkconv} and \eqref{wkconv1} imply
$$
\mathcal{I}_4\to 0\quad \text{as }n\to \infty.
$$
In conclusion, we have
\begin{align*}
\limsup_{n\to\infty}\mathcal{I}_5\leq  C\limsup_{n\to\infty}\Vert \tau(\varepsilon(\mathbf{u}_n))\Vert_{\textbf{L}^{p'}(\Omega)}\Vert \pmb\psi_{n,j}\Vert_{\textbf{W}^{1,p}(\Omega)}\leq C\epsilon_j,\quad \forall j\in \mathbb{N},
\end{align*}
by \eqref{ta} and \eqref{limsup}. Therefore, we obtain from \eqref{fluidapp2bb} that
\begin{align}
\limsup_{n\to\infty}\int_\Omega \left(\nu(\vartheta_{n} + \Theta_0) \tau(x,\varepsilon(\mathbf{u}_{n}))-\nu(\vartheta + \Theta_0) \tau(x,\varepsilon(\mathbf{u}))\right):\varepsilon(\mathbf{w}_{n,j})dx\leq C\epsilon_j,
\label{sp}
\end{align}
for any $j\in\mathbb{N}$.
We point out that, in the case of the Stokes problem, the term $\mathcal{I}_3$ is absent, and thus it is enough to consider $p\in(1,2)$.

Estimate \eqref{sp} entails, by Lemma \ref{app2} with $\phi_n:=\vartheta_n+\Theta_0$,  $\phi:=\vartheta+\Theta_0$ (cf. \eqref{convtheta}) and $\delta_j:=C\epsilon_j$, that, for any $\zeta\in(0,1)$, 
\begin{align}
\limsup_{n\to\infty}\int_\Omega \left[\nu(\vartheta+\Theta_0)(\tau(x,\varepsilon(\mathbf{u}_n))-\tau(x,\varepsilon(\mathbf{u}))):(\varepsilon(\mathbf{u}_n)-\varepsilon(\mathbf{u}))\right]^\zeta dx =0.
    \label{fond1}
\end{align}
Hence, the limit as $n\to\infty$ is also zero, for the integrands are nonnegative. 
As a consequence, up to a subsequence, we have, for almost any $x\in \Omega$,
$$
\nu(\vartheta(x)+\Theta_0(x))\left[\tau(x,\varepsilon(\mathbf{u}_n(x)))-\tau(x,\varepsilon(\mathbf{u}(x)))\right]:(\varepsilon(\mathbf{u}_n(x))-\varepsilon(\mathbf{u}(x)))\to 0,
$$
as $n\to\infty$.
We can now apply Ref. \refcite{Murat}, Lemma 6, with $\beta_k(\cdot)=\beta(\cdot)=\nu(\vartheta(x)+\Theta_0(x))\tau(x,\cdot)$, $\xi_k=\varepsilon(\mathbf{u}_n(x))$, $\xi=\varepsilon(\mathbf{u}(x))$. Then, recalling that $\nu(\cdot)\geq \nu_1>0$, we find
\begin{align}
\varepsilon(\mathbf{u}_n)\to \varepsilon(\mathbf{u}),
    \label{ep}
\end{align}
almost everywhere in $\Omega$. This yields, from assumption $(\textbf{H}2)$,
$$
\tau(\cdot,\varepsilon(\mathbf{u}_n(\cdot)))\to \tau(\cdot,\varepsilon(\mathbf{u}(\cdot))),
$$
almost everywhere in $\Omega$. Thus, on account of \eqref{ta}, by generalized Lebesgue's Theorem, we deduce that
$$
\tau(\cdot,\varepsilon(\mathbf{u}_n(\cdot)))\rightharpoonup \tau(\cdot,\varepsilon(\mathbf{u}(\cdot)))\quad\text{ in }\textbf{L}^{p'}(\Omega).
$$
Exploiting this result and the convergences \eqref{convu} and \eqref{convtheta}, we can thus easily pass to the limit as $n\to \infty$ in \eqref{fluidapp2}, with $\mathbf{v}\in \mathcal{V}$, obtaining in the end
\begin{equation}
\int_\Omega \nu(\vartheta + \Theta_0) \chi_{ij}\varepsilon_{ij}(\mathbf{v})dx -\int_\Omega u_j u_i \partial_{x_j} v_i dx
= \langle \mathbf{f}, \mathbf{v} \rangle \qquad \forall\, \mathbf{v}\in \mathcal{V}\label{fluid2}.
\end{equation}
This concludes the proof.
}
}
\subsection{Proof of Theorems \ref{existregol} and \ref{existregol2}}
\label{pp2}
\subsubsection{Preliminary results}
We first introduce some technical tools necessary to carry out the proof. 
\paragraph{Regularity for the Stokes system.}
We consider the $d-$dimensional system (in which we consider the tensor $\tau$ expressed in \eqref{Carreau} or \eqref{Carreau2}):
\begin{align}
\begin{cases}
-\text{div}(\nu(\phi)\tau(D\textbf{u}))+\nabla \pi=\textbf{f},\\
\text{div }\textbf{u}=0,
\end{cases}
\label{Stokes}
\end{align}
with no-slip boundary conditions, for $1<p<2$, where $\phi$ is a generic function and $\textbf{f}\in \textbf{L}^{p^\prime}(\Omega)$. For the case with constant viscosity $\nu$ and $\tau$ as in \eqref{Carreau}, when $p\in\left({3}/{2},2\right)$ and $d=3$, there exists a unique (considering zero mean for the pressure) couple $(\textbf{u},\pi)$ with $\textbf{u}\in \textbf{W}^{1,\overline{q}}(\Omega)\cap \textbf{W}^{2,l}(\Omega)$ and $\nabla\pi\in \textbf{L}^l(\Omega)$ and it holds\cite{Beirao1}
\begin{align}
&\Vert\textbf{u}\Vert_{\textbf{W}^{1,\overline{q}}(\Omega)}\leq C(1+\Vert \textbf{f}\Vert^{\frac{3}{2p-1}}_{\textbf{L}^{p^\prime}(\Omega)}),\label{est1}\\&
\Vert\textbf{u}\Vert_{\textbf{W}^{2,l}(\Omega)}\leq C(\Vert \textbf{f}\Vert_{\textbf{L}^{p^\prime}(\Omega)}+\Vert \textbf{f} \Vert^{\frac{5-p}{2p-1}}_{\textbf{L}^{p^\prime}(\Omega)}),
\label{est2}
\end{align}
where $\overline{q}$ and $l$ are defined in \eqref{qbar}.
Note that the constants $C$ in these two estimates depend on the parameters of the problem and only on $\Vert \nabla\textbf{u}\Vert_{\textbf{L}^p(\Omega)}$ and $\Vert\pi\Vert_{L^{p^\prime}(\Omega)}$, which in this case are \textit{a priori} proven to be bounded and depending only on $\Vert \textbf{f}\Vert_{\textbf{W}^{-1,p^\prime}(\Omega)}$.
{In the case $d=2$ and $p\in\left(1,2\right)$ we have instead the following result\cite{BerselliRou}: supposing $\tau$ given by \eqref{Carreau} or \eqref{Carreau2} (actually the result holds for even more general assumptions on $\tau$), if $\textbf{f}\in \textbf{L}^{p^\prime}(\Omega)$ and $(\textbf{u},\pi)$ is a weak solution to our problem, then 
\begin{align}
\textbf{u}\in \textbf{W}^{1,\overline{q}}(\Omega)\cap\textbf{W}^{2,l}(\Omega), \quad \pi\in W^{1,l}(\Omega),
\label{reg}
\end{align} 
for $\overline{q}$ and $l$ defined as in \eqref{qbar}. This also holds for $d=3$ in the case $p\in(1,3/2]$.}
Following the ideas of Lemma 4 in Ref. \refcite{Abels}, we consider the case with variable viscosity, showing that we can reduce the problem to the weak formulation of \eqref{Stokes}, with $\textbf{f}$ substituted by a suitable $\textbf{F}\in \textbf{L}^{p^\prime}(\Omega)$.
We state the following 
\begin{thm}
	Assume that $\nu\in W^{1,\infty}(\R)$ and let $\Omega\subset \mathbb{R}^d,$ $d=2,3$ be a bounded domain of class $C^2$, such that $0<\nu_*\leq \nu(\cdot)\leq \nu^*$ in $\R$, $\phi\in W^{1,\infty}(\Omega)$ and $\textbf{f}\in \textbf{L}^{p^\prime}(\Omega)$. Consider the (unique) weak solution to \eqref{Stokes}. If $p\in(1,2)$ we have, for $d=2,3$ and assuming \eqref{Carreau} when $d=3$, \eqref{Carreau} or \eqref{Carreau2} when $d=2$,
	$$\textbf{u}\in\textbf{W}^{1,\overline{q}}(\Omega)\cap \textbf{W}^{2,l}(\Omega),\qquad \nabla\pi\in \textbf{L}^l(\Omega),$$ with $\overline{q}$ and $l$ defined in \eqref{qbar}.
		\label{THM1}
\end{thm}
\begin{proof}
	We consider $d=3$ and \eqref{Carreau}. In order to begin the proof, we first observe that, appealing to the theory of monotone operators we know that there exists a unique weak solution to \eqref{Stokes} (see, e.g., Ref. \refcite{Li}, Theorems 2.1-2.2). Moreover, we note that, by basic techniques, we get 
	\begin{align}
	\int_\Omega\nu(\phi)(1+\vert D\textbf{u}\vert)^{p-2}\vert D\textbf{u}\vert^2 dx \geq \nu_* 2^{p-2}\left(\int_\Omega\vert D\textbf{u}\vert^pdx-\vert\Omega\vert\right),
	\label{2p}
	\end{align}
	thus, testing the weak formulation with $\textbf{u}$,
	$$
	\Vert D\textbf{u}\Vert^p_{\textbf{L}^p(\Omega)}\leq \frac{2^{2-p}}{\nu_*}(\textbf{f},\textbf{u})+\vert\Omega\vert,
	$$
	so that
	\begin{align}
	\Vert \nabla\textbf{u}\Vert^{p-1}_{\textbf{L}^p(\Omega)}\leq C(\Vert \textbf{f}\Vert_{\textbf{W}^{-1,p^\prime}(\Omega)}+1).
	\label{Lp}
	\end{align}
	Note that $C$ is independent of $\phi$ and depends on $\nu_*$. Then, applying a well known Lemma (see, e.g., Lemma 3.1 of Ref. \refcite{Beirao1}), since we have, in the sense of distributions,
	$$
	\nabla\pi=\text{div}(\nu(\phi)(1+\vert D\textbf{u}\vert)^{p-2}D\textbf{u})+\textbf{f},
	$$
	we deduce, by \eqref{Lp}, 
	\begin{align}
	\Vert\pi\Vert_{L^{p^\prime}(\Omega)}\leq C(1+\Vert \textbf{f}\Vert_{\textbf{W}^{-1,p^\prime}(\Omega)}+\Vert \textbf{f}\Vert_{\textbf{W}^{-1,p^\prime}(\Omega)}^{\frac{1}{p-1}}),
	\label{Lppressure}
	\end{align}
	were $C$ depends also on $\nu^*$ and $\nu_*$. For convenience we can fix $\pi$ by assuming that its mean value in $\Omega$ vanishes.
	
	Therefore, we have obtained the necessary bounds on the two norms involved in the higher-order estimates. We need now to find a weak formulation appealing to \eqref{Stokes} with constant viscosity. 
	 We take $\textbf{v}\in \mathcal{V}$ and set  $\textbf{w}=\frac{\textbf{v}}{\nu(\phi)}-B\left[\text{div}\left(\frac{\textbf{v}}{\nu(\phi)}\right)\right]$, where $B$ is the Bogovskii operator defined in Section \ref{sec:notation}. This clearly gives $\textbf{w}\in \textbf{V}_{div}^p$. Taking then $\textbf{w}$ in the weak formulation we obtain 
	 \begin{align}
	 &\int_\Omega(1+\vert \nabla\textbf{u}\vert)^{p-2}D\textbf{u}\cdot D\textbf{v}dx\nonumber\\
  &=\left(\textbf{f},\frac{\textbf{v}}{\nu(\phi)}-B\left[\text{div}\left(\frac{\textbf{v}}{\nu(\phi)}\right)\right]\right)\nonumber 
  \\&-\left(\nu(\phi)(1+\vert D\textbf{u}\vert)^{p-2}D\textbf{u},\textbf{v}\otimes\nabla\left(\frac{1}{\nu(\phi)}\right)\right)\nonumber\\&+\left(\nu(\phi)(1+\vert D\textbf{u}\vert)^{p-2}D\textbf{u},\nabla B\left[\text{div}\left(\frac{\textbf{v}}{\nu(\phi)}\right)\right]\right).
	 \end{align}
	 Now by means of the assumptions on $\nu$ and \eqref{p2}, we immediately get 
	 \begin{align*}
	 &\left\vert\left(\textbf{f},\frac{\textbf{v}}{\nu(\phi)}-B\left[\text{div}\left(\frac{\textbf{v}}{\nu(\phi)}\right)\right]\right)\right\vert\leq \Vert\textbf{f}\Vert_{\textbf{L}^{p^\prime}(\Omega)}\left(\frac{1}{\nu_*}\Vert\textbf{v}\Vert_{\textbf{L}^p(\Omega)}+C\left\Vert\frac{\textbf{v}}{\nu(\phi)}\right\Vert_{\textbf{L}^p(\Omega)}\right)\\&\leq C\Vert\textbf{f}\Vert_{\textbf{L}^{p^\prime}(\Omega)}\Vert\textbf{v}\Vert_{\textbf{L}^p(\Omega)}.
	 \end{align*}
	 Then, by H\"{o}lder's inequality, we have
	 \begin{align*}
	 &\left\vert\left(\nu(\phi)(1+\vert D\textbf{u}\vert)^{p-2}D\textbf{u},\textbf{v}\otimes\nabla\left(\frac{1}{\nu(\phi)}\right)\right)	\right\vert\\&= \left\vert\left(\nu(\phi)(1+\vert D\textbf{u}\vert)^{p-2}D\textbf{u},\textbf{v}\otimes\left(\frac{\nu^\prime(\phi)}{\nu^2(\phi)}\nabla\phi\right)\right)	\right\vert\\&\leq C\int_\Omega(1+\left\vert D\textbf{u}\right\vert)^{p-1}\vert\textbf{v}\vert\vert\nabla\phi\vert dx \\&\leq C\norm{\nabla\phi}_{\textbf{L}^{p^\prime}(\Omega)}\norm{\textbf{v}}_{\textbf{L}^{p}(\Omega)}+C\norm{\nabla\phi}_{\textbf{L}^\infty(\Omega)}\Vert D\textbf{u}\Vert_{\textbf{L}^p(\Omega)}^{p-1}\Vert\textbf{v}\Vert_{\textbf{L}^p(\Omega)}\\&\leq C(1+\Vert D\textbf{u}\Vert_{\textbf{L}^p(\Omega)}^{p-1})\norm{\nabla\phi}_{\textbf{L}^\infty(\Omega)}\Vert\textbf{v}\Vert_{\textbf{L}^p(\Omega)},
	 \end{align*}
	 and analogously, appealing to \eqref{p1},
	 
	 \begin{align*}
	 &\left\vert\left(\nu(\phi)(1+\vert D\textbf{u}\vert)^{p-2}D\textbf{u},\nabla B\left[\text{div}\left(\frac{\textbf{v}}{\nu(\phi)}\right)\right]\right)\right\vert\\&\leq C
	 \norm{ 1+\vert D\textbf{u}\vert}_{\textbf{L}^p(\Omega)}^{p-1}\left\Vert\nabla B\left[\text{div}\left(\frac{\textbf{v}}{\nu(\phi)}\right)\right]\right\Vert_{\textbf{L}^p(\Omega)}\\&\leq C
	 \left(1+\Vert D\textbf{u}\Vert_{\textbf{L}^p(\Omega)}^{p-1}\right)\left\Vert\text{div}\left(\frac{\textbf{v}}{\nu(\phi)}\right)\right\Vert_{\textbf{L}^p(\Omega)}\\&=C\left(1+\Vert D\textbf{u}\Vert_{\textbf{L}^p(\Omega)}^{p-1}\right)\left\Vert\left(\nabla\frac{1}{\nu(\phi)}\cdot \textbf{v}\right)\right\Vert_{\textbf{L}^p(\Omega)}\\&=C\left(1+\Vert D\textbf{u}\Vert_{\textbf{L}^p(\Omega)}^{p-1}\right)\left\Vert\left(\frac{\nu^\prime(\phi)}{\nu^2(\phi)}\nabla\phi\cdot \textbf{v}\right)\right\Vert_{\textbf{L}^p(\Omega)}\\&\leq C\left(1+\Vert D\textbf{u}\Vert_{\textbf{L}^p(\Omega)}^{p-1}\right)\norm{\nabla\phi}_{\textbf{L}^\infty(\Omega)}\Vert\textbf{v}\Vert_{\textbf{L}^p(\Omega)}.
	 \end{align*}
	 Therefore, we can define a linear continuous functional $\textbf{F}: \mathcal{V} \to \mathbb{R}$ by setting 
	 $$
	 \int_\Omega(1+\vert D\textbf{u}\vert)^{p-2}D\textbf{u}\cdot \nabla\textbf{v}dx\nonumber=<\textbf{F},\textbf{v}>\qquad\forall \textbf{v}\in \mathcal{V},
	 $$
	 and $\textbf{F}$ can be uniquely extended by density to a linear continuous operator (not relabeled) over $\textbf{L}^p(\Omega)$. Thus, again by a density argument, we obtain 
	  	 \begin{align}
	  \int_\Omega(1+\vert D\textbf{u}\vert)^{p-2}D\textbf{u}\cdot \nabla\textbf{v}dx=<\textbf{F},\textbf{v}>\qquad\forall \textbf{v}\in \textbf{V}_{div}^p,
	  \label{Stokes2}
	  \end{align}
	  where, identifying $\textbf{L}^{p^\prime}(\Omega)$ with the dual of $\textbf{L}^p(\Omega)$, we have
	  \begin{equation}
	  \Vert \textbf{F}\Vert_{\textbf{L}^{p^\prime}(\Omega)}\leq C\left(\Vert\textbf{f}\Vert_{\textbf{L}^{p^\prime}(\Omega)}+\left(1+\Vert D\textbf{u}\Vert_{\textbf{L}^p(\Omega)}^{p-1}\right)\norm{\nabla\phi}_{\textbf{L}^\infty(\Omega)}\right),
	  \end{equation} 
	  which means, recalling \eqref{Lp}, 
	    \begin{equation}
	  \Vert \textbf{F}\Vert_{\textbf{L}^{p^\prime}(\Omega)}\leq C\left(\Vert\textbf{f}\Vert_{\textbf{L}^{p^\prime}(\Omega)}+\norm{\nabla\phi}_{\textbf{L}^\infty(\Omega)}\right).
	  \label{F}
	  \end{equation} 
	Let us notice that \eqref{Stokes2} coincides (substituting $\textbf{f}$  with $\textbf{F}$) with the weak formulation of \eqref{Stokes} with constant viscosity, being $\phi\in W^{1,\infty}(\Omega)$. Hence we can apply the aforementioned regularity results to obtain, for $d=3$, from \eqref{est1}-\eqref{est2},  
	  $$\textbf{u}\in \textbf{W}^{1,\overline{q}}(\Omega)\cap \textbf{W}^{2,l}(\Omega),\qquad \nabla\pi\in \textbf{L}^l(\Omega),$$ with $\overline{q}$ and $l$ defined as in \eqref{qbar}. Note that, due to \eqref{est1}, the exponent of the $\textbf{L}^{p^\prime}$-norm of the forcing term $\textbf{F}$ appearing in the estimate is $\frac{ 3}{2p-1}$, i.e.,
	  \begin{equation}
	  \Vert\textbf{u}\Vert_{\textbf{W}^{1,\overline{q}}(\Omega)}\leq C(1+\Vert \textbf{F}\Vert^{\frac{ 3}{2p-1}}_{\textbf{L}^{p^\prime}(\Omega)}).
	  \label{exp}
	  \end{equation}
  We stress again that this has been possible since we have obtained the \textit{a priori} estimates \eqref{Lp}-\eqref{Lppressure} independently of the presence of the variable viscosity: indeed, the constant $C$ in \eqref{exp} depends on $\Vert \nabla\textbf{u}\Vert_{\textbf{L}^p(\Omega)}$ and $\Vert\pi\Vert_{L^{p^\prime}(\Omega)}$. {By the same arguments, we obtain the identical regularity also in the case $p\in(1,3/2]$, thanks to Ref. \refcite{BerselliRou}, Theorem 2.29.}
  
  For the case $d=2$, notice that, up to minor modifications (e.g., when we assume \eqref{Carreau2}, the constant in \eqref{2p} becomes $\nu_*2^{\frac{p-2}{2}}$), the same results \eqref{Stokes2}-\eqref{F} still hold also for the case of $\tau$ given by \eqref{Carreau2}. Indeed, observe that $(1+\vert D\textbf{u}\vert^2)^{\frac{p-2}{2}}\leq C(1+\vert D\textbf{u}\vert)^{p-2}$. Therefore, the regularity result follows from \eqref{Lp} (indeed the weak solution needs to belong to $\textbf{W}^{1,p}(\Omega)$ independently of $\phi$), \eqref{F} and the regularity result \eqref{reg}. The proof is finished.
\end{proof}
\paragraph{Regularity for the Navier-Stokes system.}
We consider the $d-$dimensional system 
\begin{align}
\begin{cases}
-\text{div}(\nu(\phi)(1+\vert D\textbf{u}\vert)^{p-2}D\textbf{u})+(\textbf{u}\cdot\nabla)\textbf{u}+\nabla \pi=\textbf{f},\\
\text{div }\textbf{u}=0,
\end{cases}
\label{NavierStokes}
\end{align}
in $\Omega$, with no-slip boundary conditions, for $1<p<2$, where $\phi$ is a generic function and $\textbf{f}\in \textbf{L}^{p^\prime}(\Omega)$.

Following Ref. \refcite{Beirao2}, Section 6, for the case $d=3$, we can prove the following:

\begin{thm}
	Assume that $\nu\in W^{1,\infty}(\R)$, such that $0<\nu_*\leq \nu(\cdot)\leq \nu^*$ in $\R$, $\phi\in W^{1,\infty}(\Omega)$ and $\textbf{f}\in \textbf{L}^{p^\prime}(\Omega)$. Consider a weak solution to \eqref{NavierStokes}. If $p>p_0=20/11$ we have 
	$$\textbf{u}\in \textbf{W}^{1,\overline{q}}(\Omega)\cap \textbf{W}^{2,l}(\Omega),\qquad \nabla\pi\in \textbf{L}^l(\Omega),$$ with $\overline{q}$ and $l$ defined as in  \eqref{qbar}.
	\label{THM2}
\end{thm}
\begin{proof}
 We  observe that if we consider $\tilde{\textbf{f}}=\textbf{f}-(\textbf{u}\cdot\nabla)\textbf{u}$ we can recast the problem exactly as in \eqref{Stokes}, with $\tilde{\textbf{f}}$ in place of $\textbf{f}$. Before doing this, we note that, as in the Stokes case, the crucial point is to show that the \textit{a priori} estimates of $\Vert \nabla\textbf{u}\Vert_p$ and $\Vert\pi\Vert_{p^\prime}$ are independent of the extra term $(\textbf{u}\cdot\nabla)\textbf{u}$. Indeed, the constant bounding the more regular norms in Theorem \ref{THM1} depends on them. This is quite immediate, as noticed in Appendix A of Ref. \refcite{Beirao1}, since $((\textbf{u}\cdot\nabla)\textbf{u},\textbf{u})=0$ and thus this term does not appear in the energy estimate leading to \eqref{Lp}, which can be carried out identically. Concerning the pressure $\pi$, we notice that if $p\geq9/5$, which is already guaranteed, being $p_0>9/5$ {. Observe that there is no need of finding a better estimate for the pressure, since the lower bound $p_0$ comes from the treatment of the convective term. Using the Sobolev embedding $\textbf{W}^{1,p}(\Omega)\hookrightarrow \textbf{L}^{2p'}(\Omega)$, we find
 $$
 \Vert \mathbf{u}\Vert_{\textbf{L}^{2p'}(\Omega)}\leq C\Vert \nabla\mathbf{u}\Vert_{\textbf{L}^p(\Omega)},
 $$
 and thus
 }
	$$
	\Vert(\textbf{u}\cdot\nabla)\textbf{u}\Vert_{\textbf{W}^{-1,p^\prime}(\Omega)}\leq C\Vert\textbf{u}\Vert^2_{\textbf{L}^{2p^\prime}(\Omega)}\leq C\Vert\nabla\textbf{u}\Vert_{\textbf{L}^p(\Omega)}^2.
	$$
	This result and \eqref{Lp} yield
	
		\begin{align}
	\Vert\pi\Vert_{L^{p^\prime}(\Omega)}&\leq C(\Vert(\textbf{u}\cdot\nabla)\textbf{u}\Vert_{\textbf{W}^{-1,p^\prime}(\Omega)}+\Vert \textbf{f}\Vert_{\textbf{W}^{-1,p^\prime}(\Omega)}+1)\nonumber\\&\leq C(\Vert \textbf{f}\Vert_{\textbf{W}^{-1,p^\prime}(\Omega)}^{\frac{2}{p-1}}+\Vert \textbf{f}\Vert_{\textbf{W}^{-1,p^\prime}(\Omega)}+1),
	\label{Lppressure2}
	\end{align}
	which is enough for our purposes, {being the constant $C>0$ independent of $\phi$}. Thus, we only need to estimate $\Vert(\textbf{u}\cdot\nabla)\textbf{u}\Vert_{\textbf{L}^{p^\prime}(\Omega)}$ in order to follow the proof Theorem \ref{THM1}. We can repeat word by word the estimates devised in Ref. \refcite{Beirao2}, Section 6 to obtain in the end, in force of the validity of \eqref{Lp}, that, if $p>p_0$,
	$$
	\Vert(\textbf{u}\cdot\nabla)\textbf{u}\Vert_{\textbf{L}^{p^\prime}(\Omega)}\leq C\Vert\nabla\textbf{u}\Vert_{\textbf{L}^{\overline{q}}(\Omega)}^\gamma,
	$$
	where $\overline{q}$ is defined in \eqref{qbar} and $0\leq\gamma<\frac{2p-1}{3}$. We thus follow the proof of Theorem \ref{THM1}: in particular, from \eqref{exp} we have in $\textbf{F}$ the extra term $(\textbf{u}\cdot\nabla)\textbf{u}$ (indeed, we have $\tilde{\textbf{f}}$ in place of $\textbf{f}$), therefore, being $\phi\in W^{1,\infty}(\Omega)$, we deduce 
	  \begin{equation}
	\Vert\textbf{u}\Vert_{\textbf{W}^{1,\overline{q}}(\Omega)}\leq C\left(1+\Vert(\textbf{u}\cdot\nabla)\textbf{u}\Vert_{\textbf{L}^{p^\prime}(\Omega)}^{\frac{3}{2p-1}}\right)\leq C\left(1+\Vert\nabla\textbf{u}\Vert_{\textbf{L}^{\overline{q}}(\Omega)}^{\frac{3\gamma}{2p-1}}\right)
	\label{exp2}
	\end{equation}
	and, since $\frac{3\gamma}{2p-1}<1$, by Young's inequality we infer that $	\Vert\textbf{u}\Vert_{\textbf{W}^{1,\overline{q}}(\Omega)}$ is bounded. From this result, which bounds $\Vert(\textbf{u}\cdot\nabla)\textbf{u}\Vert_{\textbf{L}^{p^\prime}(\Omega)}$, we also get the $\textbf{W}^{2,l}$-regularity of $\textbf{u}$ given in the statement of this theorem (see \eqref{est2}). This ends the proof. 
\end{proof}
\paragraph{Useful estimates in 2D and 3D.} We now recall some important estimates used in the sequel. We start with a Sobolev embedding in 3D (see Ref. \refcite{Brezis}): let $p\in(3/2,2)$, then
\begin{align}
W^{1,p}(\Omega)\hookrightarrow L^{3+\delta}(\Omega),
\label{embedding}
\end{align}
for $\delta=\frac{6p-9}{3-p}>0$.
Moreover, we also have
$$
W^{2,3+\delta}(\Omega)\hookrightarrow W^{1,\infty}(\Omega),
$$
and the following Gagliardo-Nirenberg's inequality  
\begin{align}
\Vert f\Vert_{W^{1,\infty}(\Omega)}\leq C\Vert f\Vert^{\chi_3}_{W^{1,2}(\Omega)}\Vert f\Vert_{W^{2,3+\delta}(\Omega)}^{1-\chi_3},
\label{Gagliardo}
\end{align}
with $\chi_3=\frac{2\delta}{9+5\delta}<1$. {Observe that the validity of this inequality which allows to set the lower bound for $p$ at $3/2$ when $d=3$, also in the case of the Stokes problem (see Remark \ref{validity}).} 

We continue with a similar Sobolev embedding in 2D (see again Ref. \refcite{Brezis}): let $p\in(1,2)$, then
\begin{align}
W^{1,p}(\Omega)\hookrightarrow L^{2+\delta}(\Omega),
\label{embedding2}
\end{align}
for $\delta=\frac{4(p-1)}{2-p}>0$.
In this case, we have the Gagliardo-Nirenberg inequality  
\begin{align}
\Vert f\Vert_{W^{1,\infty}(\Omega)}\leq C\Vert f\Vert^{\chi_2}_{W^{1,2}(\Omega)}\Vert f\Vert_{W^{2,2+\delta}(\Omega)}^{1-\chi_2},
\label{Gagliardo2}
\end{align}
with $\chi_2=\frac{\delta}{2(1+\delta)}<1$.
We can now prove Theorems \ref{existregol} and \ref{existregol2}.
\subsubsection{Proof of Theorems \ref{existregol} and \ref{existregol2}}
	The existence of a weak solution is the result of Theorem \ref{exist}. In order to gain the additional regularity we want to apply Theorem \ref{THM2} to this specific weak solution, which can be considered as a weak solution to \eqref{NavierStokes}, neglecting the equation for $\Theta=\vartheta+\Theta_0$ (with $\phi=\Theta$). Therefore, the only assumption to be verified is that $\vartheta\in W^{1,\infty}(\Omega)$ and then the proof is finished, since, due to \eqref{extra} and \eqref{extraB1} for the 3D and 2D case, respectively, $\Theta_0\in W^{2,d+\delta}(\Omega)$, for $\delta$ depending on $d$ (see \eqref{delta} and \eqref{a1}).
	By well-known elliptic regularity results {(indeed $\vartheta$ is a weak solution in the present case, see Remark \ref{test})}
	\begin{equation}
	\Vert\vartheta\Vert_{W^{2, d+\delta}(\Omega)} \leq \frac{1}{\kappa}\left(\Vert g\Vert_{L^{ d+\delta}(\Omega)}+\Vert\textbf{u}\cdot \nabla\vartheta\Vert_{L^{ d+\delta}(\Omega)}+\Vert\textbf{u}\cdot \nabla\Theta_0\Vert_{L^{ d+\delta}(\Omega)}\right).
	\label{theta1}
	\end{equation}
	We need to observe that, due to \eqref{embedding} and \eqref{Gagliardo} for $d=3$ and \eqref{embedding2} and \eqref{Gagliardo2} for $d=2$, by Young's inequality,
	\begin{align*}
		\frac{1}{\kappa}\Vert\textbf{u}\cdot \nabla\vartheta\Vert_{L^{ d+\delta}(\Omega)}&\leq\frac{1}{\kappa} \Vert\textbf{u}\Vert_{\textbf{L}^{ d+\delta}(\Omega)}\Vert\nabla\vartheta\Vert_{\textbf{L}^\infty(\Omega)}\\&\leq C\Vert\textbf{u}\Vert_{\textbf{W}^{1,p}(\Omega)}\Vert \vartheta\Vert^{\chi_d}_{W^{1,2}(\Omega)}\Vert \vartheta\Vert_{W^{2, d+\delta}(\Omega)}^{1-\chi_d}\\&\leq C\Vert\textbf{u}\Vert_{\textbf{W}^{1,p}(\Omega)}^{1/\chi_d}\Vert \vartheta\Vert_{W^{1,2}(\Omega)}+\frac{1}{2}\Vert \vartheta\Vert_{W^{2, d+\delta}(\Omega)}.
	\end{align*}
Note that this is possible since $1-\chi_d <1$. Moreover, again by \eqref{embedding} and the properties of the lift operator, we have that
	\begin{align*}
\frac{1}{\kappa}\Vert\textbf{u}\cdot \nabla\Theta_0\Vert_{\textbf{L}^{ d+\delta}(\Omega)}&\leq\frac{1}{\kappa} \Vert\textbf{u}\Vert_{\textbf{L}^{ d+\delta}(\Omega)}\Vert\nabla\Theta_0\Vert_{\textbf{L}^\infty(\Omega)}\\&\leq C\Vert\textbf{u}\Vert_{\textbf{W}^{1,p}(\Omega)}\Vert \Theta_0\Vert_{W^{2, d+\delta}(\Omega)}\\&\leq C\Vert\textbf{u}\Vert_{\textbf{W}^{1,p}(\Omega)}\Vert \theta\Vert_{W^{2-\frac{1}{ d+\delta}, d+\delta}(\partial\Omega)}.
\end{align*}
Collecting these results and using them in \eqref{theta1}, we immediately deduce that $\Vert\vartheta\Vert_{W^{2, d+\delta}(\Omega)}$ is bounded, meaning that, by the already mentioned Sobolev embedding, $\Theta=\vartheta+\Theta_0\in W^{2, d+\delta}(\Omega)\hookrightarrow W^{1,\infty}(\Omega)$. Note that in the case $d=3$ with \eqref{Carreau} this is enough to apply Theorem \ref{THM2} and conclude the proof, whereas for $d=2$, if we consider the Stokes problem and either \eqref{Carreau} or \eqref{Carreau2}, neglecting the convective term $(\textbf{u}\cdot\nabla)\textbf{u}$, by Theorem \ref{THM1} the proof follows. Note that, in the case of the Stokes problem also for $d=3$, we can repeat the same proof and conclude in force of Theorem \ref{THM1}, which holds also for any $p\in(3/2,2)$. 
The first part of the proof is thus concluded.

If we consider the additional stronger hypotheses \eqref{extra2} and \eqref{extraB2}, for $d=3$ and $d=2$, respectively, we first note that $\Theta_0\in W^{3, s}(\Omega)$, with {$s=\min\{\overline{q},d+\delta\}$}.
In the case $d=3$ we consider both the Navier-Stokes problem with the restriction $p\in(p_0,2)$ and the Stokes problem with $p\in(3/2,2)$ (with \eqref{Carreau}). In the case $d=2$ we consider the Stokes problem {(with either \eqref{Carreau} or \eqref{Carreau2}) with $p\in(1,2)$}. We can now apply a bootstrap argument.  Indeed, again by elliptic regularity we have 
	\begin{align}
&\Vert\vartheta\Vert_{W^{3, s}(\Omega)}\leq \frac{1}{\kappa}\left(\Vert g\Vert_{W^{1, s}(\Omega)}+\Vert\textbf{u}\cdot \nabla\vartheta\Vert_{W^{1, d+\delta}(\Omega)}+\Vert\textbf{u}\cdot \nabla\Theta_0\Vert_{W^{1, s}(\Omega)}\right),
\label{theta6}
\end{align}
{with $s=\min\{\overline{q},d+\delta\}$.}
Clearly, the only new terms to be controlled are the ones related to the spatial gradients. We have
$$
\Vert \nabla(\textbf{u}\cdot \nabla\vartheta)\Vert_{\textbf{L}^{ s}(\Omega)}\leq \Vert\nabla\textbf{u}\Vert_{\textbf{L}^{ s}(\Omega)}\Vert\vartheta\Vert_{W^{1,\infty}(\Omega)}+\Vert\textbf{u}\Vert_{\textbf{L}^\infty(\Omega)}\Vert\vartheta\Vert_{W^{2, s}(\Omega)},
$$
but, for $d=3$, $ 3+\delta<\overline{q}$ for $p\in (3/4,2)$, {entailing $s=3+\delta$.} Due to the embeddings \eqref{embedding} and $\textbf{W}^{1,\overline{q}}(\Omega)\hookrightarrow \textbf{L}^\infty(\Omega)$, we infer that this term is then bounded thanks to the regularity obtained in the first part of the proof. {For the case $d=2$, we have $s=\overline{q}<2+\delta$ and since we have $\mathbf{u}\in\textbf{W}^{1,\overline{q}}(\Omega)\hookrightarrow \textbf{L}^\infty(\Omega)$, being $\overline{q}>2$ for any $p\in(1,2)$, we get the same result using for $\vartheta$ the embedding $W^{2,s}(\Omega)\hookrightarrow W^{2,2+\delta}(\Omega)$}. 
Similarly, we have 
$$
\Vert \nabla(\textbf{u}\cdot \nabla\Theta_0)\Vert_{\textbf{L}^{ s}(\Omega)} \leq \Vert\nabla\textbf{u}\Vert_{\textbf{L}^{ s}(\Omega)}\Vert\Theta_0\Vert_{W^{1,\infty}(\Omega)}+\Vert\textbf{u}\Vert_{\textbf{L}^\infty(\Omega)}\Vert\Theta_0\Vert_{W^{2, s}(\Omega)},
$$
where $\Vert\Theta_0\Vert_{W^{1,\infty}(\Omega)}\leq C\Vert\Theta_0\Vert_{W^{2, s}(\Omega)}\leq C\Vert\theta\Vert_{W^{2-\frac{1}{ s}, s}(\Omega)}.$ {Indeed, for $d=3$, we have $s=3+\delta$ and $W^{2,3+\delta}(\Omega)\hookrightarrow W^{1,\infty}(\Omega)$, whereas, when $d=2$, $2<s=\overline{q}<2+\delta$, so that, again, $W^{2,s}(\Omega)\hookrightarrow W^{1,\infty}(\Omega)$.} Therefore, this term is bounded, concluding the proof, thanks to assumption \eqref{extra2} for $d=3$ and \eqref{extraB2} for $d=2$.
\subsection{Proofs of Section \ref{sec:approx}}
\label{appro}
\subsubsection{Proof of Theorem \ref{existreg}}
Let $\mathbf{u}^\# \in \textbf{L}^4_{div}$ be given.
Then, recalling that $\mbox{div }\mathbf{u}^\#=0$,
it is not difficult to prove that there is a unique $\vartheta^\#=\vartheta^\#(\mathbf{u}^\#)\in V$ which solves
\begin{equation}
\int_\Omega \kappa \nabla \vartheta^\# \cdot \nabla \phi + \int_\Omega u^\#_j\vartheta^\# \partial_{x_j}\phi = \langle g,\phi \rangle
+ \int_\Omega u^\#_j\Theta_0 \partial_{x_j}\phi \qquad \forall\,\phi\in V_0. \label{tempfix}
\end{equation}
Consider now the problem of finding $\mathbf{u}^* \in \textbf{V}^r_{div}$ such that
\begin{align}
&\int_\Omega \left[\nu(\vartheta^\# + \Theta_0)\left( \tau_{ij}(x,\varepsilon(\mathbf{u}^*)) + \sigma \vert \varepsilon(\mathbf{u}^*))\vert^{r-2} \varepsilon_{ij}(\mathbf{u}^*)\right)\right]\varepsilon_{ij}(\mathbf{v})dx \nonumber\\
&-\int_\Omega u^\#_j u^\#_i \partial_{x_j} v_i dx = \langle \mathbf{f}, \mathbf{v} \rangle \qquad \forall\, \mathbf{v}\in \textbf{V}^r_{div}. \label{fluidfix}
\end{align}
This problem has a unique solution.
Thus we have constructed a nonlinear mapping $\mathcal{F}: \textbf{L}^4_{div} \to \textbf{L}^4_{div}$ defined as follows $\mathcal{F}(\mathbf{u}^\#)=\mathbf{u}^*$.
It is easy to realize that $\mathcal{F}$ maps bounded sets of $\textbf{L}^4_{div}$ into bounded sets of $\textbf{V}^r_{div}\hookrightarrow\hookrightarrow \textbf{L}^4_{div}$, {for $r\geq2$}. Just take $\mathbf{v}=\mathbf{u}^*$ in \eqref{fluidfix}.
We now prove that $\mathcal{F}$ is also continuous so $\mathcal{F}$ is a compact operator. Indeed, let $\{ \mathbf{u}_n\}$ be such that $\mathbf{u}_n \to \mathbf{u}^\#$ in $\textbf{L}^4_{div}$. Then there exists $\{\vartheta_{n_h}\}$ which converges (weakly) in $V_0$ and
strongly in $H$ to $\vartheta^\#$.
On the other hand, $\{ \mathcal{F}(\mathbf{u}_{n_h})\} \subset \textbf{L}^4_{div}$ is also bounded in
$\textbf{V}^r_{div}$. Using classical arguments (see, for instance, Ref. \refcite{Li}), we can prove that there exists a subsequence $\{\mathbf{u}_{n_{h_m}}\}$ such that
$\{\mathcal{F}(\mathbf{u}_{n_{h_m}})\}$
converges weakly in $\textbf{V}^r_{div}$ and strongly in $\textbf{L}^4_{div}$ to a solution $\tilde{\mathbf{u}}$ to \eqref{fluidfix} which coincides, due to uniqueness, with $\mathcal{F}(\mathbf{u}^\#)$.
The class limit $\Lambda$ of $\{ \mathcal{F}(\mathbf{u}_n)\}$ is non-empty and bounded in $\textbf{V}^r_{div}$. Arguing as above, it is easy to show that any $\mathbf{w}\in \Lambda$ is such that $\mathbf{w}=\mathcal{F}(\mathbf{u}^\#)$. Hence $\mathcal{F}$ is continuous from $\textbf{L}^4_{div}$ to itself.

Observe now that if $\mathbf{w}\in \textbf{L}^4_{div}$ is such that $\mathbf{w}= \lambda\mathcal{F}(\mathbf{w})$ for some $\lambda\in(0,1)$ then there
exists $C_1>0$ such that $\Vert \mathbf{w}\Vert_{\textbf{L}^4_{div}}\leq C_1$. Indeed, we have
\begin{align}
&\int_\Omega \left[\nu(\vartheta(\mathbf{w} + \Theta_0) \left(\tau_{ij}(x,\varepsilon(\lambda^{-1}\mathbf{w})) + \sigma \vert \varepsilon(\lambda^{-1}\mathbf{w}))\vert^{r-2} \varepsilon_{ij}(\lambda^{-1} \mathbf{w})\right)\right]\varepsilon_{ij}(\mathbf{v})dx \nonumber\\
&-\int_\Omega w_j w_i \partial_{x_j} v_i dx = \langle \mathbf{f}, \mathbf{v} \rangle \qquad \forall\, \mathbf{v}\in \textbf{V}^r_{div}. \label{fluidfix2}
\end{align}
Then, taking $\mathbf{v}=\mathbf{w}$ and using Young's  and Korn's inequalities, we deduce the existence of $C_2$ depending on $d$, $\Omega$, $\sigma$, and $\Vert\mathbf{f}\Vert_{\textbf{W}^{-1,p^\prime}(\Omega)}$ such that
\begin{equation*}
\Vert \mathbf{u}\Vert_{\textbf{V}^r_{div}} \leq C_2.
\end{equation*}
Thus we can apply Schaefer's fixed point theorem and conclude that  $\mathcal{F}$ has a fixed point in $\textbf{L}^4_{div}$. This is equivalent to say that the approximating problem \eqref{fluidapp}-\eqref{tempapp} has a solution $(\mathbf{u}_\sigma, \vartheta_\sigma)$.
\subsubsection{Proof of Theorems \ref{uniquebis} and \ref{unique}}
\label{uniqueness}
We start from the proof of Theorem  \ref{unique} to introduce the setting, since the proof of Theorem \ref{uniquebis} exploits similar estimates and notations. 
Suppose $(\mathbf{u}^k,\vartheta^k) \in \textbf{V}^r_{div} \times V_0$, $k=1,2$, satisfy the following system
\begin{align}
&\int_\Omega \left[\nu(\vartheta + \Theta_0)\left( \tau_{ij}(x,\varepsilon(\mathbf{u})) +\sigma \vert \varepsilon(\mathbf{u}))\vert^{r-2} \varepsilon_{ij}(\mathbf{u})\right)\right]\varepsilon_{ij}(\mathbf{v})dx \nonumber\\
& -\int_\Omega u_j u_i \partial_{x_j} v_i dx = \langle \mathbf{f}, \mathbf{v} \rangle \qquad \forall\, \mathbf{v}\in \textbf{V}^r_{div}, \label{fluidapps}\\
&\int_\Omega \kappa \nabla \vartheta \cdot \nabla \phi + \int_\Omega u_j\vartheta \partial_{x_j}\phi = \langle g,\phi \rangle
+ \int_\Omega u_j\Theta_0 \partial_{x_j}\phi \qquad \forall\,\phi\in V_0. \label{tempapps}
\end{align}
Suppose $r=2$ and $d=3$, that is, the worst case {in the approximating scheme}.
Then set $\tilde{\mathbf{u}}=\mathbf{u}^1 -\mathbf{u}^2$ and $\tilde{\vartheta}= \vartheta^1 - \vartheta^2$ and observe that $(\tilde{\mathbf{u}},\tilde\vartheta)$
satisfies
\begin{align}
&\int_\Omega \nu(\vartheta^1 + \Theta_0) \left[(\tau_{ij}(x,\varepsilon(\mathbf{u}^1))-  \tau_{ij}(x,\varepsilon(\mathbf{u}^2))\right]\varepsilon_{ij}(\mathbf{v})dx
\nonumber\\&+\sigma \int_\Omega \nu(\vartheta^1 + \Theta_0)\nabla\tilde{\mathbf{u}}\cdot\nabla\mathbf{v}dx \nonumber\\&=-\int_\Omega \left[\nu(\vartheta^1 + \Theta_0) - \nu(\vartheta^2 + \Theta_0)\right](\tau_{ij}(x,\varepsilon(\mathbf{u}^2)) \varepsilon_{ij}(\mathbf{v})dx\nonumber\\
&-\sigma\int_\Omega \left[\nu(\vartheta^1 + \Theta_0) - \nu(\vartheta^2 + \Theta_0)\right] \nabla \mathbf{u}^2\cdot\nabla\mathbf{v}dx\nonumber\\
&+ \int_\Omega (\tilde{u}_j u^1_i + u^2_j \tilde{u}_i) \partial_{x_j} v_i dx
\qquad \forall\, \mathbf{v}\in \textbf{V}^r_{div}, \label{fluiddiff}\\
&\int_\Omega \kappa \nabla \tilde\vartheta \cdot \nabla \phi + \int_\Omega u^1_j\tilde\vartheta \partial_{x_j}\phi =
-\int_\Omega \tilde u_j\vartheta^2 \partial_{x_j}\phi + \int_\Omega \tilde u_j\Theta_0 \partial_{x_j}\phi \qquad \forall\,\phi\in V_0. \label{tempdiff}
\end{align}
Take now $\mathbf{v}=\tilde{\mathbf{u}}$ in \eqref{fluiddiff} and $\phi=\tilde\vartheta$ in \eqref{tempdiff}. Recalling that $\mathbf{u}_1$ is divergence free, this gives
\begin{align}
&\int_\Omega \nu(\vartheta^1 + \Theta_0) \left[(\tau_{ij}(x,\varepsilon(\mathbf{u}^1))-  \tau_{ij}(x,\varepsilon(\mathbf{u}^2))\right]\varepsilon_{ij}(\tilde{\mathbf{u}})dx
+\sigma \int_\Omega \nu(\vartheta^1 + \Theta_0)\vert \nabla\tilde{\mathbf{u}}\vert^2 dx \nonumber\\
&=-\int_\Omega \left[\nu(\vartheta^1 + \Theta_0) - \nu(\vartheta^2 + \Theta_0)\right](\tau_{ij}(x,\varepsilon(\mathbf{u}^2)) \varepsilon_{ij}(\tilde{\mathbf{u}})dx\nonumber\\
&-\sigma\int_\Omega \left[\nu(\vartheta^1 + \Theta_0) - \nu(\vartheta^2 + \Theta_0)\right] \nabla \mathbf{u}^2\cdot\nabla\tilde{\mathbf{u}}dx\nonumber\\
&+ \int_\Omega (\tilde{u}_j u^1_i + u^2_j \tilde{u}_i) \partial_{x_j} \tilde{u}_i dx
\qquad \forall\, \mathbf{v}\in \textbf{V}^r_{div}, \label{fluiddiff2}\\
&\kappa\int_\Omega \vert \nabla \tilde\vartheta\vert^2 =
-\int_\Omega \tilde u_j\vartheta^2 \partial_{x_j}\tilde\vartheta + \int_\Omega \tilde u_j\Theta_0 \partial_{x_j}\tilde\vartheta \qquad \forall\,\phi\in V_0. \label{tempdiff2}
\end{align}
Using (H1) and (H4), we deduce the inequalities
\begin{align}
&\sigma\nu_1 \int_\Omega \vert \nabla\tilde{\mathbf{u}}\vert^2 dx \leq I_1,  \label{fluiddiff3}\\
&\kappa\int_\Omega \vert \nabla \tilde\vartheta\vert^2 \leq I_2, \label{tempdiff3}
\end{align}
where
\begin{align}
I_1&=-\int_\Omega \left[\nu(\vartheta^1 + \Theta_0) - \nu(\vartheta^2 + \Theta_0)\right](\tau_{ij}(x,\varepsilon(\mathbf{u}^2)) \varepsilon_{ij}(\tilde{\mathbf{u}})dx\nonumber\\
&-\sigma\int_\Omega \left[\nu(\vartheta^1 + \Theta_0) - \nu(\vartheta^2 + \Theta_0)\right] \nabla \mathbf{u}^2\cdot\nabla\tilde{\mathbf{u}}dx\nonumber\\
&+ \int_\Omega (\tilde{u}_j u^1_i + u^2_j \tilde{u}_i) \partial_{x_j} \tilde{u}_i dx,
\label{term1}\\
&I_2 = -\int_\Omega \tilde u_j(\vartheta^2 - \Theta_0)\partial_{x_j}\tilde\vartheta. \label{term2}
\end{align}
Observe now that, recalling (H1), thanks to H\"{o}lder's inequality we have
\begin{align}
I_1 &\leq \Vert \nu^\prime\Vert_{L^\infty(\mathbb{R})}\Vert \tilde \vartheta \Vert_{L^6(\Omega)}\left(\Vert \tau(\cdot,\varepsilon(\mathbf{u}^2))\Vert_{\textbf{L}^3(\Omega)} \Vert \varepsilon(\tilde{\mathbf{u}})\Vert_{\textbf{L}^2(\Omega)} + \sigma N \Vert \nabla \tilde{\mathbf{u}}\Vert_{\textbf{L}^2(\Omega)}\right)\nonumber\\
&+ (\Vert \mathbf{u}^1 \Vert_{\textbf{L}^4(\Omega)} + \Vert \mathbf{u}^2 \Vert_{\textbf{L}^4(\Omega)}) \Vert \tilde{\mathbf{u}}\Vert_{\textbf{L}^4(\Omega)}\Vert \nabla \tilde{\mathbf{u}}\Vert_{\textbf{L}^2(\Omega)}.
\label{estterm1}
\end{align}
Here we have $N=\Vert \nabla \mathbf{u}^2\Vert_{\textbf{L}^3(\Omega)}$.

On the other hand, owing to (H3), we can find  $C_3=C_3(\tau_2)>0$ such that
\begin{align}
\nonumber\Vert \tau(\cdot,\varepsilon(\mathbf{u}^2))\Vert_{\textbf{L}^3(\Omega)} &\leq C(\tau_1) \left(1+\Vert \varepsilon(\mathbf{u}^2)\Vert^{p-1}_{\textbf{L}^{3p-3}(\Omega)}\right)\\&\leq C(\tau_1)\left(1+\Vert \varepsilon(\mathbf{u}^2)\Vert^{p-1}_{\textbf{L}^{3}(\Omega)}\right)\leq {C_3(\tau_1)\left(1+N^{p-1}\right)},
\label{estterm11}
\end{align}
for any $p\in(1,2].$ Moreover, we know that there exists a positive constant $C_4$ depending on $\tau_1$, $\Vert\mathbf{f}\Vert_{\textbf{W}^{-1,p^\prime}(\Omega)}$ and $\nu_1$ such that
\begin{equation}
\Vert \varepsilon(\mathbf{u}^k)\Vert_{\textbf{L}^{p}(\Omega)} + \sqrt{\sigma} \Vert \nabla\mathbf{u}^k\Vert_{\textbf{L}^{2}(\Omega)}
\leq C_4, \quad k=1,2. \label{boundvel}
\end{equation}
Indeed, we can take $\mathbf{v}=\mathbf{u}^k$ in \eqref{fluidapps} written for $(\mathbf{u}^k,\vartheta^k)$  and using (H1), (H3), (H5), and Young's inequality.
Thus, using Poincar\'{e}'s and Korn's inequalities and the continuous embedding $\textbf{V}^2_{div} \hookrightarrow \textbf{L}^4(\Omega)^3$, we get
\begin{align}
\nonumber I_1 &\leq \Vert \nu^\prime\Vert_{L^\infty(\mathbb{R})}\Vert \tilde \vartheta \Vert_{L^6(\Omega)}\\&\nonumber \times\left({C_6}\left(1+{N^{p-1}}\right)\Vert \varepsilon(\tilde{\mathbf{u}})\Vert_{\textbf{L}^2(\Omega)}
+ \sigma N\Vert \nabla \tilde{\mathbf{u}}\Vert_{\textbf{L}^2(\Omega)}\right)
\\&+ \frac{C_7}{\sqrt{\sigma}}\Vert \nabla \tilde{\mathbf{u}}\Vert^2_{\textbf{L}^2(\Omega)},
\label{estterm12}
\end{align}
where $C_6$ and $C_7$ depend on $C_4$, on the Poincar\'{e} constant and on the constant of the embedding quoted above (therefore they depend on $\Omega$).

Regarding $I_2$, we have
\begin{equation}
I_2 \leq \Vert \tilde{\mathbf{u}}\Vert_{\textbf{L}^3_{div}} (\Vert\vartheta^2\Vert_{L^6(\Omega)} + \Vert \Theta_0\Vert_{L^6(\Omega)}) \Vert \nabla \tilde\vartheta \Vert. \label{estterm2}
\end{equation}
Taking $\phi=\vartheta^2$ in \eqref{tempapps} written for $(\mathbf{u}^2,\vartheta^2)$ we get
\begin{equation}
\kappa \Vert\nabla \vartheta^2 \Vert^2 \leq \Vert g\Vert_{V^\prime}\Vert \vartheta^2 \Vert_{H^1(\Omega)}
+ \Vert \mathbf{u}^2\Vert_{\textbf{L}^3_{div}} \Vert \Theta_0\Vert_{L^6(\Omega)}\Vert\nabla\vartheta^2\Vert.
\label{f}
\end{equation}
Thus, using Poincar\'{e}'s, we can find $C_8=C_8(\Omega)>0$ such that
\begin{equation*}
\kappa \Vert\nabla \vartheta^2 \Vert^2 \leq \left(C_8\Vert g\Vert_{V^\prime}
+ \Vert \mathbf{u}^2\Vert_{\textbf{L}^3_{div}}\Vert \Theta_0\Vert_{L^6(\Omega)}\right)\Vert\nabla\vartheta^2\Vert
\end{equation*}
and Young's inequality give
\begin{equation*}
\frac{\kappa}{2} \Vert\nabla \vartheta^2 \Vert^2 \leq \frac{1}{2\kappa} \left(C_7\Vert g\Vert_{V^\prime}
+ \Vert \mathbf{u}^2\Vert_{\textbf{L}^3_{div}}\Vert \Theta_0\Vert_{L^6(\Omega)}\right)^2.
\end{equation*}
Thus, recalling \eqref{boundvel}, the continuous embedding $\textbf{V}^2_{div}\hookrightarrow \textbf{L}^{3}(\Omega)$, Korn's and Poincar\'{e}'s inequalities,
we can find a constant $C_9>0$ depending on $C_4$, $\Omega$, $\tau_1$, $\tau_2$, and $\Vert\mathbf{f}\Vert_{\textbf{W}^{-1,p^\prime}(\Omega)}$ such that
\begin{equation*}
\Vert \vartheta^2 \Vert_{V_0} \leq \frac{C_9}{\kappa}\left(\Vert g\Vert_{V^\prime} + \frac{\Vert \Theta_0\Vert_{V}}{\sqrt{\sigma}}\right).
\end{equation*}
Hence, from \eqref{estterm2} we infer
\begin{equation}
I_2 \leq \frac{C_{10}}{\kappa}\left(\Vert g\Vert_{V^\prime} + \frac{\Vert \Theta_0\Vert_{V}}{\sqrt{\sigma}}\right) \Vert\tilde{\mathbf{u}}\Vert_{\textbf{L}^3_{div}} \Vert \nabla \tilde\vartheta \Vert, \label{estterm21}
\end{equation}
where $C_{10}>0$ has the same dependencies as $C_9$ does. Then, using Young's inequality, from \eqref{tempdiff3} we obtain
\begin{equation}
\frac{\kappa}{2} \Vert \nabla\tilde\vartheta \Vert^2_{H^3}\leq \frac{C^2_{10}}{2\kappa^2} \left(\Vert g\Vert_{V^\prime} + \frac{\Vert \Theta_0\Vert_{V}}{\sqrt{\sigma}}\right)^2\Vert\tilde{\mathbf{u}}\Vert^2_{\textbf{L}^3_{div}}. \label{tempdiff4}
\end{equation}
Therefore, on account of \eqref{estterm21}, from \eqref{estterm12} Korn's and Poincar\'e's inequalities and Sobolev embeddings, we deduce
\begin{align}
I_1 &\leq C_{11} \frac{\Vert \nu^\prime\Vert_{L^\infty(\mathbb{R})}}{\kappa^{3/2}} \left(\Vert g\Vert_{V^\prime} + \frac{\Vert \Theta_0\Vert_{V}}{\sqrt{\sigma}}\right) \left(1+{N^{p-1}}+
\sigma N\right)\Vert \nabla \tilde{\mathbf{u}}\Vert^2_{\textbf{L}^2(\Omega)}\nonumber\\
&+ \frac{C_7}{\sqrt{\sigma}}\Vert \nabla \tilde{\mathbf{u}}\Vert^2_{\textbf{L}^2(\Omega)}.
\label{estterm13}
\end{align}
Here $C_{11}>0$ has the same dependencies as the other constants.

Combining \eqref{fluiddiff3} with \eqref{estterm13} we eventually get
\begin{align*}
&\left[\sigma\nu_1 - C_{11}\frac{\Vert \nu^\prime\Vert_{L^\infty(\mathbb{R})}}{\kappa^{3/2}}\left(\Vert g\Vert_{V^\prime} + \frac{\Vert \Theta_0\Vert_{V}}{\sqrt{\sigma}}\right)\left(1+{N^{p-1}}+\sigma N\right)
+ \frac{C_7}{\sqrt{\sigma}}\right] \\&\times\Vert \nabla\tilde{\mathbf{u}}\Vert^2_{\textbf{L}^2(\Omega)}\leq 0,
\end{align*}
which eventually yields uniqueness provided that \eqref{uniqcond} holds with $M_1=C_{11}$ and $M_2=C_7$.

Let us now consider the proof of Theorem \ref{uniquebis}. This corresponds to the case when $\tau$ is defined as in \eqref{Carreau2}, $\sigma=0$ and there is no convective term $(\textbf{u}\cdot \nabla)\textbf{u}$. First notice that the weak formulation for $\tilde{\mathbf{u}}$ is again \eqref{fluiddiff}-\eqref{tempdiff} (without the convective term in \eqref{fluiddiff}), since the test functions $\phi\in V_0$ are sufficiently regular also in the case of $d=3$ (differently from what observed in Remark \ref{test}). Indeed, this is possible {by the regularity assumed for $\mathbf{u}$ and thanks to the embedding} $\textbf{W}^{1,p}(\Omega)\hookrightarrow\textbf{L}^3(\Omega)$ for $p\geq 3/2$. 

In this situation we cannot exploit \eqref{fluiddiff3}, but we exploit a more refined estimate of the first term in \eqref{fluiddiff2}. We can suppose either $d=2$ or $d=3$. In particular, {assuming  $\tau$ to be given by \eqref{Carreau2}}, by means of Lemma \ref{lemma:Carreau}, exploiting H\"older's inequality with a negative exponent, we get 
\begin{align*}
&\int_\Omega \nu(\vartheta^1 + \Theta_0) \left[(\tau_{ij}(x,\varepsilon(\mathbf{u}^1))-  \tau_{ij}(x,\varepsilon(\mathbf{u}^2))\right]\varepsilon_{ij}(\tilde{\mathbf{u}})dx\\&\geq C\nu_1\left(1+\Vert \textbf{u}^1\Vert_{\textbf{W}^{1,p}(\Omega)}+\Vert \textbf{u}^2\Vert_{\textbf{W}^{1,p}(\Omega)}\right)^{p-2}\Vert\tilde{\textbf{u}}\Vert_{\textbf{W}^{1,p}(\Omega)}^2\geq C_{12}\Vert\tilde{\textbf{u}}\Vert_{\textbf{W}^{1,p}(\Omega)}^2,
\end{align*}
thanks to the fact that, by standard inequalities, for $k=1,2$,
\begin{eqnarray}\label{stab:result_vel}
 \rVert   \mathbf{u}^k\rVert_{\textbf{W}^{1,p}(\Omega)} \leq {C}\|\mathbf{f}\|_{\textbf{L}^{p'}(\Omega)}.
\end{eqnarray}
{Note that $C_{12}>0$ also depends on $\nu_1$.}
We can estimate the remaining terms in $I_1$ by means of H\"older's inequality, Sobolev embeddings and (H3),
\begin{align}
I_1&=-\int_\Omega \left[\nu(\vartheta^1 + \Theta_0) - \nu(\vartheta^2 + \Theta_0)\right]\tau_{ij}(x,\varepsilon(\mathbf{u}^2)) \varepsilon_{ij}(\tilde{\mathbf{u}})dx\nonumber\\&\leq\nonumber C\Vert\tau(\cdot,\varepsilon(\mathbf{u}^2))\Vert_{\textbf{L}^\alpha(\Omega)}\Vert \varepsilon(\tilde{\textbf{u}})\Vert_{\textbf{L}^p(\Omega)}\Vert\vartheta \Vert_{\textbf{L}^\beta(\Omega)}\\&\leq \frac{C_{12}}{2}\Vert\tilde{\textbf{u}}\Vert_{\textbf{L}^p(\Omega)^d}^2+C_{13}\left(1+\Vert\textbf{u}^2\Vert^{2p-2}_{\textbf{W}^{1,\alpha(p-1)}(\Omega)}\right)\Vert\nabla\vartheta\Vert^2,
\label{term1bis}
\end{align}
where, for $d=2$, $\alpha=\frac{q}{p-1}$ and $\beta=(1-1/\alpha-1/p)^{-1}$, with $q$ defined in the assumptions,  whereas, for $d=3$, we set $\alpha=\frac{6p}{5p-6}$ and $\beta=6$.
{Concerning the temperature, we have (see \eqref{term2})}
\begin{equation}
I_2 \leq \Vert \tilde{\mathbf{u}}\Vert_{\textbf{L}^\gamma(\Omega)} (\Vert\vartheta^2\Vert_{L^t(\Omega)} + \Vert \Theta_0\Vert_{L^t(\Omega)}) \Vert \nabla \tilde\vartheta \Vert, \label{estterm2a}
\end{equation}
where $\gamma=\frac{2p}{2-p}$, $t=\frac{2\gamma}{\gamma-2}$ for $d=2$ and $\gamma=3$, $t=6$ for $d=3$. Moreover, taking again $\phi=\vartheta^2$ in \eqref{tempapps} written for $(\mathbf{u}^2,\vartheta^2)$ we get
\begin{equation}
\kappa \Vert\nabla \vartheta^2 \Vert^2 \leq \Vert g\Vert_{V^\prime}\Vert \vartheta^2 \Vert_{H^1(\Omega)}
+ \Vert \mathbf{u}^2\Vert_{\textbf{L}(\Omega)^\gamma }\Vert \Theta_0\Vert_{L^t(\Omega)}\Vert\nabla\vartheta^2\Vert,
\label{f2}
\end{equation}
for the same exponents $\gamma,t$. Indeed, with this choice of $\gamma$ we have, for $d=2$, the embedding $\textbf{W}^{1,p}(\Omega)\hookrightarrow \textbf{L}^\gamma(\Omega)$ for any $p\in(1,2)$ and thus also $\Vert\textbf{u}^2\Vert_{\textbf{L}^\gamma(\Omega)}\leq C$ as it is needed in this estimate. Notice that also for $d=3$, when $p\geq3/2$ we get $\textbf{W}^{1,p}(\Omega)\hookrightarrow \textbf{L}^3(\Omega)$ and thus  $\Vert\textbf{u}^2\Vert_{\textbf{L}^3(\Omega)}\leq C$. Exploiting then the embedding $H^1(\Omega)\hookrightarrow L^t(\Omega)$ for any $t\in [2,\infty)$ if $d=2$, for $t=6$ if $d=3$, we get, similarly to \eqref{tempdiff4}, 
\begin{align}
\frac{\kappa}{2} \Vert \nabla\tilde\vartheta \Vert^2&\leq \frac{C^2_{10}}{2\kappa^2} \left(\Vert g\Vert_{V^\prime} + {\Vert \Theta_0\Vert_{V}}\right)^2\Vert\tilde{\mathbf{u}}\Vert^2_{\textbf{L}^\gamma(\Omega)}, \label{tempdiff5}
\end{align}
Therefore, for $p\in(1,2)$ in the case $d=2$ and for $p\in[3/2,2)$ when $d=3$, we get 
\begin{align}
\frac{\kappa}{2} \Vert \nabla\tilde\vartheta \Vert^2&\leq \frac{C_{14}}{\kappa^2} \left(\Vert g\Vert_{V^\prime} + {\Vert \Theta_0\Vert_{V}}\right)^2\Vert\tilde{\textbf{u}}\Vert_{\textbf{W}^{1,p}(\Omega)}^2. \label{tempdiff6}
\end{align}
Putting together these results we obtain in the end 
\begin{align*}
 \Vert\tilde{\textbf{u}}\Vert_{\textbf{W}^{1,p}(\Omega)}^2\left(\frac{C_{12}}{2}-\frac{C_{13}C_{14}}{\kappa^3} \left(\Vert g\Vert_{V^\prime} + {\Vert \Theta_0\Vert_{V}}\right)^2\left(1+\Vert\textbf{u}^2\Vert^{2p-2}_{\textbf{W}^{1,\alpha(p-1)}(\Omega)}\right)\right)\leq 0,
\end{align*}
from which we deduce uniqueness provided that \eqref{cond} holds with $M_3=C_{12}$ and $M_4=C_{13}C_{14}$. The proof is finished.
\subsubsection{Proof of Theorem \ref{unique2}}
First of all we prove the following
\begin{lemma}
Let the assumptions of Theorem \ref{unique} hold, together with \eqref{g}. Then a solution $\vartheta_\sigma$ given by Theorem \ref{existreg}, for any $\sigma>0$, belongs to $L^\infty(\Omega)$.
\end{lemma}
\begin{proof}
We recall that by Theorem \ref{existreg} we have $\textbf{u}_\sigma\in \textbf{V}_{div}^2$ and $\vartheta_\sigma\in V_0$. We concentrate on \eqref{Sy10}. By a classical elliptic regularity result, we get 
$$
\Vert\vartheta_\sigma\Vert_{H^2(\Omega)}\leq \frac{1}{\kappa}\left(\Vert g\Vert+\Vert\textbf{u}_\sigma\cdot \nabla\vartheta_\sigma\Vert+\Vert\textbf{u}_\sigma\cdot \nabla\Theta_0\Vert\right),
$$
but, by H\"{o}lder and Sobolev-Gagliardo-Nirenberg's inequalities, using also Sobolev embeddings, we have
$$
\Vert\textbf{u}_\sigma\cdot \nabla\vartheta_\sigma\Vert\leq \Vert\textbf{u}_\sigma\Vert_{\textbf{L}^6(\Omega)} \Vert \nabla\vartheta_\sigma\Vert_{\textbf{L}^3(\Omega)}\leq C\Vert\nabla\textbf{u}_\sigma\Vert\Vert \nabla\vartheta_\sigma\Vert^{1/2}\Vert \vartheta_\sigma\Vert^{1/2}_{H^2(\Omega)}.
$$
Analogously, recalling the properties of the lift operator, i.e., $\Vert\Theta_0\Vert_{H^2(\Omega)}\leq C\Vert \theta\Vert_{H^{3/2}(\partial\Omega)}$,
we find
\begin{align*}
   & \Vert\textbf{u}_\sigma\cdot \nabla\Theta_0\Vert\leq \Vert\textbf{u}_\sigma\Vert_{\textbf{L}^6(\Omega)} \Vert \nabla\Theta_0\Vert_{\textbf{L}^3(\Omega)}\\&\leq C\Vert\nabla\textbf{u}_\sigma\Vert\Vert \Theta_0\Vert_{H^2(\Omega)}\leq C\Vert\nabla\textbf{u}_\sigma\Vert\Vert \theta\Vert_{H^{3/2}(\partial\Omega)}.
\end{align*}
Therefore, by Young's inequality, we immediately deduce 
$$
\Vert\vartheta_\sigma\Vert_{H^2(\Omega)}\leq \frac{C}{\kappa^2}\Vert\nabla\textbf{u}_\sigma\Vert^2\Vert\nabla\vartheta_\sigma\Vert+\frac{1}{2}\Vert \vartheta_\sigma\Vert_{H^2(\Omega)}+\frac{C}{\kappa}\Vert\nabla\textbf{u}_\sigma\Vert\Vert \theta\Vert_{H^{3/2}(\partial\Omega)}+\frac{1}{\kappa}\Vert g\Vert, 
$$
i.e.,
\begin{equation}
\Vert\vartheta_\sigma\Vert_{H^2(\Omega)}\leq \frac{C}{\kappa^2}\Vert\nabla\textbf{u}_\sigma\Vert^2\Vert\nabla\vartheta_\sigma\Vert+\frac{C}{\kappa}\Vert\nabla\textbf{u}_\sigma\Vert\Vert \theta\Vert_{H^{3/2}(\partial\Omega)}, 
\label{theta}
\end{equation}
meaning that $\vartheta_\sigma\in H^2(\Omega)$ (and thus also $\vartheta_\sigma+\Theta_0\in H^2(\Omega)$). In particular, this implies that $\vartheta_\sigma$ is bounded in $L^\infty(\Omega)$.
\end{proof}

{We now prove Theorem \ref{unique2}.
Suppose that $(\mathbf{u}^k,\vartheta^k) \in V^r_{div} \times V_0$, $k=1,2$, satisfy the following system
	\begin{align}
	&\int_\Omega \left[\nu(\vartheta + \Theta_0)\left( \tau_{ij}(x,\varepsilon(\mathbf{u}))  
    +\sigma \vert \varepsilon(\mathbf{u}))\vert^{r-2} \varepsilon_{ij}(\mathbf{u})\right)\right]\varepsilon_{ij}(\mathbf{v})dx \nonumber\\
	& -\int_\Omega u_j u_i \partial_{x_j} v_i dx = \langle \mathbf{f}, \mathbf{v} \rangle \qquad \forall\, \mathbf{v}\in \textbf{V}^r_{div} \label{fluidapps2}\\
	&\int_\Omega \kappa \nabla \vartheta \cdot \nabla \phi + \int_\Omega u_j\vartheta \partial_{x_j}\phi = \langle g,\phi \rangle
	+ \int_\Omega u_j\Theta_0 \partial_{x_j}\phi \qquad \forall\,\phi\in V_0. \label{tempapps2}
	\end{align}
	Suppose also, as before, $r=2$ and $d=3$. Then, as in the proof of Theorem \ref{unique}, set $\tilde{\mathbf{u}}=\mathbf{u}^1 -\mathbf{u}^2$ and $\tilde{\vartheta}= \vartheta^1 - \vartheta^2$ and observe that $(\tilde{\mathbf{u}},\tilde\vartheta)$
	satisfies
	\begin{align}
	&\int_\Omega \nu(\vartheta^1 + \Theta_0) \left[(\tau_{ij}(x,\varepsilon(\mathbf{u}^1))-  \tau_{ij}(x,\varepsilon(\mathbf{u}^2))\right]\varepsilon_{ij}(\mathbf{v})dx +\sigma \int_\Omega \nu(\vartheta^1 + \Theta_0)\varepsilon_{ij}(\tilde{\mathbf{u}})\varepsilon_{ij}(\mathbf{v})dx \nonumber\\
	&=-\int_\Omega \left[\nu(\vartheta^1 + \Theta_0) - \nu(\vartheta^2 + \Theta_0)\right](\tau_{ij}(x,\varepsilon(\mathbf{u}^2)) \varepsilon_{ij}(\mathbf{v})dx\nonumber\\
	&-\sigma\int_\Omega \left[\nu(\vartheta^1 + \Theta_0) - \nu(\vartheta^2 + \Theta_0)\right] \varepsilon_{ij}({\mathbf{u}^2})\varepsilon_{ij}(\mathbf{v})dx\nonumber\\
	&+ \int_\Omega (\tilde{u}_j u^1_i + u^2_j \tilde{u}_i) \partial_{x_j} v_i dx
	\qquad \forall\, \mathbf{v}\in \textbf{V}^r_{div} \label{fluiddifff}\\
	&\int_\Omega \kappa \nabla \tilde\vartheta \cdot \nabla \phi + \int_\Omega u^1_j\tilde\vartheta \partial_{x_j}\phi =
	-\int_\Omega \tilde u_j\vartheta^2 \partial_{x_j}\phi + \int_\Omega \tilde u_j\Theta_0 \partial_{x_j}\phi \qquad \forall\,\phi\in V_0. \label{tempdifff}
	\end{align}
	Take $\mathbf{v}=\textbf{A}^{-1}\tilde{\mathbf{u}}$ in \eqref{fluiddiff}, and $\phi=\tilde\vartheta$ in \eqref{tempdiff}. Recalling that $\mathbf{u}_1$ is divergence free, this gives
	\begin{align}
	&\int_\Omega \nu(\vartheta^1 + \Theta_0) \left[(\tau_{ij}(x,\varepsilon(\mathbf{u}^1))-  \tau_{ij}(x,\varepsilon(\mathbf{u}^2))\right]\varepsilon_{ij}(\textbf{A}^{-1}\tilde{\mathbf{u}})dx\nonumber
	\\&+\sigma \int_\Omega \nu(\vartheta^1 + \Theta_0)\varepsilon_{ij}(\tilde{\mathbf{u}})\varepsilon_{ij}(\textbf{A}^{-1}\tilde{\mathbf{u}})dx \nonumber\\
	&=-\int_\Omega \left[\nu(\vartheta^1 + \Theta_0) - \nu(\vartheta^2 + \Theta_0)\right](\tau_{ij}(x,\varepsilon(\mathbf{u}^2)) \varepsilon_{ij}(\textbf{A}^{-1}\tilde{\mathbf{u}})dx\nonumber\\
	&-\sigma\int_\Omega \left[\nu(\vartheta^1 + \Theta_0) - \nu(\vartheta^2 + \Theta_0)\right] \varepsilon_{ij}({\mathbf{u}^2})\varepsilon_{ij}(\textbf{A}^{-1}\tilde{\mathbf{u}})dx\nonumber\\
	&+ \int_\Omega (\tilde{u}_j u^1_i + u^2_j \tilde{u}_i) \partial_{x_j} (\textbf{A}^{-1}\tilde{\textbf{u}})_i dx
	\qquad \forall\, \mathbf{v}\in \textbf{V}^r_{div}, \label{fluiddifff2}\\
	&\kappa\int_\Omega \vert \nabla \tilde\vartheta\vert^2 =
	-\int_\Omega \tilde u_j\vartheta^2 \partial_{x_j}\tilde\vartheta + \int_\Omega \tilde u_j\Theta_0 \partial_{x_j}\tilde\vartheta \qquad \forall\,\phi\in V_0. \label{tempdifff2}
	\end{align}
	From now on, for the sake of simplicity, we will also use the notation $D\textbf{u}$ to indicate the symmetrized strain tensor. We observe that 
	\begin{align*}
	&\sigma \int_\Omega \nu(\vartheta^1 + \Theta_0)\varepsilon_{ij}(\tilde{\mathbf{u}})\varepsilon_{ij}(\textbf{A}^{-1}\tilde{\mathbf{u}})dx=\sigma (\nu(\vartheta^1 + \Theta_0)D\tilde{\textbf{u}},D\textbf{A}^{-1}\tilde{\mathbf{u}})\\&=\sigma (\nu(\vartheta^1 + \Theta_0)\nabla\tilde{\textbf{u}},D\textbf{A}^{-1}\tilde{\mathbf{u}})
	\end{align*}
	and
	\begin{align*}
	&\sigma (\nu(\vartheta^1 + \Theta_0)\nabla\tilde{\textbf{u}},D\textbf{A}^{-1}\tilde{\mathbf{u}})=-\sigma(\text{div}(\nu(\vartheta^1 + \Theta_0)D\textbf{A}^{-1}\tilde{\mathbf{u}}),\tilde{\mathbf{u}})\\&=-\frac{\sigma}{2}(\nu(\vartheta^1 + \Theta_0)\Delta\textbf{A}^{-1}\tilde{\mathbf{u}},\tilde{\mathbf{u}})-\sigma(D\textbf{A}^{-1}\tilde{\mathbf{u}}\nabla\Theta_1\nu^\prime(\vartheta^1 + \Theta_0),\tilde{\mathbf{u}}).
	\end{align*}
	Now by the classical properties of $\textbf{A}$ (see, e.g., Ref. \refcite{Giorginitemam}, Appendix B, we know that, for a suitable $\tilde{p}\in H^1(\Omega)\cap L^2_{(0)}(\Omega)$, we have
	$$
	-\Delta\textbf{A}^{-1}\tilde{\mathbf{u}}+\nabla \tilde{p}=\tilde{\mathbf{u}}\qquad\text{a.e. in }\Omega,
	$$
	and
	$$
	\Vert\textbf{A}^{-1}\tilde{\mathbf{u}}\Vert_{\textbf{H}^2(\Omega)}\leq C\Vert\tilde{\mathbf{u}}\Vert,\qquad \Vert \tilde{p}\Vert_{H^1(\Omega)}\leq C\Vert\tilde{\mathbf{u}}\Vert.
	$$
	Thus, by (H1), we get
	\begin{align*}
	&-\frac{1}{2}(\nu(\vartheta^1 + \Theta_0)\Delta\textbf{A}^{-1}\tilde{\mathbf{u}},\tilde{\mathbf{u}})=\frac{1}{2}(\nu(\vartheta^1 + \Theta_0)\tilde{\mathbf{u}},\tilde{\mathbf{u}})-\frac{1}{2}(\nu(\vartheta^1 + \Theta_0)\nabla \tilde{p},\tilde{\mathbf{u}})\\&\geq \frac{\nu_1}{2}\Vert\tilde{\mathbf{u}}\Vert^2-\frac{1}{2}(\nu(\vartheta^1 + \Theta_0)\nabla \tilde{p},\tilde{\mathbf{u}}).
	\end{align*}
	Using also (H4), we deduce the inequalities
	\begin{align}
	&\frac{\sigma\nu_1}{2} \int_\Omega \vert \tilde{\mathbf{u}}\vert^2 dx \leq I_1, \label{fluiddifff3}\\
	&\kappa\int_\Omega \vert \nabla \tilde\vartheta\vert^2 \leq I_2, \label{tempdifff3}
	\end{align}
	where
	\begin{align}
	I_1&=-\int_\Omega \left[\nu(\vartheta^1 + \Theta_0) - \nu(\vartheta^2 + \Theta_0)\right](\tau_{ij}(x,\varepsilon(\mathbf{u}^2)) \varepsilon_{ij}(\textbf{A}^{-1}\tilde{\mathbf{u}})dx\nonumber\\
	&-\sigma\int_\Omega \left[\nu(\vartheta^1 + \Theta_0) - \nu(\vartheta^2 + \Theta_0)\right] \varepsilon_{ij}({\mathbf{u}^2})\varepsilon_{ij}(\textbf{A}^{-1}\tilde{\mathbf{u}})dx\nonumber\\
	&+ \int_\Omega (\tilde{u}_j u^1_i + u^2_j \tilde{u}_i) \partial_{x_j} (\textbf{A}^{-1}\tilde{\textbf{u}})_i dx\nonumber\\&+\frac{\sigma}{2}(\nu(\vartheta^1 + \Theta_0)\nabla p,\tilde{\mathbf{u}})+\sigma(D\textbf{A}^{-1}\tilde{\mathbf{u}}\nabla\Theta_1\nu^\prime(\vartheta^1 + \Theta_0),\tilde{\mathbf{u}}),
	\label{terms1}\\
	&I_2 = -\int_\Omega \tilde u_j(\vartheta^2 - \Theta_0)\partial_{x_j}\tilde\vartheta. \label{terms2}
	\end{align}
	Observe now that, recalling (H1), thanks to H\"{o}lder's and Korn's inequalities, the properties of $\textbf{A}$ and the Sobolev embedding $\textbf{V}^2_{div} \hookrightarrow \textbf{L}^q(\Omega)$, $q\in[2,6]$, we have
	\begin{align}
	&I_1 \leq \Vert \nu^\prime\Vert_{L^\infty(\mathbb{R})}\Vert \tilde \vartheta \Vert_{L^6(\Omega)}\nonumber\\&\times\left(\Vert (\tau(\cdot,\varepsilon(\mathbf{u}^2))\Vert_{\textbf{L}^2(\Omega)} \Vert \nabla\textbf{A}^{-1}\tilde{\mathbf{u}}\Vert_{\textbf{L}^3(\Omega)} + \sigma\Vert\nabla\textbf{u}^2\Vert_{\textbf{L}^2(\Omega)}\Vert \nabla \textbf{A}^{-1}\tilde{\mathbf{u}}\Vert_{\textbf{L}^3(\Omega)}\right)\nonumber\\
	&+ (\Vert \mathbf{u}^1 \Vert_{\textbf{L}^4(\Omega)} + \Vert \mathbf{u}^2 \Vert_{\textbf{L}^4(\Omega)}) \Vert \tilde{\mathbf{u}}\Vert_{\textbf{L}^2(\Omega)}\Vert \nabla\textbf{A}^{-1}\tilde{\mathbf{u}}\Vert_{\textbf{L}^4(\Omega)}\nonumber\\&+\sigma C\left(\Vert\nu\Vert_{L^\infty(\R)}+\Vert \nu^\prime\Vert_{L^\infty(\mathbb{R})}\Vert\nabla\Theta_1\Vert_{\textbf{L}^4(\Omega)}\right)\Vert\tilde{\mathbf{u}}\Vert^2\nonumber\\&\leq C_{15}\Vert \nu^\prime\Vert_{L^\infty(\mathbb{R})}\Vert \tilde \vartheta \Vert_{L^6(\Omega)}\left(\Vert \tau(\cdot,\varepsilon(\mathbf{u}^2))\Vert_{\textbf{L}^3(\Omega)} \Vert \tilde{\mathbf{u}}\Vert+ \sigma\Vert\nabla\textbf{u}^2\Vert_{\textbf{L}^2(\Omega)}\Vert \tilde{\mathbf{u}}\Vert\right)\nonumber\\
	&+ (\Vert \mathbf{u}^1 \Vert_{\textbf{L}^4(\Omega)} + \Vert \mathbf{u}^2 \Vert_{\textbf{L}^4(\Omega)})\Vert \tilde{\mathbf{u}}\Vert^2\nonumber\\&+\sigma C_{16}\left(\Vert\nu\Vert_{L^\infty(\R)}+\Vert \nu^\prime\Vert_{L^\infty(\mathbb{R})}\Vert\nabla\Theta_1\Vert_{\textbf{L}^4(\Omega)}\right)\Vert\tilde{\mathbf{u}}\Vert^2\nonumber.
	\end{align}
	On the other hand, owing to (H3), since $p\in(1,2)$, we can find  $C=C(\tau_2)>0$ such that
	\begin{align}
	&\Vert \tau(\cdot,\varepsilon(\mathbf{u}^2))\Vert_{\textbf{L}^2(\Omega)} \leq C \left(1+\Vert \varepsilon(\mathbf{u}^2)\Vert^{p-1}_{\textbf{L}^{2p-2}(\Omega)}\right)\leq C\left(1+\Vert \varepsilon(\mathbf{u}^2)\Vert^{p-1}_{\textbf{L}^{2}(\Omega)}\right).
	\label{esttermm11}
	\end{align}
	Also, we know that there exists a positive constant $C_{17}$ depending on $\tau_1$, $\nu_1$ and $\Vert\mathbf{f}\Vert_{\textbf{W}^{-1,p^\prime}(\Omega)}$ such that
	\begin{equation}
	\Vert \varepsilon(\mathbf{u}^k)\Vert_{\textbf{L}^{p}(\Omega)} + \sqrt{\sigma} \Vert \nabla\mathbf{u}^k\Vert_{\textbf{L}^{2}(\Omega)}
	\leq C_{17}, \quad k=1,2. \label{boundvel2}
	\end{equation}
	Indeed, we can take $\mathbf{v}=\mathbf{u}^k$ in \eqref{fluidapps} written for $(\mathbf{u}^k,\vartheta^k)$  and using (H1), (H3), (H5), and Young's inequality we get \eqref{boundvel2}. This means that the right hand side of \eqref{esttermm11} is bounded (note that this is possible because $\sigma>0$).
	Thus, using Poincar\'{e}'s inequality and the continuous embedding $\textbf{V}^2_{div} \hookrightarrow \textbf{L}^4(\Omega)$, we get
	\begin{align}
	I_1 &\leq C_{15}\Vert \nu^\prime\Vert_{L^\infty(\mathbb{R})}\Vert \tilde \vartheta \Vert_{L^6(\Omega)}\left(\Vert \tau(\cdot,\varepsilon(\mathbf{u}^2))\Vert_{\textbf{L}^3(\Omega)} \Vert \tilde{\mathbf{u}}\Vert+ \sigma\Vert\nabla\textbf{u}^2\Vert_{\textbf{L}^2(\Omega)}\Vert \tilde{\mathbf{u}}\Vert\right)\nonumber\\
	&+ (\Vert \mathbf{u}^1 \Vert_{\textbf{L}^4(\Omega)} + \Vert \mathbf{u}^2 \Vert_{\textbf{L}^4(\Omega)})\Vert \tilde{\mathbf{u}}\Vert^2\nonumber\\&+\sigma C_{16}\left(\Vert\nu\Vert_{L^\infty(\R)}+\Vert \nu^\prime\Vert_{L^\infty(\mathbb{R})}\Vert\nabla\Theta_1\Vert_{\textbf{L}^4(\Omega)}\right)\Vert\tilde{\mathbf{u}}\Vert^2\nonumber\\&\leq C_{18}\Vert \nu^\prime\Vert_{L^\infty(\mathbb{R})}\Vert \tilde \vartheta \Vert_{L^6(\Omega)}\left( \left(1+\frac{1}{(\sqrt{\sigma})^{p-1}}\right)\Vert \tilde{\mathbf{u}}\Vert+ \frac{\sigma}{\sqrt{\sigma}}\Vert \tilde{\mathbf{u}}\Vert\right)\nonumber\\&
	+ \frac{C_{19}}{\sqrt{\sigma}}\Vert \tilde{\mathbf{u}}\Vert^2+\sigma C_{16}\left(\Vert\nu\Vert_{L^\infty(\R)}+\Vert \nu^\prime\Vert_{L^\infty(\mathbb{R})}\Vert\nabla\Theta_1\Vert_{\textbf{L}^4(\Omega)}\right)\Vert\tilde{\mathbf{u}}\Vert^2,
	\label{esttermm12}
	\end{align}
where $C_{18},\ C_{19}$ also depend on the Poincar\'{e} constant and on the constants of the embeddings mentioned above (in particular, they depend on $\Omega$). Clearly they also depend on $\Vert\textbf{f}\Vert_{\textbf{W}^{-1,p^\prime}(\Omega)}$ and $\nu_1$.
	
	Regarding $I_2$, we have
	\begin{equation}
	I_2 \leq \Vert \tilde{\mathbf{u}}\Vert (\Vert\vartheta^2\Vert_{L^\infty(\Omega)} + \Vert \Theta_0\Vert_{L^\infty(\Omega)}) \Vert \nabla \tilde\vartheta \Vert. \label{esttermm2}
	\end{equation}
	Considering \eqref{theta} written for $(\mathbf{u}^2,\vartheta^2)$ we get
	\begin{equation}
\Vert\vartheta^2\Vert_{H^2(\Omega)}\leq \frac{C_{20}}{\kappa^2}\Vert\nabla\textbf{u}^2\Vert^2\Vert\nabla\vartheta^2\Vert+\frac{C_{21}}{\kappa}\Vert\nabla\textbf{u}^2\Vert\Vert \theta\Vert_{H^{3/2}(\partial\Omega)}+\frac{1}{\kappa}\Vert g\Vert.
\label{theta2}
\end{equation}
	Thus, by \eqref{boundvel2}, we can find $C_{22}>0$ such that
	\begin{equation}
\Vert\vartheta^2\Vert_{H^2(\Omega)}\leq C_{22}\left(\frac{1}{\kappa^2{\sigma}}\Vert\nabla\vartheta^2\Vert+\frac{1}{\kappa\sqrt{\sigma}}\Vert \theta\Vert_{H^{3/2}(\partial\Omega)}+\frac{1}{\kappa}\Vert g\Vert\right).
\label{theta3}
\end{equation}
Moreover, taking $\phi=\vartheta^2$ in \eqref{tempapps} written for $(\mathbf{u}^2,\vartheta^2)$ we get
	\begin{equation*}
	\kappa \Vert\nabla \vartheta^2 \Vert^2 \leq \Vert g\Vert_{V^\prime}\Vert \vartheta^2 \Vert_{H^1(\Omega)}
	+ \Vert \mathbf{u}^2\Vert_{\textbf{L}^3(\Omega)} \Vert \Theta_0\Vert_{L^6(\Omega)}\Vert\nabla\vartheta^2\Vert
	\end{equation*}
	and using Poincar\'{e}'s inequality, we find
	\begin{equation*}
	\kappa \Vert\nabla \vartheta^2 \Vert^2 \leq \left(C\Vert g\Vert_{V^\prime}
	+ \Vert \mathbf{u}^2\Vert_{\textbf{L}^3(\Omega)}\Vert \Theta_0\Vert_{L^6(\Omega)}\right)\Vert\nabla\vartheta^2\Vert.
	\end{equation*}
	Then Young's inequality gives
	\begin{equation*}
	\frac{\kappa}{2} \Vert\nabla \vartheta^2 \Vert^2 \leq \frac{1}{2\kappa} \left(C\Vert g\Vert_{V^\prime}
	+ \Vert \mathbf{u}^2\Vert_{\textbf{\textbf{L}}^3(\Omega)}\Vert \Theta_0\Vert_{L^6(\Omega)}\right)^2
	\end{equation*}
	so that, recalling \eqref{boundvel2},
	we can find a constant $C_{23}>0$ depending on $\Omega$, $\tau_1$, $\tau_2$, $\nu_1$ and $\Vert\mathbf{f}\Vert_{\textbf{W}^{-1,p^\prime}(\Omega)}$ such that
	\begin{equation*}
	\Vert \vartheta^2 \Vert_{V_0} \leq \frac{C_{23}}{\kappa}\left(\Vert g\Vert_{V^\prime} + \frac{\Vert \Theta_0\Vert_{V}}{\sqrt{\sigma}}\right).
	\end{equation*}
	Hence, from \eqref{theta3} we infer
	\begin{equation}
\Vert\vartheta^2\Vert_{H^2(\Omega)}\leq C_{24}\left(\frac{1}{\kappa^3\sigma}\left(\Vert g\Vert_{V^\prime} + \frac{\Vert \Theta_0\Vert_{V}}{\sqrt{\sigma}}\right)+\frac{1}{\kappa\sqrt{\sigma}}\Vert \theta\Vert_{H^{3/2}(\partial\Omega)}+\frac{1}{\kappa}\Vert g\Vert\right).
\label{theta4}
\end{equation}
	 Analogously, being $H^2(\Omega)\hookrightarrow L^\infty(\Omega)$, we deduce
		\begin{equation}
	\Vert\vartheta^2\Vert_{L^\infty(\Omega)}\leq C_{25}\left(\frac{1}{\kappa^3\sigma}\left(\Vert g\Vert_{V^\prime} + \frac{\Vert \Theta_0\Vert_{V}}{\sqrt{\sigma}}\right)+\frac{1}{\kappa\sqrt{\sigma}}\Vert \theta\Vert_{H^{3/2}(\partial\Omega)}+\frac{1}{\kappa}\Vert g\Vert\right).
	\label{theta5}
	\end{equation}

	Therefore, on account of \eqref{theta4}, recalling that $\Vert\Theta_0\Vert_{L^\infty(\Omega)}\leq C\Vert \theta\Vert_{H^{3/2}(\partial\Omega)}$, we deduce
	\begin{equation}
I_2 \leq C_{26}\Vert \tilde{\mathbf{u}}\Vert \left(\frac{1}{\kappa^3\sigma}\left(\Vert g\Vert_{V^\prime} + \frac{\Vert \Theta_0\Vert_{V}}{\sqrt{\sigma}}\right)+\frac{1}{\kappa\sqrt{\sigma}}\Vert \theta\Vert_{H^{3/2}(\partial\Omega)}+\frac{1}{\kappa}\Vert g\Vert\right)\Vert \nabla \tilde\vartheta \Vert. \label{esttermm3}
\end{equation}
 Then, on account of \eqref{tempdifff3}, using Young's inequality we deduce
\begin{align}
 \Vert \nabla \tilde\vartheta \Vert^2\leq \frac{C_{27}}{\kappa^2}\Vert \tilde{\mathbf{u}}\Vert^2 \left(\frac{1}{\kappa^3\sigma}\left(\Vert g\Vert_{V^\prime} + \frac{\Vert \Theta_0\Vert_{V}}{\sqrt{\sigma}}\right)+\frac{1}{\kappa\sqrt{\sigma}}\Vert \theta\Vert_{H^{3/2}(\partial\Omega)}+\frac{1}{\kappa}\Vert g\Vert\right)^2.
 \label{I2}
\end{align}
	Combining \eqref{estterm12} with \eqref{I2}, recalling that $\Theta^1=\vartheta^1+\Theta_0$, thanks to Young's and Poincaré's inequalities (note that \eqref{theta4} is still valid for $\vartheta^1$) and Sobolev embeddings, we get
		\begin{align}
	&I_1 \leq C_{28}\Vert \nu^\prime\Vert_{L^\infty(\mathbb{R})}^2\Vert \tilde \vartheta \Vert_{L^6(\Omega)}^2\left( 1+\frac{1}{(\sqrt{\sigma})^{p-1}}+ \frac{\sigma}{\sqrt{\sigma}}\right)^2\nonumber\\&\nonumber
	+ C_{29}\left(1+\frac{1}{\sqrt{\sigma}}+\sigma \left(\Vert\nu\Vert_{L^\infty(\R)}\nonumber+\Vert \nu^\prime\Vert_{L^\infty(\mathbb{R})}\left(\Vert\vartheta_1\Vert_{H^2(\Omega)}+\Vert \theta\Vert_{H^{3/2}(\partial\Omega)}\right)\right)\right)\Vert\tilde{\mathbf{u}}\Vert^2\\&\leq   C_{28}\Vert \nu^\prime\Vert_{L^\infty(\mathbb{R})}^2\Vert \tilde \vartheta \Vert_{L^6(\Omega)}^2\left( 1+\frac{1}{(\sqrt{\sigma})^{p-1}}+ \sqrt{\sigma}\right)^2\nonumber\nonumber
	+ C_{30}\left(1+\frac{1}{\sqrt{\sigma}}+\sigma \left(\Vert\nu\Vert_{L^\infty(\R)}\right.\right.\\&\left.\left.+\Vert \nu^\prime\Vert_{L^\infty(\mathbb{R})}\left(\left(\frac{1}{\kappa^3\sigma}\left(\Vert g\Vert_{V^\prime} + \frac{\Vert \Theta_0\Vert_{V}}{\sqrt{\sigma}}\right)+\frac{1}{\kappa\sqrt{\sigma}}\Vert \theta\Vert_{H^{3/2}(\partial\Omega)}\right.\right.\right.\right.\nonumber\\&\nonumber\left.\left.\left.\left.+\frac{1}{\kappa}\Vert g\Vert\right)+\Vert \theta\Vert_{H^{3/2}(\partial\Omega)}\right)\right)\right)\Vert\tilde{\mathbf{u}}\Vert^2\\&\leq  \frac{C_{31}}{\kappa^2}\Vert \nu^\prime\Vert_{L^\infty(\mathbb{R})}^2\left( 1+\frac{1}{(\sqrt{\sigma})^{p-1}}+ \sqrt{\sigma}\right)^2\nonumber\\&\nonumber\times \left(\frac{1}{\kappa^3\sigma}\left(\Vert g\Vert_{V^\prime} + \frac{\Vert \Theta_0\Vert_{V}}{\sqrt{\sigma}}\right)+\frac{1}{\kappa\sqrt{\sigma}}\Vert \theta\Vert_{H^{3/2}(\partial\Omega)}+\frac{1}{\kappa}\Vert g\Vert\right)^2\Vert \tilde{\mathbf{u}}\Vert^2\\&\nonumber
	+ C_{30}\left(1+\frac{1}{\sqrt{\sigma}}+\sigma \left(\Vert\nu\Vert_{L^\infty(\R)}\right.\right.\\&\left.\left.+\Vert \nu^\prime\Vert_{L^\infty(\mathbb{R})}\left(\left(\frac{1}{\kappa^3\sigma}\left(\Vert g\Vert_{V^\prime} + \frac{\Vert \Theta_0\Vert_{V}}{\sqrt{\sigma}}\right)+\frac{1}{\kappa\sqrt{\sigma}}\Vert \theta\Vert_{H^{3/2}(\partial\Omega)}\right.\right.\right.\right.\nonumber\\&\left.\left.\left.\left.+\frac{1}{\kappa}\Vert g\Vert\right)+\Vert \theta\Vert_{H^{3/2}(\partial\Omega)}\right)\right)\right)\Vert\tilde{\mathbf{u}}\Vert^2.
	\label{esttermm12bis}
	\end{align}
	In conclusion, recalling that $\Vert\Theta_0\Vert_V\leq \tilde{C}\Vert \theta\Vert_{H^{1/2}(\partial\Omega)}$ and owing to  \eqref{fluiddifff3}, the condition
			\begin{align}
	&\left[-\frac{\sigma\nu_1}{2}+\frac{C_{31}}{\kappa^2}\Vert \nu^\prime\Vert_{L^\infty(\mathbb{R})}^2\left( 1+\frac{1}{(\sqrt{\sigma})^{p-1}}+ \sqrt{\sigma}\right)^2\nonumber\right.\\&\left.\nonumber\times \left(\frac{1}{\kappa^3\sigma}\left(\Vert g\Vert_{V^\prime} + \frac{\tilde{C}\Vert \theta\Vert_{H^{1/2}(\partial\Omega)}}{\sqrt{\sigma}}\right)+\frac{1}{\kappa\sqrt{\sigma}}\Vert \theta\Vert_{H^{3/2}(\partial\Omega)}+\frac{1}{\kappa}\Vert g\Vert\right)^2\right.\\&\left.\nonumber
	+ C_{30}\left(1+\frac{1}{\sqrt{\sigma}}+\sigma \left(\Vert\nu\Vert_{L^\infty(\R)}\right.\right.\right.\\&\left.\left.\left.+\Vert \nu^\prime\Vert_{L^\infty(\mathbb{R})}\left(\left(\frac{1}{\kappa^3\sigma}\left(\Vert g\Vert_{V^\prime} + \frac{\tilde{C}\Vert \theta\Vert_{H^{1/2}(\partial\Omega)}}{\sqrt{\sigma}}\right)+\frac{1}{\kappa\sqrt{\sigma}}\Vert \theta\Vert_{H^{3/2}(\partial\Omega)}\right.\right.\right.\right.\nonumber\right.\\&\left.\left.\left.\left.\left.+\frac{1}{\kappa}\Vert g\Vert\right)+\Vert \theta\Vert_{H^{3/2}(\partial\Omega)}\right)\right)\right)\right]\Vert\tilde{\mathbf{u}}\Vert^2\geq0,
	\label{esttermm12tris}
	\end{align}
	yields uniqueness provided that \eqref{uniqcond2} holds with $M_1=C_{31}$, $M_2=\tilde{C}$ and $M_3=C_{30}$.
	}
	
	
	\section{Finite element approximation: the Carreau law case }
	\label{numerical}
	\subsection{Preliminaries}
 In this section, we consider the finite element approximation of a specific instance of the boundary value problem \eqref{Sy04}-\eqref{Sy07}, again under the same assumptions \textbf{(H1)}-\textbf{(H7)} namely we drop the convective term $(\mathbf{u}\cdot \nabla)\mathbf{u}$ and we focus on the relevant case of non-Newtonian fluids governed by the Carreau law, i.e. $\tau(\varepsilon(\mathbf{u}))=\eta(\vert\varepsilon(\mathbf{u})\vert^2)\varepsilon(\mathbf{u})$ with
\begin{eqnarray}
&&\eta(z):=\eta_\infty + (\eta_0-\eta_\infty)(1+\lambda z)^{(p-2)/2}\label{carreau_law}\\
&&\eta_0>\eta_\infty\geq 0, \lambda>0 \text{~and~} p\in (1,2).\label{carreau_law_1}
\end{eqnarray}

For the sequel, it is instrumental to recall the following result (cf. Lemmas 3.1 and 3.2 of Ref. \refcite{Barrett-Liu:1993})
where $\vert \cdot \vert$ denotes the euclidean matrix norm, i.e. for $\mathbf{K}\in \R^{n\times n}$ real matrix, $\vert \textbf{K}\vert^2=\sum_{i,j=1}^n (K_{i,j})^2$. 
\begin{lemma}\label{lemma:Carreau}
Let $\eta$ obey the Carreau law \eqref{carreau_law} with $p\in (1,2)$. Then there exists positive constants $C_i$, $i=1,\ldots,3$ such that for all symmetric matrices $\mathbf{K},\mathbf{L}\in \mathbb{R}^{2\time 2}$
there holds
\begin{eqnarray}
 \vert \eta(\vert \mathbf{K}\vert^2)\mathbf{K}-\eta(\vert \mathbf{L}\vert^2)\mathbf{L}\vert &\leq& C_1 \vert \mathbf{K} - \mathbf{L}\vert \qquad \text{if~} \eta_\infty \not=0\nonumber\\
 \vert \eta(\vert \mathbf{K}\vert^2)\mathbf{K}-\eta(\vert \mathbf{L}\vert^2)\mathbf{L}\vert &\leq& C_3 \vert \mathbf{K} - \mathbf{L}\vert^{p-1} 
 \qquad \text{if~} \eta_\infty =0.\nonumber
 \end{eqnarray} 
Moreover it holds
\begin{align*}
&\sum_{i,j=1}^2 
\left(\eta(\vert \mathbf{K}\vert^2)K_{i,j}-\eta(\vert \mathbf{L}\vert^2)L_{i,j}\right) (K_{i,j}-L_{i,j}) \\&\geq 
\{\eta_\infty + C_2[1 + \vert \mathbf{K}\vert + \vert \mathbf{L}\vert]^{p-2}\}\vert \mathbf{K} - \mathbf{L}\vert^2.&\nonumber
\end{align*}
\end{lemma}
We first recall the variational formulation of the continuous problem.  Let us choose: 
$$V = H^1(\Omega),~~ {V}_0 = H^1_{0}(\Omega),~~\mathbf{V} = \textbf{W}^{1,s}(\Omega),~~\mathbf{V}_0 = \textbf{W}^{1,s}_{0}(\Omega),$$
$$ Q = L^{s'}(\Omega),~~ Q_* = L^{s'}_0(\Omega) = \left\{q \in Q | \int_{\Omega} q = 0\right\},$$
where  $1/s'+1/s=1$  with $s=p$ if $\eta_\infty=0$ and $s=2$ otherwise.
We set $\|\cdot\|_{\mathbf{V}}:=\vert \cdot\vert_{\textbf{W}^{1,s}(\Omega)}$ and $\|\cdot\|_{Q}=\|\cdot\|_{Q_*}:=\| \cdot\|_{L^{s'}(\Omega)}$. 
Moreover, let us introduce the following forms
\begin{equation*}
\begin{aligned}
    &a_1(\vartheta, \mathbf{u} ;\mathbf{v}) = \int_{\Omega}2 \nu(\vartheta) \eta(\vert {\varepsilon}(\mathbf{u})\vert^2)\varepsilon(\mathbf{u}) : \varepsilon(\mathbf{v}), 
    \qquad a_2(\vartheta,\varrho) = \int_{\Omega} \kappa \nabla \vartheta \cdot \nabla \varrho,  
    \\&b(\mathbf{v},p) = -\int_{\Omega}p \nabla \cdot \mathbf{v},
    \qquad c(\mathbf{u},\vartheta, \varrho) = \int_{\Omega} (\mathbf{u}\cdot \nabla)\vartheta \varrho, 
    \\&f(\mathbf{v}) = \int_{\Omega} \mathbf{f} \cdot \mathbf{v},
    \qquad g(\varrho) = \int_{\Omega} g \varrho. 
    \end{aligned}
\end{equation*}
Since the velocity space is not divergence free, we rewrite the variational formulation of a weak solution given in Definition \ref{weaksol} as follows: let $p\in(1,2)$ if $d=2$ and $p\in(3/2,2)$ for $d=3$ { when $\eta_\infty=0$ and $p\in(1,2)$ for $d=3$ when $\eta_\infty=0$}. Find $(\mathbf{u},\pi,\vartheta) \in \mathbf{V}_0 \times Q_* \times V_0$ such that
\begin{align}  \label{eq::liftweak}
  &a_1(\vartheta+\Theta_0,\mathbf{u}; \mathbf{v}) + b(\mathbf{v},\pi) =f(\mathbf{v}), \qquad \\&
  b(\mathbf{u},q)=0,\label{eq::middlelift1} \qquad \\&
  a_2(\vartheta,\varrho) + c(\mathbf{u};\vartheta,\varrho)=g(\rho) - a_2(\Theta_0,\varrho)-c(\mathbf{u},\Theta_0,\varrho),
  \label{eq::variastokesnllift}
\end{align}
$\forall (\mathbf{v}, q ,\varrho) \in \mathbf{V}_0\times Q_* \times V_0$.
\begin{oss}
The weak formulation is well posed since, also in the case $\eta_\infty=0$, differently from the general case of Remark \ref{test}, we are considering a range of $p$'s for $d=2,3$ ensuring the embedding $\textbf{V}_0\hookrightarrow \textbf{L}^3(\Omega)$. Therefore $\varrho\in V_0$ is allowed as a test function in \eqref{eq::variastokesnllift}. 
\end{oss}
We collect well known results that will be employed in the following. 
First, the bilinear form $b(\mathbf{v},q)$ is continuous, that is there exists $M>0$ such that 
\begin{equation}
b(\mathbf{v},q) \leq M \lVert \mathbf{v} \rVert_{\mathbf{V}} \lVert q \rVert_{Q_*} \quad \forall (\mathbf{v},q) \in \mathbf{V} \times {Q_*}
\label{eq::bcont}
\end{equation}
and it satisfies the following compatibility condition\cite{brezziboffi,girault}: there exists $c=c(\Omega)$ such that
\begin{equation}
    c\lVert q \rVert_{Q_*} \leq \sup_{\mathbf{v}\in \mathbf{V}_0} \dfrac{b(\mathbf{v},q)}{\lVert \mathbf{v}\rVert_{\mathbf{V}}} \quad \forall q \in Q_*.
    \label{eq::bcompati}
\end{equation}
Second, the trilinear form $c(\mathbf{u},\theta, \varrho)$ satisfies the antisymmetry property when $\nabla \cdot \mathbf{u} = 0$. Indeed
\begin{equation}\label{eq:antisimmetry_of_c}
  \begin{aligned}
  c(\mathbf{u},\theta, \varrho) &= - c(\mathbf{u},\varrho, \theta),\\
  c(\mathbf{u},\theta, \theta) &= 0. 
  \end{aligned}
\end{equation}
%
%

Moreover, for all $(\mathbf{u},\theta,\varrho) \in \mathbf{V}_0 \times V_0 \times V_0$, using the generalized H\"older inequality and the Sobolev embedding, 
\begin{equation}
 c(\mathbf{u},\theta, \varrho) \leq  \lVert \mathbf{u} \rVert_{\textbf{W}^{1,p}(\Omega)} \lVert \nabla \theta \rVert\lVert \nabla \varrho \rVert
\end{equation}
where $p\in(1,2)$ for $d=2$ and $p\in[3/2,2)$ for $d=3$.
Indeed, for $d=2$, for any $\iota>2$,
\begin{equation*}
 c(\mathbf{u},\theta, \varrho) \leq \lVert  \mathbf{u} \rVert_{\textbf{L}^\iota(\Omega)} \lVert \varrho \rVert_{L^{\frac{2\iota}{\iota-2}}(\Omega)} \lVert \nabla \theta \rVert \leq \lVert \mathbf{u} \rVert_{\textbf{W}^{1,p}(\Omega)} \lVert \nabla \theta \rVert\lVert \nabla \varrho \rVert,
\end{equation*}
due to the embedding $H^1(\Omega)\hookrightarrow L^\iota(\Omega), \forall \iota\in[2,\infty)$. Therefore, being $\iota$ arbitrary, it is enough to use $\iota=\frac{2p}{2-p}>2$ to obtain the embedding: for any $p\in(1,2)$ the inequality is true. For $d=3$, we have at most
\begin{equation*}
 c(\mathbf{u},\theta, \varrho) \leq \lVert  \mathbf{u} \rVert_{\textbf{L}^3(\Omega)} \lVert \varrho \rVert_{{L}^{6}(\Omega)} \lVert \nabla \theta \rVert \leq \lVert \mathbf{u} \rVert_{\textbf{W}^{1,p}(\Omega)} \lVert \nabla \theta \rVert\lVert \nabla \varrho \rVert,
\end{equation*}
due to the embedding $H^1(\Omega)\hookrightarrow L^6(\Omega)$. Therefore, we need $p\in[3/2,2)$.  
{Notice that this is coherent with the theoretical results (see, e.g., Remark \ref{validity}).}
\subsection{Definition of the discrete problem}
Let $\Omega\subset \mathbb{R}^d$, $d=2,3$ be a polygonal/polyedral domain. We introduce (see, e.g,. Refs. \refcite{BrennerScott,Ciarlet:2002}) a regular family of triangulations $\mathcal{T}_h$ of  $\Omega$ where $h$ denotes the maximum of the diameters of the elements  $K\in\mathcal{T}_h$. 


Let $\mathbf{V}_h\subset \mathbf{V}$, $ {Q}_h\subset Q_*$ and $V_h\subset V$ be finite dimensional spaces and we set
    $\mathbf{V}_{0,h} = \mathbf{V}_{h} \cap \mathbf{V}_{0} $,  ${V}_{0,h} = {V}_{h} \cap V_0$.
We assume that the following approximability properties hold:
\begin{itemize}
\item[(P1)] $\inf_{s_h \in V_h} \|\nabla (\theta - s_h)\| \to 0\quad$ as $h\to 0\quad \forall \theta \in {V}$, 
\item[(P2)] $\inf_{\mathbf{v}_h \in \mathbf{V}_h} \|\nabla (\mathbf{v} - \mathbf{v}_h)\|_{\mathbf{V}} \to 0\quad$ as $h\to 0\quad \forall \mathbf{v} \in {\mathbf{V}}$, 
\item[(P3)] $\inf_{q_h \in Q_h} \|\pi - q_h\|_{Q_*} \to 0\quad$ as $h\to 0\quad \forall \pi \in {Q_*}$.
\end{itemize}
Moreover, we assume (see, e.g., Refs. \refcite{brezziboffi,girault} for examples in the context of finite elements)  that the discrete spaces  $\mathbf{V}_h, Q_h$ are chosen so that there exists $\beta>0$ independent of $h$ such that
\begin{equation}
    \beta \lVert q_h \rVert_{Q_*} \leq \sup_{\mathbf{v}_h \in \mathbf{V}_{0,h}} \dfrac{b(\mathbf{v}_h,q_h)}{\lVert \mathbf{v}_h\rVert_{\mathbf{V}}} \quad \forall q_h \in Q_h.
    \label{eq::bcompatidis}
\end{equation}
We introduce the following discrete form  
\begin{equation}
    c_h(\mathbf{u}_h,\vartheta_h,\varrho_h) = \dfrac{1}{2}\left(c(\mathbf{u}_h,\vartheta_h,\varrho_h) - c(\mathbf{u}_h,\varrho_h,\vartheta_h)\right)
    \label{eq::antisim}
\end{equation}
to recover the antisymmetry property of its continuous counterpart (cf. \eqref{eq:antisimmetry_of_c}).
Similarly to the continuous case, it holds
\begin{equation}\label{c_h:cont}
 c_h(\mathbf{u}_h,\theta_h, \varrho_h)   \leq  \lVert \mathbf{u}_h \rVert_{\textbf{W}^{1,{q}}(\Omega)} \lVert \nabla \theta_h \rVert\lVert \nabla \varrho_h \rVert
\end{equation}
where $q\in(1,2)$ for $d=2$ and $q\in[3/2,2)$ for $d=3$.
{Let us denote, with a slight abuse of notation, by $\Theta_0\in V_h$  the $V$-projection on $V_h$ of the continuous lifting defined in \eqref{eq:Theta0}, in such a way that the control $\Vert \Theta_0\Vert_{V}\leq C\Vert \theta\Vert_{H^{1/ 2}(\Omega)}$ is still valid}.  Then  the finite dimensional approximation of problem \eqref{eq::liftweak}-\eqref{eq::variastokesnllift}, reads: find $(\mathbf{u}_h,\pi_h,\vartheta_h) \in \mathbf{V}_{0,h} \times Q_h\times V_{0,h}$ such that
\begin{align}\label{eq::liftweakdisc2}
&a_1(\vartheta_h+\Theta_0,\mathbf{u}_h;\mathbf{v}_h) + b(\mathbf{v}_h,\pi_h) = \mathbf{f}(\mathbf{v}_h) \\
&b(\mathbf{u}_h,q_h) = 0 \label{eq::middlelift2}\\
&a_2(\vartheta_h,\varrho_h) + c_h(\vartheta_h,\varrho_h;\mathbf{u}_h)=g(\varrho_h) - a_2(\Theta_0,\varrho_h)-c_h(\mathbf{u}_h,\Theta_0,\varrho_h)
\label{eq::variadisclift}
\end{align}
for all $(\mathbf{v}_h,q_h,\varrho_h) \in \mathbf{V}_{0,h}\times Q_h \times V_{0,h}$.
Let us first remark that in the case of $\eta_\infty=0$, which is the most interesting case, the existence of a solution to \eqref{eq::liftweakdisc2}-\eqref{eq::variadisclift} follows by adapting the proof of Theorem \ref{exist}, employing the pair $\textbf{V}_h^{div}\times V_h$, $\textbf{V}_h^{div}=\{\textbf{v}_h\in  \mathbf{V}_{0,h}: b(\mathbf{v}_h,q_h)=0 ~\forall q_h\in Q_h \}$, so that \eqref{eq::liftweakdisc2}-\eqref{eq::variadisclift} is equivalent to the weak formulation of Definition \ref{weaksol} on the spaces $\textbf{V}_h^{div}\times V_h$. Conditional uniqueness is then a consequence of Theorem \ref{uniquebis}, exploiting again the pair $\textbf{V}_h^{div}\times V_h$. Analogously, in the case $\eta_{\infty}>0$, the existence of a solution to \eqref{eq::liftweakdisc2}-\eqref{eq::variadisclift} follows by adapting the proof of Theorem \ref{existreg} with the pair $\textbf{V}_h^{div}\times V_h$, whereas conditional uniqueness is a consequence of Theorem \ref{unique} with $\sigma=\eta_\infty$ and $r=2$, adapted to the space $\textbf{V}_h^{div}\times V_h$. Notice that Theorem \ref{unique2} cannot be applied, since the solution $\mathbf{u}_h$ does not enjoy zero divergence and moreover we cannot ensure $\textbf{A}^{-1}\textbf{u}_h\in \textbf{V}_h$.

\subsection{A priori error analysis: main results}
Let us now state and prove the main results of this section, namely a priori error estimates for the discrete solution of  \eqref{eq::liftweakdisc2}-\eqref{eq::variadisclift}. First we state the most relevant result, which corresponds to the case when $\eta_\infty=0$. 
\begin{theorem} \label{theoesti:1}
Let $\eta_\infty=0$, $p\in (1, 2)$ for $d=2$ and  $p\in (3/2, 2)$ for $d=3$. Let $(\mathbf{u},\pi,\vartheta)$ be a solution of \eqref{eq::liftweak}-\eqref{eq::variastokesnllift}.
Let $(\mathbf{u}_h,\pi_h,\vartheta_h)$ be a solution of \eqref{eq::liftweakdisc2}-\eqref{eq::variadisclift} where we assume  $\mathbf{u}_h \in \textbf{W}^{1,(p-1)p_2}(\Omega)$, with $p_2=q/(p-1)$ with $q>p$ when $d=2$ and $p_2=6p/(5p-6)$ when $d=3$. In addition, we suppose that the following smallness condition on $\mathbf{u}_h$ holds
\begin{enumerate}
\item[(A1)] $ 
\kappa-2 \left[ 
\frac{\tilde{C}_2}{\kappa} \|g\| + \tilde{C}_3\left(1+C_f\right)  \| \theta\|_{H^{1/2}(\Gamma)}  \right]D_6 \lVert \varepsilon(\mathbf{u}_h) \rVert^{p-1}_{\textbf{L}^{(p-1)p_2}(\Omega)} >0,
$
\end{enumerate}
with $C_f={\tilde{C}_1}(1+\|\mathbf{f}\|_{\textbf{W}^{1,p'}(\Omega)})^{1/(p-1)}$, {$\tilde{C}_1>0$ depending on $\nu_1$, while $\tilde{C}_2,\tilde{C}_3, D_6>0$} depend on the data of the problem.
Then there exists $h_0>0$ such that, for any $h\leq h_0$,  the following inequalities hold  
\begin{eqnarray}
&&\lVert \nabla({\vartheta}-{\vartheta}_h) \rVert \lesssim \min_{(\mathbf{v}_h,q_h,s_h) \in \mathbf{V}_h\times Q_h \times V_h} \left(\lVert\nabla( \vartheta - s_h) \rVert + \lVert \varepsilon(\mathbf{u}-\mathbf{v}_h) \rVert^{p-1}_{\textbf{L}^p(\Omega)}\right. \nonumber\\
&& \left.  \hspace{6cm} + ~\lVert \pi -q_h \rVert_{L^{p'}(\Omega)}\right)\label{eq::testi:etazero}\\
&&\lVert \varepsilon(\mathbf{u}-\mathbf{u}_h) \rVert_{\textbf{L}^p(\Omega)} \lesssim  \min_{(\mathbf{v}_h,q_h,s_h) \in \mathbf{V}_h\times Q_h \times V_h} \left(\lVert \varepsilon(\mathbf{u}-\mathbf{v}_h) \rVert_{\textbf{L}^p(\Omega)}^{p-1} + \lVert \pi -q_h \rVert_{\textbf{L}^{p'}(\Omega)} \right.\nonumber\\
&&\left . \hspace{6cm} + ~\lVert \nabla(\vartheta - s_h) \rVert\right),  \label{eq::uesti:etazero}\\ 
&&\lVert \pi -\pi_h \rVert_{{L}^{p'}(\Omega)} \lesssim
 \min_{(\mathbf{v}_h,q_h,s_h) \in \mathbf{V}_h\times Q_h \times V_h} \left(
 \lVert \pi -q_h \rVert^{p-1}_{L^{p'}(\Omega)} + \lVert \varepsilon(\mathbf{u}-\mathbf{v}_h) \rVert^{(p-1)^2}_{\textbf{L}^p(\Omega)} \right. \nonumber\\
 && \left. \hspace{6cm}  +~ \lVert \nabla(\vartheta - \vartheta_h) \rVert^{p-1}\right), \label{eq::pesti:etazero} 
\end{eqnarray}
where $\textbf{u}$ is the solution to the continuous problem, and the hidden constants are independent of $h$.
\end{theorem}
\begin{oss}
Notice that in force of Theorems \ref{existregol} and \ref{existregol2} we have that any solution $\textbf{u}$ to the corresponding continuous problem is sufficiently regular to perform the error estimates, if we assume the additional hypotheses \eqref{extra} (for $d=3$) and \eqref{extraB1} (for $d=2$). 
Therefore, we expect that also $\textbf{u}_h$, shares an analogous regularity (in case of a sufficiently regular domain $\Omega$), which is required in Theorem \ref{theoesti:1}.
\end{oss}
In conclusion, we state another result for the simpler case $\eta_\infty>0$.
\begin{theorem} \label{theoesti:2}
Let $\eta_\infty>0$, {$p\in (1, 2)$} and $(\mathbf{u},\pi,\vartheta)$ be a solution of \eqref{eq::liftweak}-\eqref{eq::variastokesnllift}.
Let $(\mathbf{u}_h,\pi_h,\vartheta_h)$ be a solution of \eqref{eq::liftweakdisc2}-\eqref{eq::variadisclift} where we assume  {$\mathbf{u}_h \in \textbf{W}^{1,p_2}(\Omega)$, with $p_2>2$ when $d=2$ and $p_2=3$} when $d=3$. In addition, we suppose that the following smallness condition on $\mathbf{u}_h$ holds 
\begin{enumerate}
\item[(B1)] $ 
\kappa-2 \left[ 
\frac{\tilde{C}_2}{\kappa} \|g\| + \tilde{C}_3\left(1+C_f\right)  \| \theta\|_{H^{1/2}(\Gamma)}  \right]D_6 \lVert \varepsilon(\mathbf{u}_h) \rVert_{\textbf{L}^{{p_2}}(\Omega)} >0,
$
\end{enumerate}
with $C_f={ \tilde{C}_1}/(\nu_1 \eta_{\infty})\|\mathbf{f}\|$, { here $\tilde{C}_1>0$ stands for the Korn constant,  while $\tilde{C}_2,\tilde{C}_3, D_6>0$} depend on the data of the problem.
Then there exists $h_0>0$ such that, for any $h\leq h_0$,  the following inequalities hold  
\begin{eqnarray}
&&\lVert \nabla({\vartheta}-{\vartheta}_h) \rVert \lesssim \min_{(\mathbf{v}_h,q_h,s_h) \in \mathbf{V}_h\times Q_h \times V_h} \left(\lVert\nabla( \vartheta - s_h) \rVert + \lVert \varepsilon(\mathbf{u}-\mathbf{v}_h) \rVert\right. \nonumber\\
&& \left.  \hspace{6cm} + ~\lVert \pi -q_h \rVert\right)\label{eq::testi:etazero1}\\
&&\lVert \varepsilon(\mathbf{u}-\mathbf{u}_h) \rVert\lesssim  \min_{(\mathbf{v}_h,q_h,s_h) \in \mathbf{V}_h\times Q_h \times V_h} \left(\lVert \varepsilon(\mathbf{u}-\mathbf{v}_h) \rVert + \lVert \pi -q_h \rVert \right.\nonumber\\
&&\left . \hspace{6cm} + ~\lVert \nabla(\vartheta - s_h) \rVert\right),  \label{eq::uesti:etazero1}\\ 
&&\lVert \pi -\pi_h \rVert \lesssim
 \min_{(\mathbf{v}_h,q_h,s_h) \in \mathbf{V}_h\times Q_h \times V_h} \left(
 \lVert \pi -q_h \rVert + \lVert \varepsilon(\mathbf{u}-\mathbf{v}_h) \rVert \right. \nonumber\\
 && \left. \hspace{6cm}  +~ \lVert \nabla(\vartheta - \vartheta_h) \rVert\right), \label{eq::pesti:etazero1} 
\end{eqnarray}
where $\textbf{u}$ is the solution to the continuous problem, and the hidden constants are independent of $h$.
\end{theorem}
\begin{oss}[{ Orders of convergence}]\label{oss:orders}
{ Assuming the validity of suitable interpolation error estimates holding for the approximation spaces $\mathbf{V}_{h} \times Q_h\times V_{h}$, from Theorems \ref{theoesti:1} and \ref{theoesti:2} it is possible to deduce precise orders of convergence of the discrete solutions towards the continuous ones. For instance, we consider 
the finite element spaces ($r\geq1$)
\begin{eqnarray}
    \mathbf{V}_h &=& \{\mathbf{v}_h \in \textbf{C}(\overline{\Omega}) | \forall K \in \mathcal{T}_h, \ \mathbf{v}_h|_{K} \in \mathbb{P}_{r+1}(K)^2\},\nonumber\\
    {Q}_h &=& \{{q}_h \in C(\overline{\Omega}) | \forall K \in \mathcal{T}_h, \ {q}_h|_{K} \in \mathbb{P}_{r}(K)\},\nonumber\\
    {V}_h &=& \{\varrho_h \in C(\overline{\Omega}) | \forall K \in \mathcal{T}_h, \ \varrho_h|_{K} \in \mathbb{P}_{r+1}(K)\}.\nonumber\end{eqnarray}
   In the case $\eta_\infty=0$ (which is the most relevant one), standard interpolation error estimates (see, e.g., Theorem 4.4.4. in Ref. \refcite{BrennerScott}) together with the regularity assumptions $(\mathbf{u}, \pi, \vartheta)\in \mathbf{W}^{r+2,p}(\Omega) \times {W}^{r+1,p'}(\Omega)\times \in H^{r+2}(\Omega)$
yields

\begin{eqnarray}
&&\lVert \nabla({\vartheta}-{\vartheta}_h) \rVert \lesssim h^{r+1}+h^{(r+1)(p-1)}+h^{r+1}=\mathcal{O}(h^{(r+1)(p-1)}),\nonumber\\
&&\lVert \varepsilon(\mathbf{u}-\mathbf{u}_h) \rVert_{L^p(\Omega)}\lesssim  h^{(r+1)(p-1)}+ h^{r+1}+h^{r+1}=\mathcal{O}(h^{(r+1)(p-1)}),\nonumber\\
&&\lVert \pi -\pi_h \rVert_{L^{p'}(\Omega)} \lesssim h^{(r+1)(p-1)}
+ h^{(r+1)(p-1)^2} + h^{(r+1)(p-1)}=\mathcal{O}(h^{(r+1)(p-1)^2})
.\nonumber
\end{eqnarray}
In the case $\eta_\infty>0$, under the regularity assumptions $$(\mathbf{u}, \pi, \vartheta)\in \mathbf{W}^{r+2,2}(\Omega) \times {W}^{r+1,2}(\Omega)\times \in H^{r+2}(\Omega),$$
we get

$$\lVert \nabla({\vartheta}-{\vartheta}_h) \rVert =\mathcal{O}(h^{r+1}), \quad\lVert \varepsilon(\mathbf{u}-\mathbf{u}_h) \rVert=\mathcal{O}(h^{r+1}),\quad \lVert \pi -\pi_h \rVert =\mathcal{O}(h^{r+1}).
$$

}Notice that in the case $\eta_\infty>0$ we get optimal approximation results, whereas when $\eta_\infty=0$ we deduce error estimates that are coherent with the ones obtained, though in the simpler context of isothermal non-Newtonian fluids, in Ref. \refcite{Botti_et_al:2021}. For further comments on these aspects we refer to the numerical results reported in Section \ref{S:numres}.
\end{oss}
\subsection{Proof of Theorems \ref{theoesti:1} and \ref{theoesti:2}}
For the sake of brevity, we consider both the cases $\eta_\infty>0$ and $\eta_\infty=0$. In particular, as before, we set $s=p$ if $\eta_\infty=0$ and $s=2$ if $\eta_\infty>0$. We preliminarily collect some instrumental results that will be employed during the proof. Let us first recall that in view of Assumption 
$\mathbf{(H1)}$ we have 
that $\nu(\Theta)$ is a bounded continuous function defined on $(0,+\infty)$ satisfying the following properties:
\begin{eqnarray}
 &&\nu \in C^1(\mathbb{R}),\\
&& 0<\nu_1\leq\nu(\xi)\leq \nu_2 \quad \text{ for } \xi \in \mathbb{R}^+,\label{eq::etadef:etazero}\\
  &&  |\nu'(\xi)| \leq \nu_3 \quad \text{ for } \xi \in \mathbb{R}^+.\label{eq::etadef2:etazero}
\end{eqnarray}
Moreover, let us  observe that, exploiting Lemma \ref{lemma:Carreau} combined with
 the definition of $a_1(\cdot,\cdot;\cdot)$,  \eqref{eq::etadef:etazero} and  the generalized (also to negative exponents) 
 H\"older's inequality,  we can prove the following inequalities holding $\forall \mathbf{u},\mathbf{v}, \mathbf{w} \in \mathbf{V}_0$:\\
 
\noindent{\bf Case $\mathbf{\eta_\infty>0}$} 
\begin{eqnarray}
&& \vert a_1(\vartheta, \mathbf{u};\mathbf{w})
  -  a_1(\vartheta, \mathbf{v};\mathbf{w})\vert \leq C_1 \nu_2 \lVert \varepsilon(\mathbf{u}-\mathbf{v})\rVert\lVert \varepsilon(\mathbf{w})\rVert, \label{eq::a1prop}	\\ 
 && a_1(\vartheta, \mathbf{u};\mathbf{u}-\mathbf{v})
  -  a_1(\vartheta, \mathbf{v};\mathbf{u}-\mathbf{v}) \geq  \nu_1 \eta_\infty 
  \lVert \varepsilon(\mathbf{u}-  \mathbf{v})\rVert^2. \label{eq::a1prop-bis}
  \end{eqnarray}
  
\noindent{\bf Case $\mathbf{\eta_\infty=0}$}
\begin{eqnarray}
&& \vert a_1(\vartheta, \mathbf{u};\mathbf{w})
  -  a_1(\vartheta, \mathbf{v};\mathbf{w})\vert \leq C_3 \nu_2  \lVert \varepsilon(\mathbf{u-v})\rVert^{p-1}_{\textbf{L}^p} \lVert \varepsilon(\mathbf{w})\rVert_{\textbf{L}^p}, \label{eq::a1prop:etazero}	\\ 
 && a_1(\vartheta, \mathbf{u};\mathbf{u}-\mathbf{v})
  -  a_1(\vartheta, \mathbf{v};\mathbf{u}-\mathbf{v})   \nonumber \\
  && \geq C_2\nu_1 
 \left(\int_{\Omega} 1 + \vert \varepsilon (\mathbf{u})\vert^p+ \vert \varepsilon (\mathbf{v})\vert^p \right)^{\frac{p-2}{p}} \|\varepsilon(\mathbf{u}-\mathbf{v})\|^2_{\textbf{L}^p}.
 \label{eq::a1prop-bis:etazero}
 \end{eqnarray}

Using the definition of  $a_2(\cdot,\cdot)$, it is trivial to see that the following inequalities hold
 \begin{eqnarray}
  &&a_2(\vartheta,\varrho) \leq \kappa  \lVert \nabla \vartheta \rVert \lVert \nabla \varrho \rVert, \quad a_2(\vartheta,\vartheta) \geq  \kappa \lVert \nabla \vartheta \rVert^2~~~~\forall \vartheta, \varrho \in V_0.\label{eq::a2prop}
\end{eqnarray}

Finally, we collect some stability estimates. 
We first observe that the following stability estimate for the discrete velocity holds
\begin{equation}\label{stab:result_vel2}
\|\nabla \mathbf{u}_h\|_{\textbf{L}^s(\Omega)} \leq C_f
\end{equation}
where $C_f={\tilde{C}_1}(1+\|\mathbf{f}\|_{\textbf{W}^{1,p'}(\Omega)})^{1/(p-1)}$ when $\eta_\infty=0$ (${\tilde{C}_1}$ depending on $\nu_1$) and $C_f={\tilde{C}_1}/(\nu_1 \eta_{\infty})\|\mathbf{f}\|$ when $\eta_\infty>0$. Indeed, in the case $\eta_\infty>0$, testing \eqref{eq::liftweakdisc2}-\eqref{eq::middlelift2}
 with $\mathbf{v}_h=\mathbf{u}_h$ and $q_h=\pi_h$ and employing the monotonicity  of $a_1(\cdot,\cdot;\cdot)$  together with the properties of $\nu$  combined with the Korn inequality, we  obtain \eqref{stab:result_vel2} where { $\tilde{C}_1$} is the Korn constant. In the case $\eta_\infty=0$, it is sufficient to proceed as in the proof of Theorem \ref{THM1} (cf. \eqref{Lp}) to get the desired stability estimate. 

Testing \eqref{eq::variadisclift}
with $\varrho_h=\theta_h$ and employing the coercivity and continuity properties of $a_2(\cdot,\cdot)$,  the Poincar\' e inequality, the continuity and antisimmetry of $c_h(\cdot,\cdot,\cdot)$, the above stability result of the discrete velocity  and  the stability of the Dirichlet lifting $\Theta_0$, we get the following stability estimate for the discrete temperature
 \begin{eqnarray}\label{stab:result_temp}
 \rVert  \nabla \vartheta_h \rVert \leq 
 \frac{1}{\kappa}[\tilde{C}_2 \|g\| + \kappa \tilde{C}_3 
 \| \theta\|_{H^{1/2}(\Gamma)} + \tilde{C}_3 C_f\| \theta\|_{H^{1/2}(\Gamma)} ].
 \end{eqnarray}
We remark that an analogous stability estimate holds for the continuous temperature as well.

We are now ready to prove \eqref{eq::testi:etazero}-\eqref{eq::pesti:etazero}.  Let us consider \eqref{eq::liftweak} which, together with \eqref{eq::liftweakdisc2},  yields 
\begin{equation*}
\begin{aligned}
    b(\mathbf{v}_h,\pi-\pi_h) &= a_1(\vartheta_h + \Theta_0,\mathbf{u}_h;\mathbf{v}_h) - a_1(\vartheta + \Theta_0,\mathbf{u};\mathbf{v}_h) \\&= a_1(\vartheta_h + \Theta_0,\mathbf{u}_h;\mathbf{v}_h) -a_1(\vartheta + \Theta_0,\mathbf{u}_h;\mathbf{v}_h)\\&
    \quad+ a_1(\vartheta_h + \Theta_0,\mathbf{u}_h;\mathbf{v}_h)-a_1(\vartheta_h + \Theta_0,\mathbf{u};\mathbf{v}_h)
    \end{aligned}
\end{equation*}
$\forall \mathbf{v}_h \in \mathbf{V}_{h,0}$. By linearity we have
\begin{equation}
\begin{aligned}
   b(\mathbf{v}_h,q_h-\pi_h) &= a_1(\vartheta_h + \Theta_0,\mathbf{u}_h;\mathbf{v}_h) -a_1(\vartheta + \Theta_0,\mathbf{u}_h;\mathbf{v}_h)\\&
    \quad+ a_1(\vartheta_h + \Theta_0,\mathbf{u}_h;\mathbf{v}_h)-a_1(\vartheta_h + \Theta_0,\mathbf{u};\mathbf{v}_h)\\
   &\quad +  b(\mathbf{v}_h,q_h-\pi)
   \end{aligned}
   \label{eq::partialpres}
\end{equation}
$\forall q_h \in Q_h$. Now, considering the compatibility condition \eqref{eq::bcompatidis} together with \eqref{eq::partialpres}, we obtain
\begin{align}
    &\beta \lVert \pi_h - q_h \rVert_{L^{s'}(\Omega)} \leq \sup_{\mathbf{v}_h \in \mathbf{V}_{0,h}} \dfrac{ b(\mathbf{v}_h,q_h-\pi_h)}{\|{v}_h\|_{\mathbf{V}}} \nonumber\\
    &=
    \sup_{\mathbf{v}_h \in \mathbf{V}_{0,h}}\left[\dfrac{(a_1(\vartheta_h + \Theta_0,\mathbf{u}_h;\mathbf{v}_h)-a_1(\vartheta + \Theta_0,\mathbf{u}_h;\mathbf{v}_h)}{\|{v}_h\|_{\mathbf{V}}}\right.\nonumber\\
&\left.+~ \dfrac{a_1(\vartheta_h + \Theta_0,\mathbf{u}_h,\mathbf{v}_h)-a_1(\vartheta_h + \Theta_0,\mathbf{u};\mathbf{v}_h)
   +  b(\mathbf{v}_h,q_h-\pi)}{\|{v}_h\|_{\mathbf{V}}}\right] \nonumber\\
   &\leq  C_s \nu_3\lVert\nabla(\vartheta -\vartheta_h)\rVert\lVert\varepsilon(\mathbf{u}_h)\rVert^{s-1}_{\textbf{L}^{(s-1)p_2}(\Omega)} \\&+ C_s \nu_2  \lVert \varepsilon(\mathbf{u}-\mathbf{u}_h)\rVert^{s-1}_{\textbf{L}^s(\Omega)}  + M \lVert \pi -q_h \rVert_{L^{s'}(\Omega)},\nonumber
 \end{align}
where $C_s=C_1$ for $s=2$, $C_s=C_3$ for $s\not=2$ (cf. \eqref{eq::a1prop} and \eqref{eq::a1prop:etazero}). Moreover, $p_2$ is the one defined in the assumptions of Theorems \ref{theoesti:1} and \ref{theoesti:2}. We note that in the last step  we employed the continuity property \eqref{eq::a1prop} (or \eqref{eq::a1prop:etazero}) for the second term, and the continuity of $b(\cdot,\cdot)$ for the last term. For the first term, we employed the properties of $\nu$ together with Lemma \ref{lemma:Carreau} and the generalized H\"older's inequality to get
\begin{eqnarray}
&&a_1(\vartheta_h + \Theta_0,\mathbf{u}_h;\mathbf{v}_h)-a_1(\vartheta + \Theta_0,\mathbf{u}_h;\mathbf{v}_h) 	\nonumber\\
&&\leq C_s\nu_3 \|\theta-\theta_h\|_{L^\beta(\Omega)} 
\left(\int_\Omega \vert \varepsilon(\mathbf{u}_h)\vert^{(s-1)p_2}\right)^{1/p_2}\|\varepsilon(\mathbf{v}_h)\|_{\textbf{L}^s(\Omega)}\nonumber\\
&&\leq C_s\nu_3 \|\theta-\theta_h\|_{L^\beta(\Omega)} 
\|\varepsilon(\mathbf{u}_h)\|^{s-1}_{\textbf{L}^{(s-1)p_2}(\Omega)}
\|\varepsilon(\mathbf{v}_h)\|_{\textbf{L}^s(\Omega)}\label{aux:estimate:a}
\end{eqnarray}
where $1/\beta+1/p_2+1/s=1$ and $1/\beta=1-1/p_2-1/s$, $p_2=q/(s-1)$ with $q>s$ when $d=2$ and $\beta=6$, $p_2=6s/(5s-6)$ when $d=3$. 

Thus, using the triangle inequality we obtain 
\begin{eqnarray}
    \lVert \pi-\pi_h \rVert_{L^{s'}(\Omega)} &\leq& \lVert \pi-q_h \rVert_{L^{s'}(\Omega)}  + \lVert \pi_h - q_h \rVert_{L^{s'}(\Omega)} \nonumber \\
    &\leq& (1+\frac{M}{\beta}) \lVert \pi-q_h \rVert_{L^{s'}(\Omega)}  + 
    \frac{C_s \nu_2}{\beta}    \lVert \varepsilon(\mathbf{u}-\mathbf{u}_h)\rVert^{s-1}_{\textbf{L}^s(\Omega)} \nonumber\\
    && + \frac{C_s \nu_3}{\beta} \lVert\nabla(\vartheta -\vartheta_h)\rVert \|\varepsilon(\mathbf{u}_h)\|^{s-1}_{\textbf{L}^{(s-1)p_2}(\Omega)}. 
    \label{eq::pressureesti:etazero}
\end{eqnarray}
Next, we estimate the error on the temperature.
Taking the difference between  \eqref{eq::variadisclift} and \eqref{eq::variastokesnllift} and choosing $\varrho = \varrho_h$ yield
\begin{equation*}
    a_2(\vartheta_h-\vartheta,\varrho_h) + c_h(\mathbf{u}_h,\vartheta_h,\varrho_h) - c_h(\mathbf{u},\vartheta,\varrho_h) = c_h(\mathbf{u}-\mathbf{u}_h,\Theta_0,\varrho_h). 
\end{equation*}
Adding and subtracting the two terms $a_2(s_h,\varrho_h)$ and $c_h(\mathbf{u}_h,s_h,\varrho_h)$ we have, for all $s_h \in V_{0,h}$ and for all $\mathbf{v}_h \in \mathbf{V}_h$
\begin{eqnarray}\label{aux:1}
    &&a_2(\vartheta_h - s_h,\varrho_h) + c_h(\mathbf{u}_h,\vartheta_h-s_h,\varrho_h) =  a_2(\vartheta - s_h,\varrho_h) + c_h(\mathbf{u}_h,\vartheta-s_h,\varrho_h) \nonumber\\
    && \qquad + ~c_h(\mathbf{u}-\mathbf{v}_h,\vartheta,\varrho_h) + c_h(\mathbf{v}_h-\mathbf{u}_h,\vartheta,\varrho_h)\nonumber \\
    && \qquad+~ c_h(\mathbf{u}-\mathbf{v}_h,\Theta_0,\varrho_h) + c_h(\mathbf{v}_h-\mathbf{u}_h,\Theta_0,\varrho_h).
\end{eqnarray}
Next, taking $\varrho_h = \vartheta_h-s_h$ in \eqref{aux:1}, noting $c_h(\mathbf{u}_h,\vartheta_h-s_h,\vartheta_h-s_h)=0$, using  the coercivity and continuity properties of $a_2(\cdot,\cdot)$, the continuity property \eqref{c_h:cont}  of $c_h(\cdot,\cdot,\cdot)$, the stability estimates for $\textbf{u}_h$, $\vartheta$ and $\Theta_0$ we have
\begin{eqnarray}\label{eq::tempestif:etazero}
\lVert \nabla(\vartheta_h -s_h)\rVert 
        &\leq& \Lambda_1
  \lVert \nabla(\vartheta -s_h) \rVert  \nonumber\\
&&+      \Lambda_2   
        \left(
        \lVert \varepsilon(\mathbf{u}-\mathbf{v}_h)\rVert_{\textbf{L}^s(\Omega)} 
        +
        \lVert \varepsilon(\mathbf{v}_h-\mathbf{u}_h)\rVert_{\textbf{L}^s(\Omega)} \right) 
\end{eqnarray}
where 
\begin{eqnarray}
&&\Lambda_1= 1+ C_f,\nonumber\\
&&\Lambda_2=  \frac{2}{\kappa}\left(
        \frac{\tilde{C}_2}{\kappa} \|g\| + \left(\tilde{C}_3 
 + \tilde{C}_3 C_f \right)  \| \theta\|_{H^{1/2}(\Gamma)}
        \right).
\end{eqnarray}
{Notice that in the case $s=2$, i.e., $\eta_\infty>0$, \eqref{eq::tempestif:etazero} holds for any $p\in(1,2)$, by \eqref{c_h:cont}, being $\textbf{W}^{1,p}(\Omega)\hookrightarrow \textbf{H}^1(\Omega)$ for any $p\in(1,2]$, whereas in the case $s=p$, so that $\eta_\infty=0$, we are forced to keep $p\in(3/2,2)$, again due to \eqref{c_h:cont}.}
To take advantage of the above inequality, we now estimate $\lVert \varepsilon(\mathbf{u}_h-\mathbf{v}_h)\rVert_{\textbf{L}^s(\Omega)}$.

Consider the momentum equation and take the difference between \eqref{eq::liftweak} and \eqref{eq::liftweakdisc2}. Choosing $\mathbf{u}_h-\mathbf{v}_h$ as test function we obtain
\begin{equation}
    a_1(\vartheta + \Theta_0,\mathbf{u};\mathbf{u}_h-\mathbf{v}_h)-a_1(\vartheta_h + \Theta_0,\mathbf{u}_h;\mathbf{u}_h-\mathbf{v}_h) + b(\mathbf{u}_h-\mathbf{v}_h,\pi-\pi_h) = 0.\label{eq::a2diff}
\end{equation}
Adding and subtracting $a_1(\vartheta + \Theta_0,\mathbf{u}_h;\mathbf{u}_h-\mathbf{v}_h)$ we get
\begin{equation}
\begin{aligned}
&a_1(\vartheta + \Theta_0,\mathbf{u};\mathbf{u}_h-\mathbf{v}_h)
-a_1(\vartheta + \Theta_0,\mathbf{u}_h;\mathbf{u}_h-\mathbf{v}_h)
\\
&
+a_1(\vartheta + \Theta_0,\mathbf{u}_h;\mathbf{u}_h-\mathbf{v}_h)-a_1(\vartheta_h + \Theta_0,\mathbf{u}_h;\mathbf{u}_h-\mathbf{v}_h) \\
&+ b(\mathbf{u}_h-\mathbf{v}_h,\pi-\pi_h) = 0.
\end{aligned}
\end{equation}
Now, recalling 
\begin{equation*}
\begin{aligned}
    b(\mathbf{u}_h-\mathbf{v}_h,\pi-\pi_h) =& \ b(\mathbf{u}_h-\mathbf{u},\pi-q_h)+b(\mathbf{u}_h-\mathbf{u},q_h-\pi_h)\\+&b(\mathbf{u}-\mathbf{v}_h,\pi-\pi_h) = b(\mathbf{u}_h-\mathbf{u},\pi-q_h) + b(\mathbf{u}-\mathbf{v}_h,\pi-\pi_h),
    \end{aligned}
\end{equation*}
adding and subtracting  
$a_1(\vartheta + \Theta_0,\mathbf{v}_h;\mathbf{u}_h-\mathbf{v}_h)$ we get
\begin{equation}
\begin{aligned}
    &a_1(\vartheta + \Theta_0,\mathbf{u}_h;\mathbf{u}_h-\mathbf{v}_h) 
-a_1(\vartheta + \Theta_0,\mathbf{v}_h;\mathbf{u}_h-\mathbf{v}_h)    
    = \\ &~a_1(\vartheta + \Theta_0,\mathbf{u};\mathbf{u}_h-\mathbf{v}_h) 
    -a_1(\vartheta + \Theta_0,\mathbf{v}_h;\mathbf{u}_h-\mathbf{v}_h) 
     \\
    &
    +a_1(\vartheta + \Theta_0,\mathbf{u}_h;\mathbf{u}_h-\mathbf{v}_h)-a_1(\vartheta_h + \Theta_0,\mathbf{u}_h;\mathbf{u}_h-\mathbf{v}_h)\\
    &+b(\mathbf{u}_h-\mathbf{u},\pi-q_h) + b(\mathbf{u}-\mathbf{v}_h,\pi-\pi_h).
    \end{aligned}
    \label{eq::eqerrorvel}
\end{equation}

From \eqref{eq::eqerrorvel} we employ, depending on the value of $\eta_\infty$, \eqref{eq::a1prop}-\eqref{eq::a1prop-bis}
(or \eqref{eq::a1prop:etazero}-\eqref{eq::a1prop-bis:etazero}),  combined  with \eqref{stab:result_vel}, \eqref{aux:estimate:a} and \eqref{eq::bcont}. This yields
\begin{equation}
    \begin{aligned}
    & C\nu_1 \lVert \varepsilon(\mathbf{u}_h -\mathbf{v}_h) \rVert_{\textbf{L}^s(\Omega)}^2\\
    & \leq C_s\nu_2\lVert \varepsilon(\mathbf{u}-\mathbf{v}_h) \rVert^{s-1}_{\textbf{L}^s(\Omega)} \lVert \varepsilon(\mathbf{u}_h-\mathbf{v}_h) \rVert_{\textbf{L}^s(\Omega)} \\&+ C_s \nu_3 \lVert \nabla(\vartheta-\vartheta_h)\rVert \|\varepsilon(\mathbf{u}_h)\|^{s-1}_{\textbf{L}^{(s-1)p_2}(\Omega)}
    \lVert \varepsilon(\mathbf{u}_h-\mathbf{v}_h) \rVert_{\textbf{L}^s(\Omega)} \\
    &+M\lVert \varepsilon(\mathbf{u}_h-\mathbf{u})\rVert_{\textbf{L}^s(\Omega)} \lVert \pi-q_h\rVert_{L^{s'}(\Omega)}+
    M\lVert \varepsilon(\mathbf{u}-\mathbf{v}_h)\rVert_{\textbf{L}^s(\Omega)} \lVert \pi-\pi_h\rVert_{L^{s'}(\Omega)}
    \end{aligned}
\end{equation}
where $C=\eta_\infty$ for $\eta_\infty>0$ and $C=
C_2 \left(1+ \| \mathbf{u}_h\|^p_{1,p}+\|  \mathbf{v}_h\|^p_{1,p}\right)^{(p-2)/p}$  for $\eta_\infty=0$.
Using \eqref{eq::pressureesti:etazero} we have
\begin{align}
 C\nu_1 \lVert \varepsilon(\mathbf{u}_h -\mathbf{v}_h) \rVert_{\textbf{L}^s(\Omega)}^2
& \leq  C_s\nu_2\lVert \varepsilon(\mathbf{u}-\mathbf{v}_h) \rVert^{s-1}_{\textbf{L}^s(\Omega)} \lVert \varepsilon(\mathbf{u}_h-\mathbf{v}_h) \rVert_{\textbf{L}^s(\Omega)}\nonumber\\
 &+C_s \nu_3 \lVert \nabla(\vartheta-\vartheta_h)\rVert \|\varepsilon(\mathbf{u}_h)\|^{s-1}_{\textbf{L}^{(s-1)p_2}(\Omega)}
    \lVert \varepsilon(\mathbf{u}_h-\mathbf{v}_h) \rVert_{\textbf{L}^s(\Omega)} \nonumber\\
    &  +M\lVert \varepsilon(\mathbf{u}_h-\mathbf{u})\rVert_{\textbf{L}^s(\Omega)} \lVert \pi-q_h\rVert_{{L}^{s'}(\Omega)} \nonumber\\
    &+M(1+\frac{M}{\beta}) \lVert \pi-q_h \rVert_{L^{s'}(\Omega)} \lVert \varepsilon(\mathbf{u}-\mathbf{v}_h)\rVert_{\textbf{L}^s(\Omega)} 
   \nonumber \\&
    +
   M\frac{C_3 \nu_2}{\beta} \lVert \varepsilon(\mathbf{u} - \mathbf{u}_h) \rVert^{s-1}_{\textbf{L}^s(\Omega)} 
   \lVert \varepsilon( \mathbf{u}-\mathbf{v}_h)\rVert_{\textbf{L}^s(\Omega)}
   \nonumber\\&
   +
   {M\frac{C_s \nu_3}{\beta} \lVert \nabla(\vartheta-\vartheta_h)\rVert \|\varepsilon(\mathbf{u}_h)\|^{s-1}_{\textbf{L}^{(s-1)p_2}(\Omega)}
    \lVert \varepsilon(\mathbf{u}_h-\mathbf{v}_h) \rVert_{\textbf{L}^s(\Omega)}}.   \nonumber
\end{align}
Then, applying the triangle inequality together with the generalized Young's  inequality  (with exponents $2/(s-1)$ and $(3-s)/2$) we have, with suitable positive constants $D_1,D_2,D_3,D_4,D_5, D_6$ independent of $h$ and possibly dependent on problem data, the following inequalities {(after taking the square root of each side)}
\begin{eqnarray}
    \lVert \varepsilon(\mathbf{u}_h -\mathbf{v}_h) \rVert_{\textbf{L}^s(\Omega)}
    &\leq& D_1 \lVert \varepsilon(\mathbf{u}-\mathbf{v}_h) \rVert_{\textbf{L}^s(\Omega)}^{s-1} + D_2 \lVert \varepsilon(\mathbf{u}-\mathbf{v}_h) \rVert_{\textbf{L}^s(\Omega)}^{s/2}\nonumber\\
    &&+ D_3 \lVert \varepsilon(\mathbf{u}-\mathbf{v}_h) \rVert_{\textbf{L}^s(\Omega)}^{\frac{1}{3-s}}+D_4 \lVert \varepsilon(\mathbf{u}-\mathbf{v}_h) \rVert_{\textbf{L}^s(\Omega)}\nonumber\\
    &&+D_5 \lVert \pi -q_h \rVert_{\textbf{L}^{s'}(\Omega)} 
    + D_6 \lVert \nabla(\vartheta-\vartheta_h)\rVert \lVert \varepsilon(\mathbf{u}_h) \rVert^{s-1}_{\textbf{L}^{(s-1)p_2}(\Omega)}\nonumber\\
  &\leq & D_7 \lVert \varepsilon(\mathbf{u}-\mathbf{v}_h) \rVert^{s-1}_{\textbf{L}^s(\Omega)} + D_5 \lVert \pi -q_h \rVert_{L^{s'}(\Omega)} \nonumber\\
    &&+~ D_6 \lVert \nabla(\vartheta-\vartheta_h)\|
    \| \varepsilon(\mathbf{u}_h)\|^{s-1}_{\textbf{L}^{(s-1)p_2}(\Omega)}
    \label{eq::velestiqf}
\end{eqnarray}
where in the last step  we assumed, for a sufficiently small $h$, the existence of  $\mathbf{v}_h$ so that   
$e_h:=\lVert \varepsilon(\mathbf{u}-\mathbf{v}_h) \rVert_{\textbf{L}^s(\Omega)} < 1$ (cf. (P1)) and we retained the lowest power of $e_h$ for  $s\in(1,2]$.

Now, using \eqref{eq::tempestif:etazero} and \eqref{eq::velestiqf} we obtain
\begin{eqnarray}
\lVert \nabla(\vartheta_h -s_h)\rVert 
        &\leq& \Lambda_1
  \lVert \nabla(\vartheta -s_h) \rVert  +  \Lambda_2   
        \left(
        \lVert \varepsilon(\mathbf{u}-\mathbf{v}_h)\rVert_{\textbf{L}^s(\Omega)} 
        +
        \lVert \varepsilon(\mathbf{v}_h-\mathbf{u}_h)\rVert_{\textbf{L}^s(\Omega)} \right)  \nonumber\\
         &\leq& \Lambda_1
  \lVert \nabla(\vartheta -s_h) \rVert +  \Lambda_2   
        \lVert \varepsilon(\mathbf{u}-\mathbf{v}_h)\rVert_{\textbf{L}^s(\Omega)} 
 \nonumber\\
 &&+ \Lambda_2\left( 
 D_5 \lVert \varepsilon(\mathbf{u}-\mathbf{v}_h) \rVert^{s-1}_{\textbf{L}^s(\Omega)} + D_2 \lVert \pi -q_h \rVert_{L^{s'}(\Omega)}\right)\nonumber\\
 && +\Lambda_2 D_6 \lVert \nabla(\vartheta-\vartheta_h)\rVert  \| \varepsilon(\mathbf{u}_h)\|^{s-1}_{\textbf{L}^{(s-1)p_2}(\Omega)}
 .\nonumber
\end{eqnarray}
Employing the triangle inequality we then get
\begin{eqnarray}
\lVert \nabla(\vartheta -\theta_h)\rVert
         &\leq& (1+\Lambda_1)
  \lVert \nabla(\vartheta -s_h) \rVert +  \Lambda_2   
        \lVert \varepsilon(\mathbf{u}-\mathbf{v}_h)\rVert_{\textbf{L}^s(\Omega)} 
 \nonumber\\
 &&+ \Lambda_2\left( 
 D_5 \lVert \varepsilon(\mathbf{u}-\mathbf{v}_h) \rVert^{s-1}_{\textbf{L}^s(\Omega)} + D_2 \lVert \pi -q_h \rVert_{L^{s'}(\Omega)}\right)\nonumber\\
 && +\Lambda_2 D_6 \lVert \nabla(\vartheta-\vartheta_h)\rVert  \| \varepsilon(\mathbf{u}_h)\|^{s-1}_{\textbf{L}^{(s-1)p_2}(\Omega)}
 .\label{thetah}
\end{eqnarray}
Therefore, in the case $s=p$, i.e., $\eta_\infty=0$, thanks to assumption (A1), we infer \eqref{eq::testi:etazero} and thus, by \eqref{eq::velestiqf}, \eqref{eq::uesti:etazero}. Finally, using both \eqref{eq::testi:etazero} and \eqref{eq::uesti:etazero}, from \eqref{eq::pressureesti:etazero} we are led to \eqref{eq::pesti:etazero}.
Analogously, in the case $s=2$, i.e., $\eta_\infty>0$, exploiting \eqref{thetah} and assuming (B1), we infer first \eqref{eq::testi:etazero1}, then \eqref{eq::uesti:etazero1} and in conclusion \eqref{eq::pesti:etazero1}. The proof is finished.

\subsection{Numerical Experiments}\label{S:numres}
The aim of this section is twofold: (a) to corroborate the theoretical estimates of Theorems \ref{theoesti:1} and \ref{theoesti:2}; (b) to explore the r\^{o}le of the regularizing parameter 
$\sigma$ in the approximating problem \eqref{Sy08}-\eqref{Sy11}. 

Having in mind these goals, we perform  the numerical tests on the two-dimensional unit square $\Omega=(0,1)^2$ by employing the following finite element spaces
\begin{eqnarray}
    \mathbf{V}_h &=& \{\mathbf{v}_h \in \textbf{C}(\overline{\Omega}) | \forall K \in \mathcal{T}_h, \ \mathbf{v}_h|_{K} \in \mathbb{P}_{r+1}(K)^2\},\nonumber\\
    {Q}_h &=& \{{q}_h \in C(\overline{\Omega}) | \forall K \in \mathcal{T}_h, \ {q}_h|_{K} \in \mathbb{P}_{r}(K)\},\nonumber\\
    {V}_h &=& \{\varrho_h \in C(\overline{\Omega}) | \forall K \in \mathcal{T}_h, \ \varrho_h|_{K} \in \mathbb{P}_{r+1}(K)\},\nonumber   
\end{eqnarray}
with $r\geq1$. Note that the compatibility condition \eqref{eq::bcompatidis} is satisfied; see, e.g., Refs. \refcite{brezziboffi,girault}. 


The discrete nonlinear problem \eqref{eq::liftweakdisc2}-\eqref{eq::variadisclift} is solved by resorting to the following fixed point strategy.

\noindent Set $(\mathbf{u}_h^{(0)},\vartheta^{(0)}_h)=(\mathbf{0},0)$, $k=0$  and iterate:

\noindent {\bf Step $1$}. Given 
$(\mathbf{u}_h^{(k)},\vartheta^{(k)}_h)$ 
find $(\mathbf{u}_h^{(k+1)}, \pi^{(k+1)}_h)$ such that for all $(\mathbf{v}_h, q_h)$ there holds
\begin{eqnarray}
\widetilde{a}_1(\vartheta^{(k)}_h+\Theta_0,\mathbf{u}_h^{(k)};\mathbf{u}_h^{(k+1)},\mathbf{v}_h) + b(\mathbf{v}_h,\pi^{(k+1)}_h) &=& \mathbf{f}(\mathbf{v}_h)\nonumber \\
b(\mathbf{u}_h^{(k+1)},q_h) &=& 0 \nonumber
\end{eqnarray}

 where $$
 \widetilde{a}_1(\vartheta, \mathbf{w} ; \mathbf{u}, \mathbf{v}) = \int_{\Omega}2 \nu(\vartheta) \eta(\vert {\varepsilon}(\mathbf{w})\vert^2)\varepsilon(\mathbf{u}) : \varepsilon(\mathbf{v}).
 $$
\noindent {\bf Step 2}. Given $\mathbf{u}_h^{(k+1)}$
find $\vartheta^{(k+1)}_h$ such that for all $\rho_h$ there holds
\begin{eqnarray}
a_2(\vartheta^{(k+1)}_h,\varrho_h) + c_h(\mathbf{u}_h^{(k)},\vartheta^{(k+1)}_h, \varrho_h)&=&g(\varrho_h) - a_2(\Theta_0,\varrho_h)-c_h(\mathbf{u}_h^{(k)},\Theta_0,\varrho_h).\nonumber
\end{eqnarray}

\noindent {\bf Step $3$}. $k+1\to k$

The iteration is stopped when 
$$
\|\mathbf{u}_h^{(k+1)}-\mathbf{u}_h^{(k)}\|_{\textbf{L}^s(\Omega)} +  \|\pi^{(k+1)}_h-\pi^{(k)}_h\|_{L^{s'}(\Omega)} + \|\vartheta^{(k+1)}_h-\vartheta^{(k)}_h\| < \texttt{tol}.
$$

Numerical tests are performed with ${\tt tol}=10^{-10}$ using the high level C++ interface of FEniCS-DOLFIN\cite{dolfin,fenics}.

\subsubsection{Test 1}\label{test1}
We consider the finite element approximation of the non-isothermal non-Newtonian flow problem governed by the Carreau law with $\nu(\xi)=e^{-\xi},\ \xi\in \R^+$,  cf. \eqref{eq::liftweakdisc2}-\eqref{eq::variadisclift}. The source term  $\mathbf{f}$ is manufactured so that  the exact solution  is given by

\begin{eqnarray}\label{eq:exactsolution}
u_x(x,y) &=& 5y \sin(x^2+y^2)+4y \sin(x^2-y^2), \\
u_y(x,y) &=& -5x \sin(x^2+y^2)+4x \sin(x^2-y^2), \\
\pi(x,y) &=& \sin(x+y), \\
\theta(x,y)&=&\cos(xy), 
\end{eqnarray}
where $\mathbf{u}=(u_x,u_y)$.  Dirichlet boundary conditions for  velocity and temperature given by the exact solution are imposed on the domain boundary. 

We first consider the problem with the Carreau law parameters in \eqref{carreau_law} defined as $\eta_\infty=0.5, \eta_0=2,  \lambda = 1$ for different values of $p=2, 1.6, 1.2$. 

The convergence results in terms of the $\textbf{L}^2$ and $\textbf{H}^1$ velocity errors, $L^2$ pressure error and $H^1$ temperature error, obtained using quadratic ($r=2$) finite elements for velocity and temperature and linear ($r=1$) finite elements for pressure, are reported in Figure \ref{fig:P2P1P2_nuinf05} (left).  

The corresponding results  when cubic ($r=3$) finite elements for velocity and temperature and quadratic ($r=2$) finite elements for pressure are adopted, as displayed in Figure  \ref{fig:P2P1P2_nuinf05} (right). In both cases, optimal convergence rates are achieved (with a slight superconvergence for the pressure error). This is in agreement with the approximation results of Theorem \ref{theoesti:2} combined with standard interpolation error estimates {(cf. Remark \ref{oss:orders})}.

\begin{figure}[hbt]
\centering
\begin{tabular}{cc}
\includegraphics[width=0.47\textwidth]{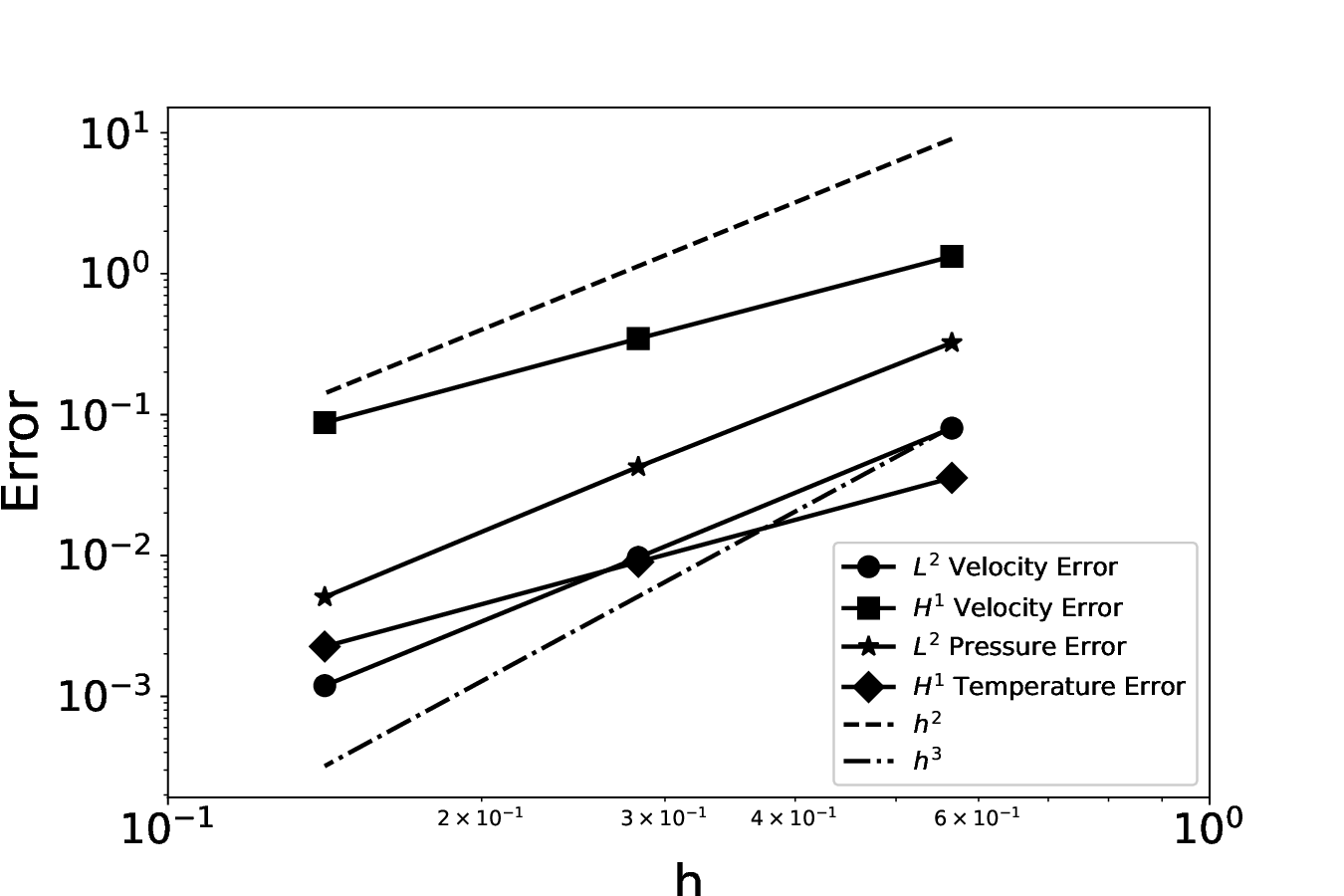} &
\includegraphics[width=0.47\textwidth]{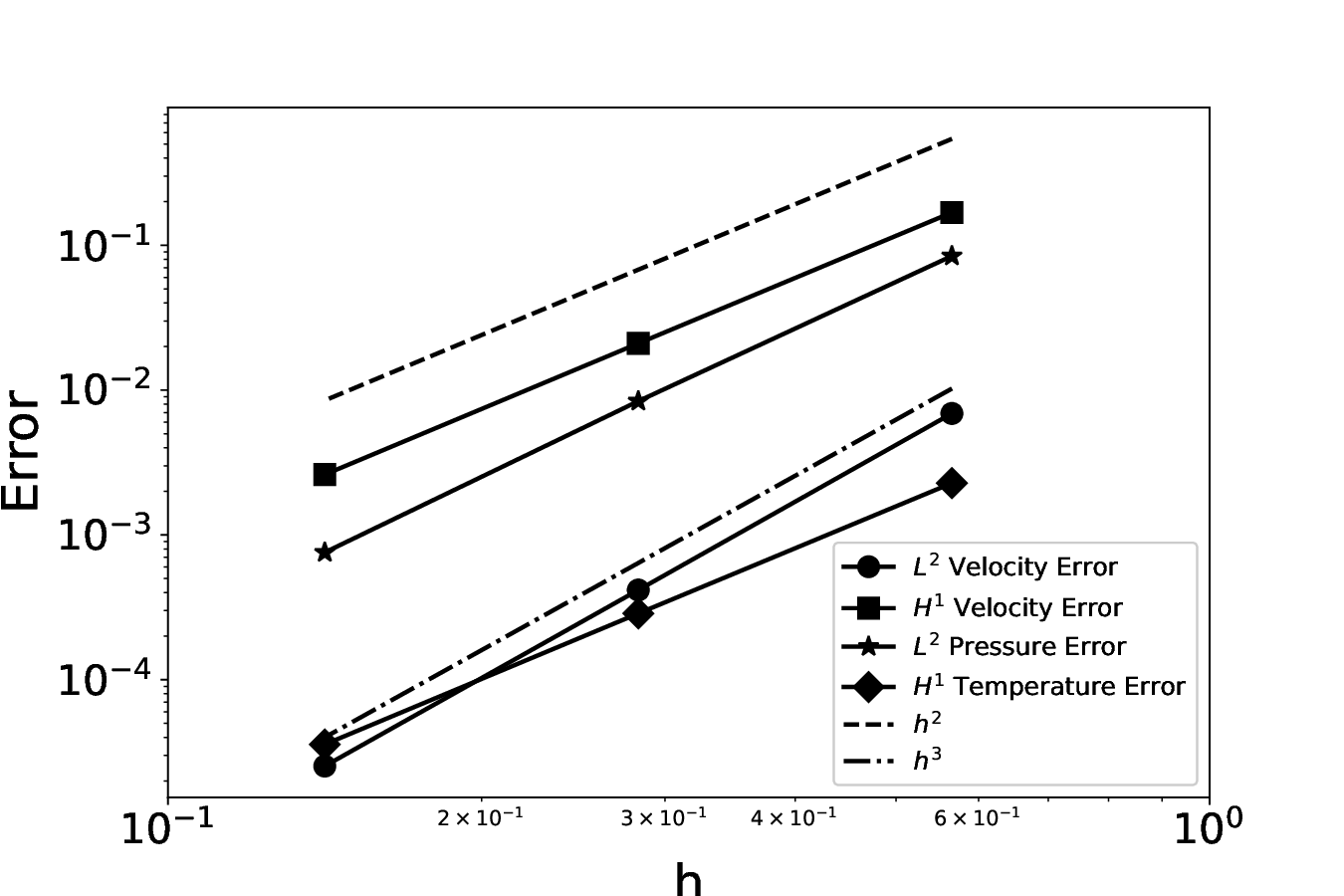} \\
\includegraphics[width=0.47\textwidth]{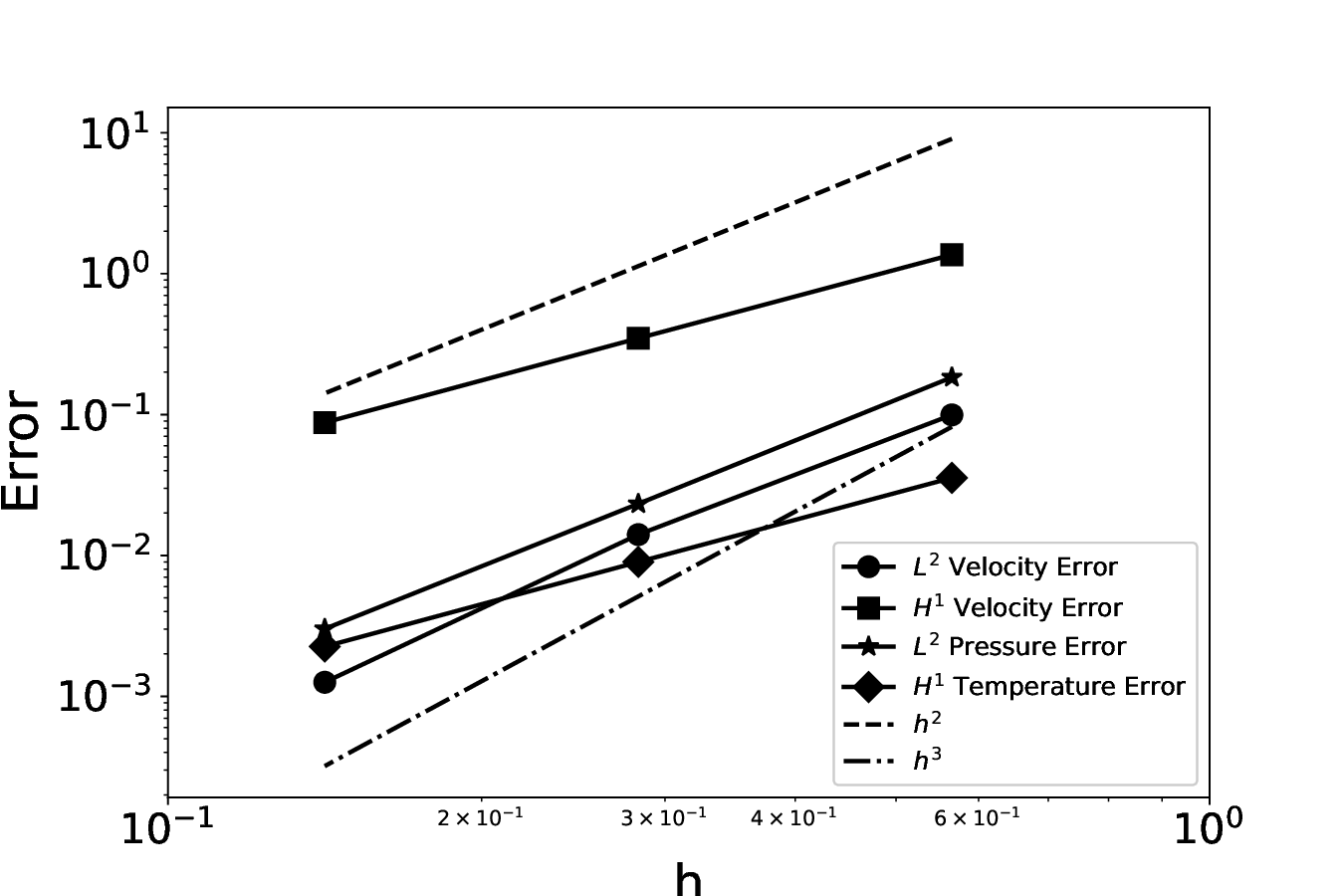} & 
\includegraphics[width=0.47\textwidth]{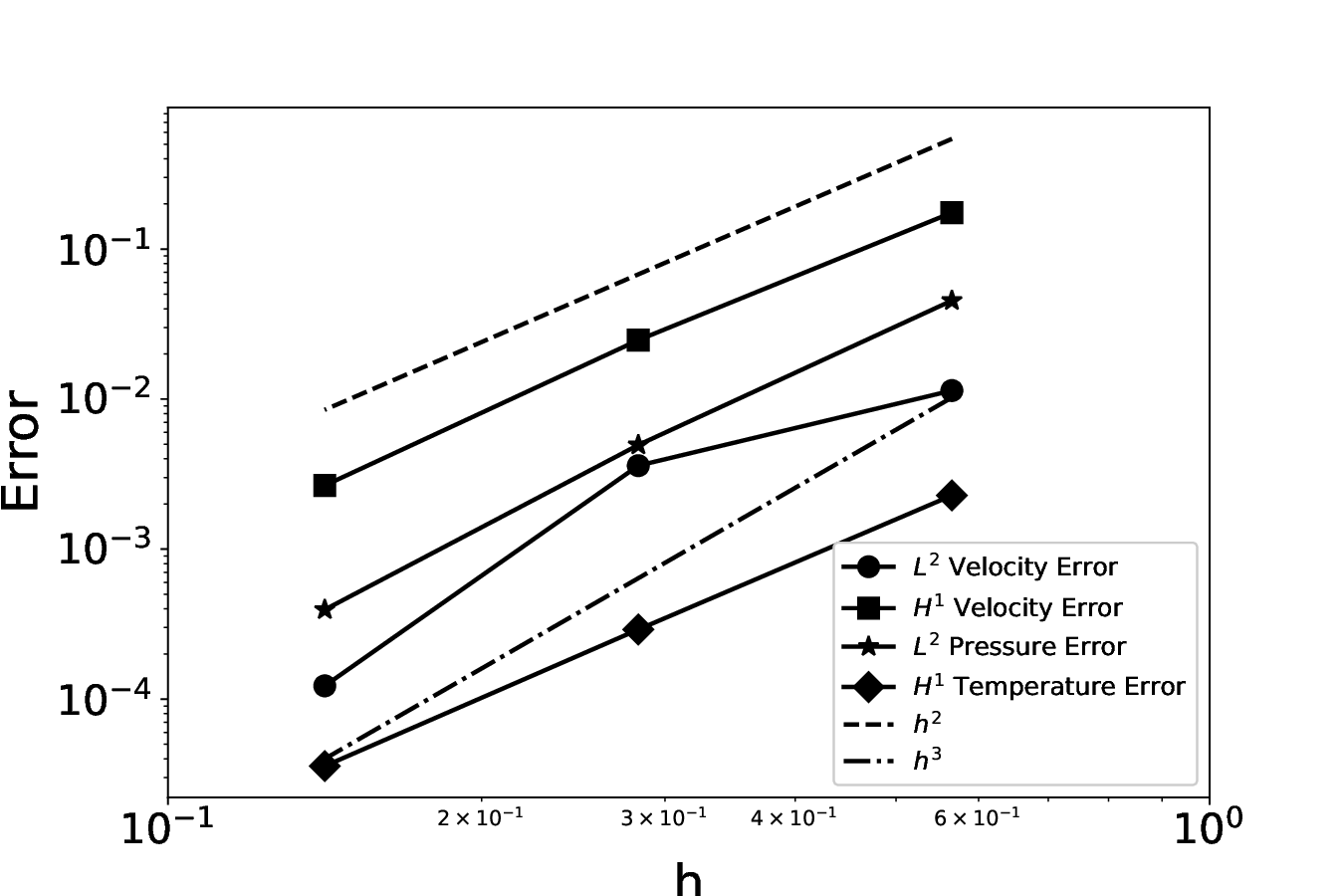} \\
\includegraphics[width=0.47\textwidth]{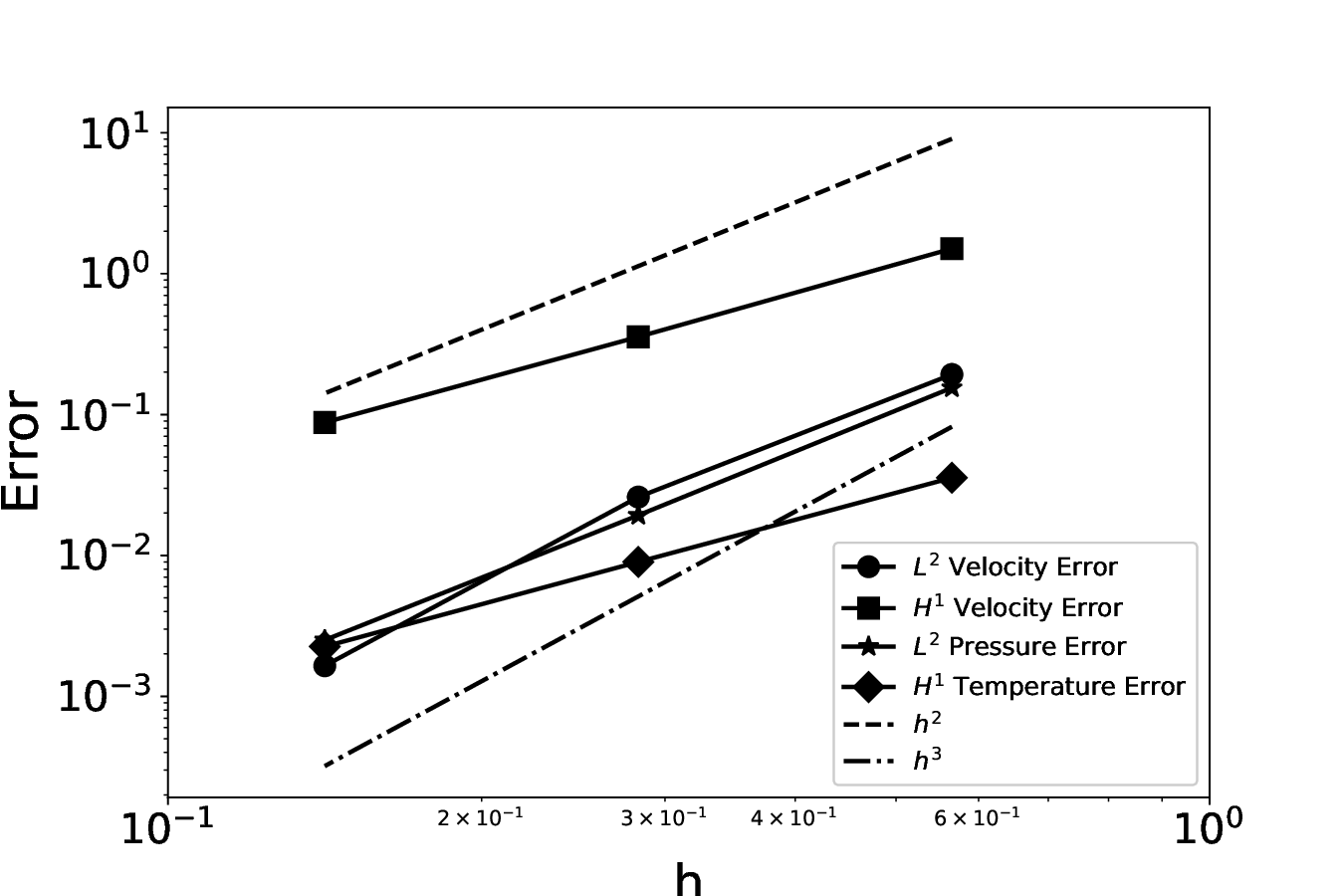} &
\includegraphics[width=0.47\textwidth]{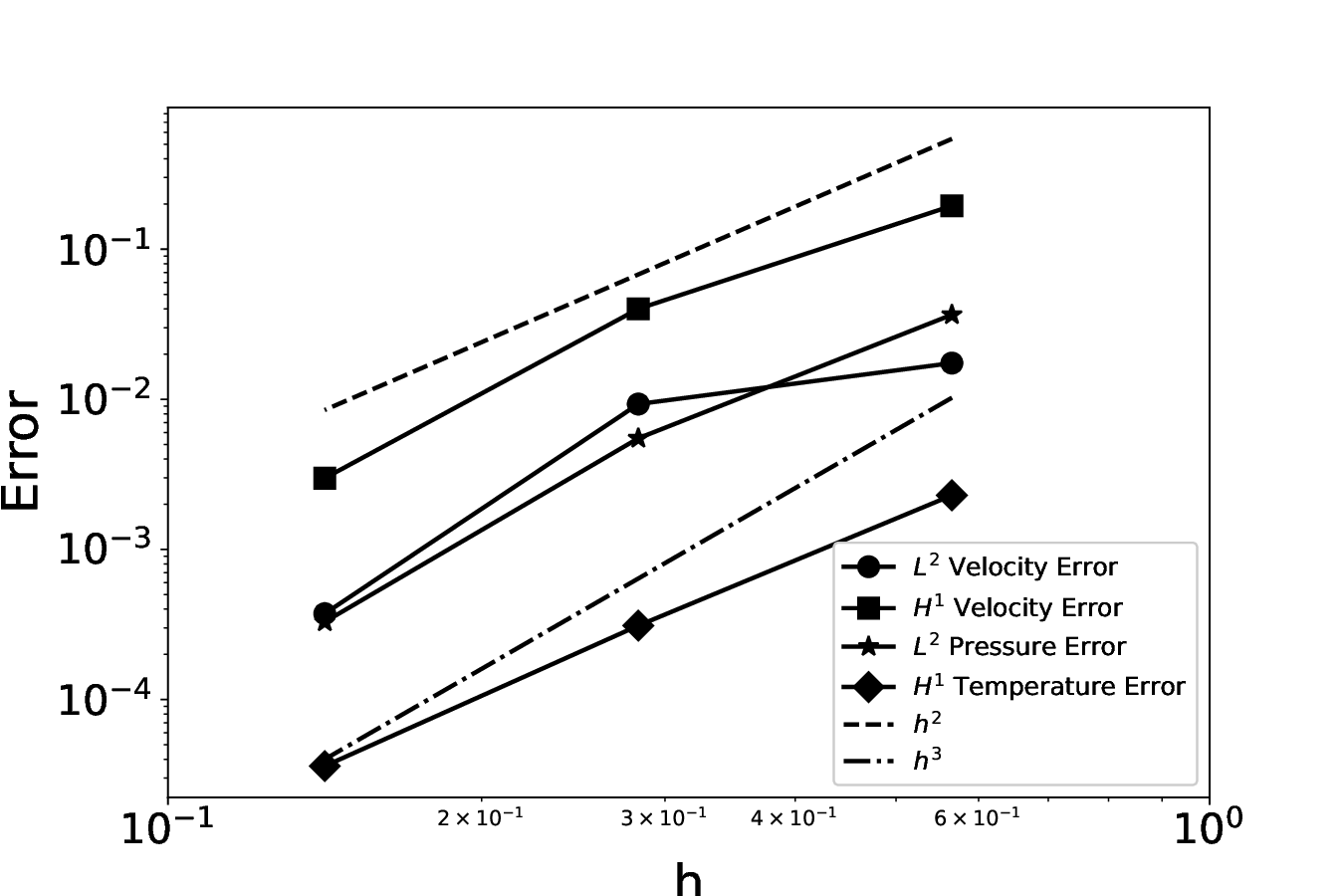} 
\end{tabular}
\caption{Test $1$. Convergence test for the Carreau model with $\eta_\infty=0.5, \eta_0=2,  \lambda = 1$ and different values of $p=2, 1.6, 1.2$ (from top to bottom) using $\mathcal{P}_2/\mathcal{P}_1/\mathcal{P}_2$ (left) and $\mathcal{P}_3/\mathcal{P}_2/\mathcal{P}_3$ (right) finite elements}
\label{fig:P2P1P2_nuinf05}
\end{figure}

\subsubsection{Test 2}\label{test2}

In this second test case, we investigate the r\^{o}le of the regularizing parameter $\sigma$ in the approximating problem \eqref{Sy08}-\eqref{Sy11} by solving the same manufactured solution problem introduced in the first test case, but considering the Carreau model with $\eta_\infty=0, \eta_0=2, \lambda = 1$ and different values of  $p$. Consistently with the value of $\eta_{\infty}$, the velocity error is computed in $\textbf{W}^{1,p}$, the pressure error in $L^{p'}$, while the temperature error is computed in $H^1$. 

The convergence results obtained using $\mathcal{P}_2/\mathcal{P}_1/\mathcal{P}_2$ finite elements are displayed in Figures \ref{fig:P2P1P2_p2_sigma}, \ref{fig:P2P1P2_p1.6_sigma} and \ref{fig:P2P1P2_p1.2_sigma} for  $p=2, 1.6, 1.2$, respectively.  In each plot the slope of the dotted reference line  is computed by employing the values of the corresponding error obtained for  $\sigma=0$ and the two smallest values of $h$. As expected, for $\sigma\not=0$ the error between the exact solution of \eqref{eq::liftweak}-\eqref{eq::variastokesnllift} with $\eta_\infty=0$ and the  approximation of the solution to the regularized problem \eqref{Sy08}-\eqref{Sy11} exhibits an asymptotic plateau for $h$ tending to zero, where the value of  the plateau decreases as $\sigma$ tends to zero. When the regularization parameter $\sigma$ is set to zero, from Theorem \ref{theoesti:1} and standard interpolation error estimates we expect the velocity and the  temperature errors to behave like $h^{2(p-1)}$, while the pressure error as $h^{2(p-1)^2}$ {(cf. Remark \ref{oss:orders})}. The obtained convergence results are coherent with the theoretical estimates in Theorem \ref{theoesti:1}, in particular the rate of convergence of the pressure error decreases with $p$. { However, we observe that in some cases (particularly for $p\not=2$) the expected asymptotic orders of convergence are exceeded; this may be related to the fact that the asymptotic regime is not yet reached.} We remark that a similar numerical behavior, in the context of isothermal non-Newtonian fluids, has been also  observed, e.g., in Ref. \refcite{Botti_et_al:2021}.

\begin{figure}[hbt]
\centering
\includegraphics[width=\textwidth]{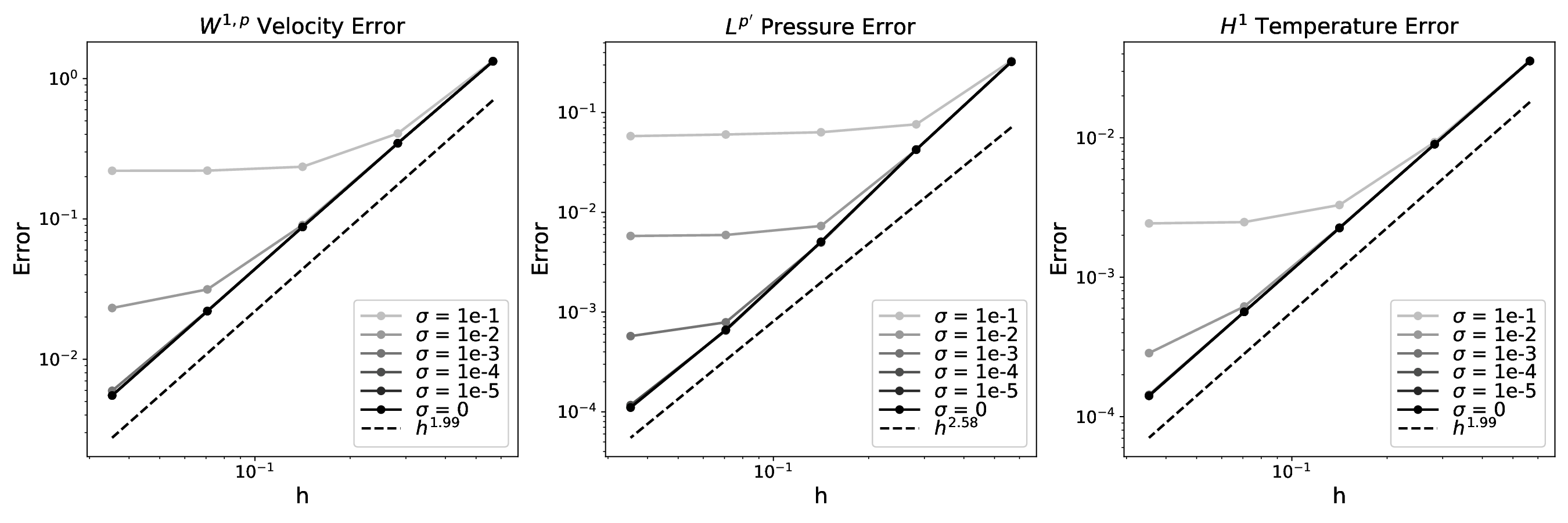}
\caption{Test $2$. Convergence test for the Carreau model with $\eta_\infty=0, \eta_0=2. \lambda = 1, p=2$ and different values of the regularization parameter $\sigma$ using $\mathcal{P}_2/\mathcal{P}_1/\mathcal{P}_2$ finite elements.}
\label{fig:P2P1P2_p2_sigma}
\end{figure}

\begin{figure}[hbt]
\centering
\includegraphics[width=\textwidth]{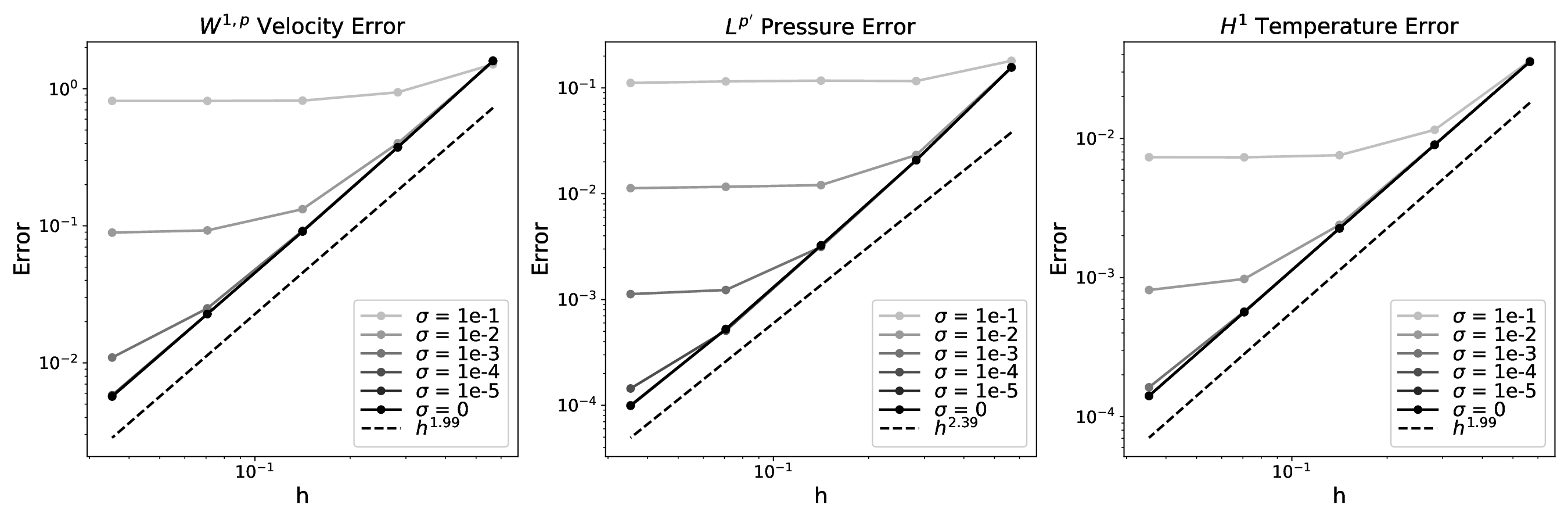}
\caption{Test $2$. Convergence test for the Carreau model with $\eta_\infty=0, \eta_0=2,  \lambda = 1, p=1.6$ and different values of the regularization parameter $\sigma$ using $\mathcal{P}_2/\mathcal{P}_1/\mathcal{P}_2$ finite elements.}
\label{fig:P2P1P2_p1.6_sigma}
\end{figure}

\begin{figure}[hbt]
\centering
\includegraphics[width=\textwidth]{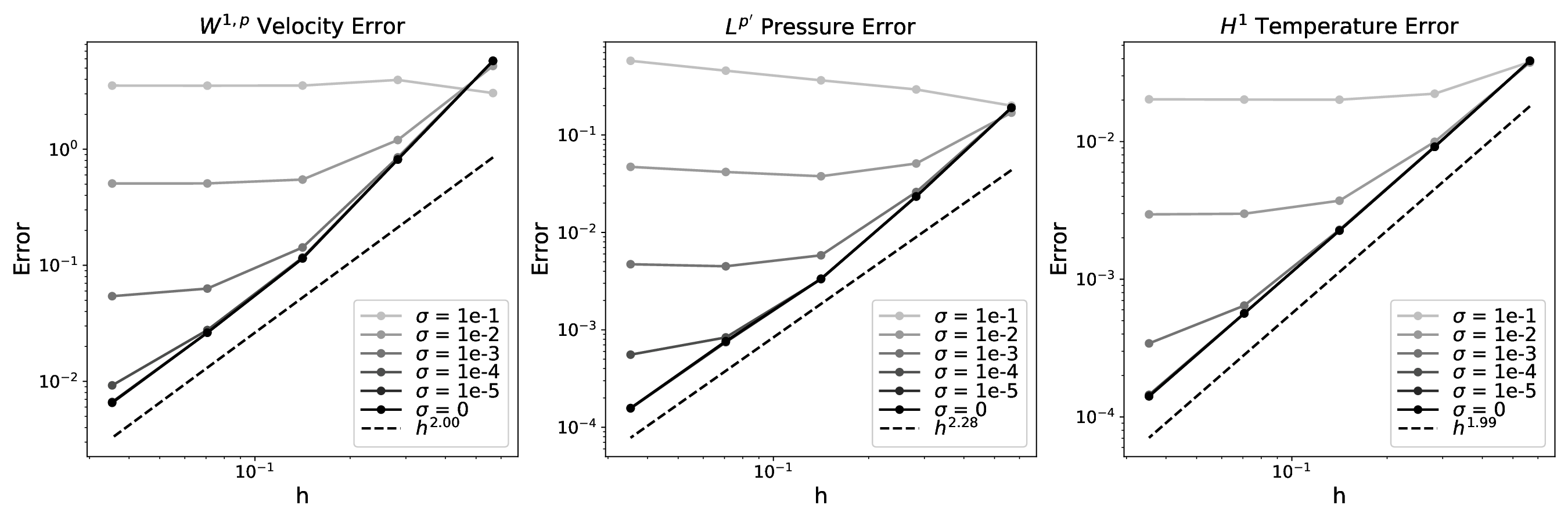}
\caption{Test $2$. Convergence test for the Carreau model with $\eta_\infty=0, \eta_0=2,  \lambda = 1, p=1.2$ and different values of the regularization parameter $\sigma$ using $\mathcal{P}_2/\mathcal{P}_1/\mathcal{P}_2$ finite elements.}
\label{fig:P2P1P2_p1.2_sigma}
\end{figure}

Similar comments apply when $\mathcal{P}_3/\mathcal{P}_2/\mathcal{P}_3$ finite elements are employed. In this case  we expect the velocity and the  temperature errors to behave like $h^{3(p-1)}$, while the pressure error to converge as $h^{3(p-1)^2}$ {(cf. Remark \ref{oss:orders})}. The corresponding results are  reported in  Figures \ref{fig:P3P2P3_p2_sigma}, \ref{fig:P3P2P3_p1.6_sigma} and \ref{fig:P3P2P3_p1.2_sigma} for  $p=2, 1.6, 1.2$, respectively. 

\begin{figure}[hbt]
\centering
\includegraphics[width=\textwidth]{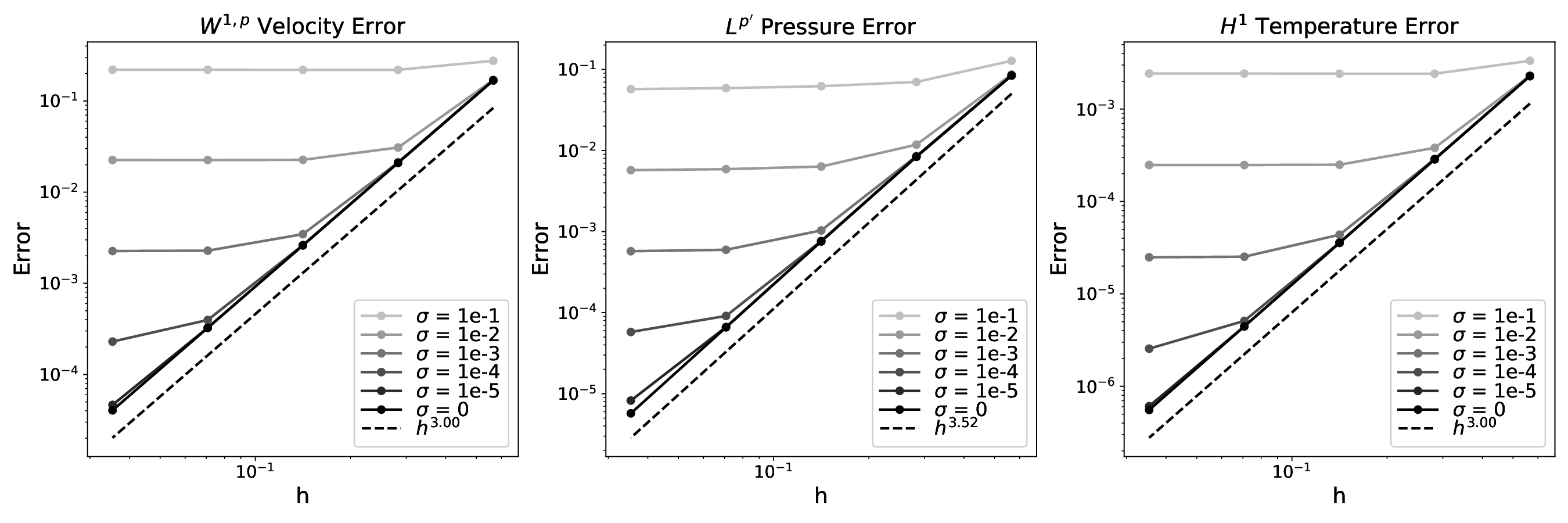}
\caption{Test $2$. Convergence test for the Carreau model with $\eta_\infty=0, \eta_0=2,  \lambda = 1, p=2$ and different values of the regularization parameter $\sigma$ using $\mathcal{P}_3/\mathcal{P}_2/\mathcal{P}_3$ finite elements.}
\label{fig:P3P2P3_p2_sigma}
\end{figure}

\begin{figure}[hbt]
\centering
\includegraphics[width=\textwidth]{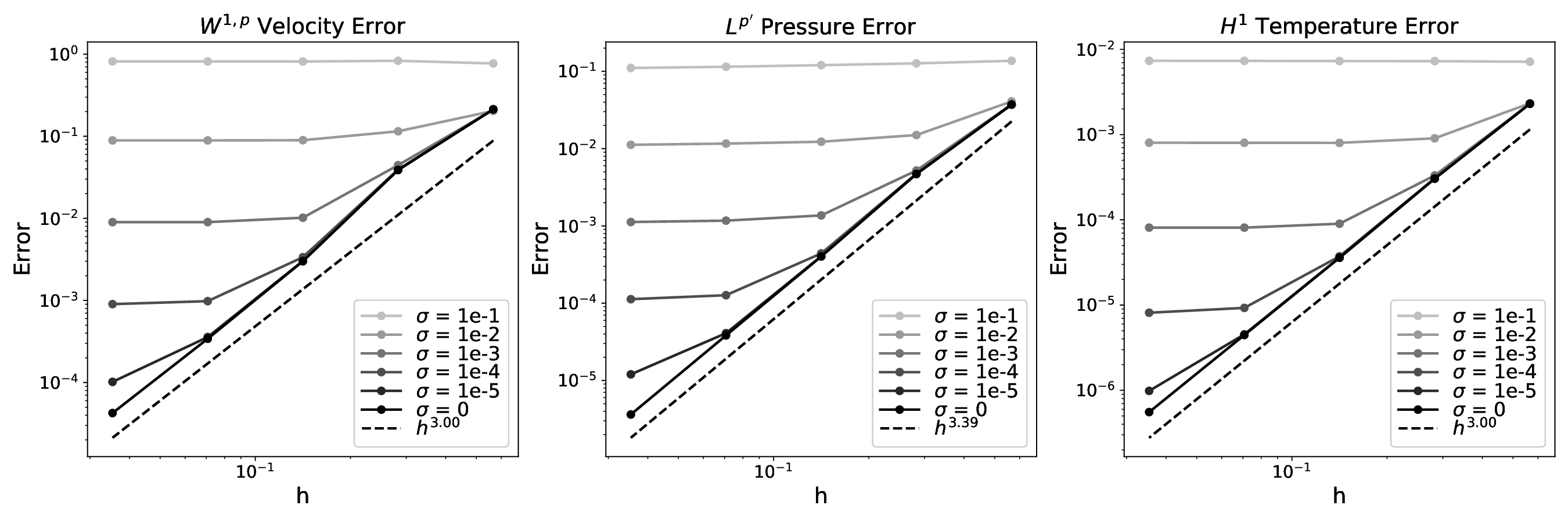}
\caption{Test $2$. Convergence test for the Carreau model with $\eta_\infty=0, \eta_0=2,  \lambda = 1, p=1.6$ and different values of the regularization parameter $\sigma$ using $\mathcal{P}_3/\mathcal{P}_2/\mathcal{P}_3$ finite elements. }
\label{fig:P3P2P3_p1.6_sigma}
\end{figure}

\clearpage

\begin{figure}[hbt]
\centering
\includegraphics[width=\textwidth]{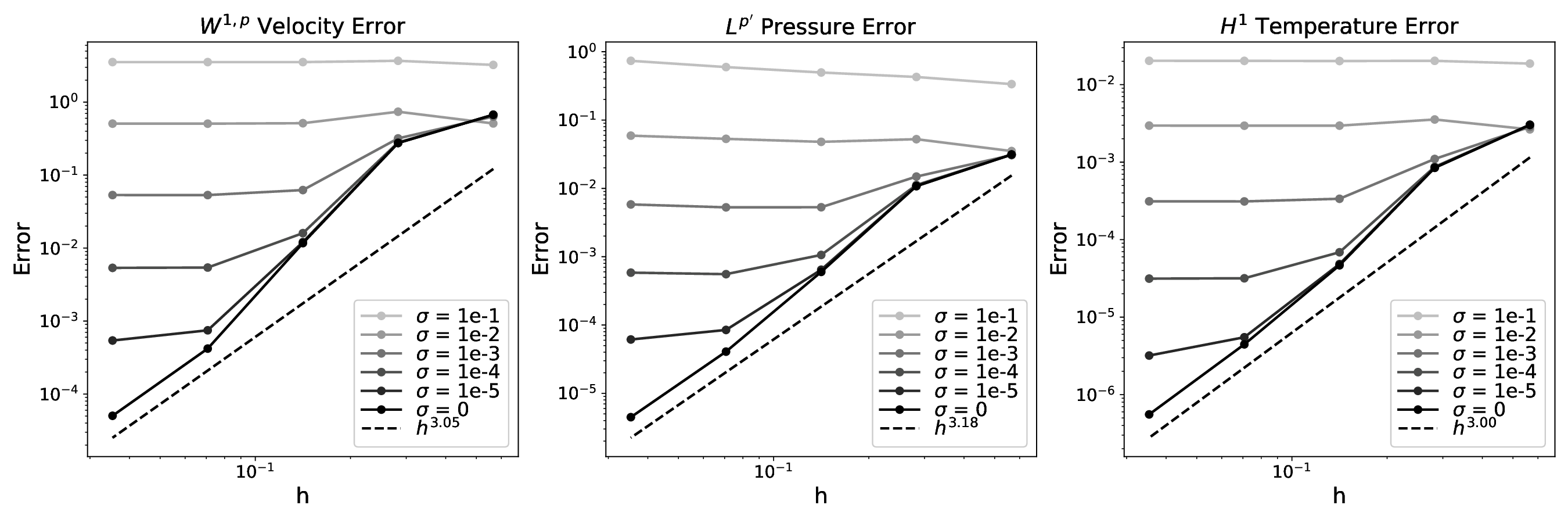}
\caption{Test $2$. Convergence test for the Carreau model with $\eta_\infty=0, \eta_0=2,  \lambda = 1, p=1.2$ and different values of the regularization parameter $\sigma$ using $\mathcal{P}_3/\mathcal{P}_2/\mathcal{P}_3$ finite elements. }
\label{fig:P3P2P3_p1.2_sigma}
\end{figure}

\section{Appendix}
{ Here we collect the main technical tools necessary for the proof of Theorem \ref{exist}. In particular, we report, for the sake of completeness, Ref. \refcite{Lip}, Theorem 2.5:
\begin{theorem}
Let $1<p<\infty$ and let $\Omega\subset \R^d$ be  bounded domain with Lipschitz boundary. Assume that $\{\mathbf{w}_n\}\subset \textbf{W}_0^{1,p}(\Omega)$ is such that $\mathbf{w}_n\rightharpoonup 0$ in $\textbf{W}_0^{1,p}(\Omega)$ as $n\to \infty$. Set 
$$
K:=\sup_{n}\Vert \mathbf{w}_n\Vert_{\textbf{W}^{1,p}(\Omega)}<\infty,\quad \gamma_n:=\Vert \mathbf{w}_n\Vert_{\textbf{L}^p(\Omega)}\to 0\quad \text{as }n\to \infty.
$$
Let $\theta_n>0$ be such that (e.g., $\theta_n:=\sqrt{\gamma_n}$)
$$
\theta_n\to 0 \quad\text{and }\quad \frac{\gamma_n}{\theta_n}\to 0\quad \text{as }n\to \infty.
$$
Let $\mu_j:=2^{2^j}$. Then there exists a sequence $\{\lambda_{n,j}\}$ of positive numbers such that 
$$
\mu_j\leq \lambda_{n,j}\leq \mu_{j+1},
$$
and a sequence $\{\mathbf{w}_{n,j}\}\subset \textbf{W}_0^{1,p}(\Omega)$ such that for all $j,n\in\mathbb{N}$
\begin{align}
\label{a01}
    &\Vert \mathbf{w}_{n,j}\Vert_{\textbf{L}^\infty(\Omega)}\leq \theta_n\to 0\quad \text{as }n\to\infty,\\&
    \Vert \nabla \mathbf{w}_{n,j}\Vert_{\textbf{L}^\infty(\Omega)}\leq c\lambda_{n,j}\leq c\mu_{j+1},\label{a02}
\end{align} 
for some $c>0$ independent of $j,n$. Moreover, for all $j\in\mathbb{N}$ and for $n\to\infty$, there holds
\begin{align*}
   &\mathbf{w}_{n,j}\to \textbf{0} \quad\text{in }\textbf{L}^s(\Omega)\quad\forall s\in[1,\infty],\\&
    \mathbf{w}_{n,j}\rightharpoonup \textbf{0}\quad \text{ in }\textbf{W}_0^{1,s}(\Omega)\quad \forall s\in[1,\infty),\\&
    \nabla\mathbf{w}_{n,j}\rightharpoonup^* \textbf{0}\quad\text{ in }\textbf{L}^\infty(\Omega).
\end{align*}
In conclusion, for all $n,j\in\mathbb{N}$, we have
\begin{align}
\Vert \nabla\mathbf{w}_{n,j}\chi_{\{\mathbf{w}_{n,j}\not=\mathbf{w}_n\}}\Vert_{\textbf{L}^p(\Omega)}\leq c\Vert \lambda_{n,j}\chi_{\{\mathbf{w}_{n,j}\not=\mathbf{w}_n\}}\Vert_{\textbf{L}^p(\Omega)}\leq c\frac{\gamma_n}{\theta_n}\mu_{j+1}+c\epsilon_j,
    \label{cc}
\end{align}
where $\epsilon_j:=K2^{-j/p}\to 0$ as $j\to \infty$. 
\label{app1}
\end{theorem}
We also state and prove the following lemma, which is a variant of Ref. \refcite{Lip}, Lemma 2.6, in order to take into account the presence of a temperature-dependent viscosity.
\begin{lemma}
\label{app2}
Let $\Omega$ and $p$ as in Theorem \ref{app1}. Let $\{\mathbf{u}_n\},\mathbf{u}\in \textbf{W}_0^{1,p}(\Omega)$ with $\mathbf{u}_n\rightharpoonup \mathbf{u}$ in $\textbf{W}_0^{1,p}(\Omega)$. Let $\mathbf{w}_n:=\mathbf{u}_n-\mathbf{u}$, and let $\mathbf{w}_{n,j}$ be the approximation of $\mathbf{w}_n$ as in Theorem \ref{app1}. Moreover, let $\{\varphi_n\},\varphi\in L^p(\Omega)$ be a sequence such that $\varphi_n\to \varphi$ in $L^p(\Omega)$ as $n\to \infty$. Assume that, for all $j\in \mathbb{N}$, we have
\begin{align}
\limsup_{n \to \infty} \int_\Omega \left(\nu(\varphi_n)\tau(x,\varepsilon(\mathbf{u}_n))-\nu(\varphi)\tau(x,\varepsilon(\mathbf{u}))\right):\varepsilon(\mathbf{w}_{n,j})dx\leq \delta_j,
\label{cca}
\end{align}
for $\delta_j>0$ such that $\delta_j\to 0$
as $j\to \infty$. Then, for any $\zeta\in(0,1)$, it holds
\begin{align}
\limsup_{n\to\infty}\int_\Omega \left[\nu(\varphi)(\tau(x,\varepsilon(\mathbf{u}_n))-\tau(x,\varepsilon(\mathbf{u}))):(\varepsilon(\mathbf{u}_n)-\varepsilon(\mathbf{u}))\right]^\zeta dx =0.
    \label{fond}
\end{align}
\end{lemma}

\begin{proof}
The assumptions imply that there is $C>0$, independent of $n$, such that
\begin{align}
\Vert \tau(\cdot,\varepsilon(\mathbf{u}_n))\Vert_{\textbf{L}^{p^\prime}(\Omega)}+\Vert \tau(\cdot,\varepsilon(\mathbf{u}))\Vert_{\textbf{L}^{p^\prime}(\Omega)}+\Vert \varepsilon(\mathbf{u}_n)\Vert_{\textbf{L}^{p}(\Omega)}+\Vert \varepsilon(\mathbf{u})\Vert_{\textbf{L}^{p}(\Omega)}\leq C.
    \label{ests}
\end{align}
Then, for any $j\in \mathbb{N}$, being $\varepsilon(\cdot)$ linear, $\{\mathbf{w}_{n,j}=\mathbf{w}_n\}=\Omega\setminus \{\mathbf{w}_{n,j}\not=\mathbf{w}_n\}$ and $\nu\in W^{1,\infty}(\R)$, we have that (see \eqref{cca})
\begin{align}
   \nonumber& \limsup_{n\to\infty} I_n^j:=\limsup_{n\to\infty}\int_{\{\mathbf{w}_{n,j}=\mathbf{w}_n\}}\nu(\varphi)(\tau(x,\varepsilon(\mathbf{u}_n))-\tau(x,\varepsilon(\mathbf{u}))):(\varepsilon(\mathbf{u}_n)-\varepsilon(\mathbf{u})) dx\\&\nonumber\leq \limsup_{n\to\infty}\int_{\{\mathbf{w}_{n,j}=\mathbf{w}_n\}}(\nu(\varphi_n)\tau(x,\varepsilon(\mathbf{u}_n))-\nu(\varphi)\tau(x,\varepsilon(\mathbf{u}))):\varepsilon(\mathbf{w}_{n,j})dx\\&\nonumber+
    \limsup_{n\to\infty}\int_{\{\mathbf{w}_{n,j}=\mathbf{w}_n\}}(\nu(\varphi)-\nu(\varphi_n))\tau(x,\varepsilon(\mathbf{u}_n)):\varepsilon(\mathbf{w}_{n,j}) dx\\&\nonumber
    \leq \delta_j+\limsup_{n\to\infty}\left\vert \int_{\{\mathbf{w}_{n,j}\not=\mathbf{w}_n\}}(\nu(\varphi_n)\tau(x,\varepsilon(\mathbf{u}_n))-\nu(\varphi)\tau(x,\varepsilon(\mathbf{u}))):\varepsilon(\mathbf{w}_{n,j})dx\right\vert\\&
    \nonumber+C\limsup_{n\to\infty}\Vert \varphi_n-\varphi\Vert_{L^p(\Omega)}\Vert \tau(x,\varepsilon(\mathbf{u}_n))\Vert_{\textbf{L}^{p^\prime}(\Omega)}\Vert \varepsilon(\mathbf{w}_{n,j})\Vert_{\textbf{L}^\infty(\Omega)}\\&\nonumber\leq \delta_j+C\limsup_{n\to\infty}(\Vert \tau(x,\varepsilon(\mathbf{u}_n))\Vert_{\textbf{L}^{p^\prime}(\Omega)}+\Vert \tau(x,\varepsilon(\mathbf{u}))\Vert_{\textbf{L}^{p^\prime}(\Omega)})\Vert \nabla\mathbf{w}_{n,j}\chi_{\{\mathbf{w}_{n,j}\not=\mathbf{w}_n\}}\Vert_{\textbf{L}^p(\Omega)}\\&\leq \delta_j+C\epsilon_j,
    \label{in}
\end{align}
where we exploited \eqref{ests}, the convergence $\varphi_n\to \varphi$ in $L^p(\Omega)$ and \eqref{a02} in the last but one inequality, whereas in the last step we exploited \eqref{cc}, recalling that $\frac{\gamma_n}{\theta_n}\to 0$ as $n\to\infty$, and \eqref{ests}. 
Recalling that, by the assumptions on $\tau$, it holds $\nu(\varphi)(\tau(x,\varepsilon(\mathbf{u}_n))-\tau(x,\varepsilon(\mathbf{u}))):\varepsilon(\mathbf{w}_n)dx\geq0$, we now have, for any $\zeta\in(0,1)$, using H\"{o}lder's inequality,
\begin{align}
  &\nonumber\limsup_{n\to\infty}\int_\Omega \left[\nu(\varphi)(\tau(x,\varepsilon(\mathbf{u}_n))-\tau(x,\varepsilon(\mathbf{u}))):\varepsilon(\mathbf{w}_n)\right]^\zeta dx \\&\nonumber\leq \limsup_{n\to\infty}\left[\int_{\{\mathbf{w}_{n,j}=\mathbf{w}_n\}}\nu(\varphi)(\tau(x,\varepsilon(\mathbf{u}_n))-\tau(x,\varepsilon(\mathbf{u}))):\varepsilon(\mathbf{w}_n)dx\right]^\zeta\vert \Omega\vert^{1-\zeta}\\&\nonumber
  +\limsup_{n\to\infty}\left[\int_{\{\mathbf{w}_{n,j}\not=\mathbf{w}_n\}}\nu(\varphi)(\tau(x,\varepsilon(\mathbf{u}_n))-\tau(x,\varepsilon(\mathbf{u}))):\varepsilon(\mathbf{w}_n)dx\right]^\zeta\left\vert \{\mathbf{w}_{n,j}\not=\mathbf{w}_n\}\right\vert^{1-\zeta}\\&\leq \limsup_{n\to\infty} (I^j_n)^\zeta\vert \Omega\vert^{1-\zeta}\nonumber\\&+C\limsup_{n\to\infty}(\Vert \tau(x,\varepsilon(\mathbf{u}_n))\Vert_{\textbf{L}^{p^\prime}(\Omega)}+\Vert \tau(x,\varepsilon(\mathbf{u}))\Vert_{\textbf{L}^{p^\prime}(\Omega)})\Vert \varepsilon(\mathbf{w}_n)\Vert_{\textbf{L}^p(\Omega)} \left\vert \{\mathbf{w}_{n,j}\not=\mathbf{w}_n\}\right\vert^{1-\zeta}\nonumber\\&\leq (\delta_j+C\epsilon_j)^\zeta\vert \Omega\vert^{1-\zeta}+C\limsup_{n\to\infty}\left\vert \{\mathbf{w}_{n,j}\not=\mathbf{w}_n\}\right\vert^{1-\zeta},
  \label{in2}
\end{align}
having applied \eqref{ests} and \eqref{in} in the last estimate. We can conclude by simply observing that, using H\"{o}lder's inequality once more, 
\begin{align*}
 & \limsup_{n\to\infty}\left\vert \{\mathbf{w}_{n,j}\not=\mathbf{w}_n\}\right\vert^{1-\zeta}=\limsup_{n\to\infty}\left\Vert \chi_{\{\mathbf{w}_{n,j}\not=\mathbf{w}_n\}}\right\Vert^{1-\zeta}_{L^1(\Omega)}\\&\leq C\lambda_{n,j}^{\zeta-1}\limsup_{n\to\infty}\left\Vert \lambda_{n,j} \chi_{\{\mathbf{w}_{n,j}\not=\mathbf{w}_n\}}\right\Vert^{1-\zeta}_{L^p(\Omega)}\leq C\epsilon_j^{1-\zeta},
\end{align*}
exploiting \eqref{cc} and the fact that $\lambda_{n,j}\geq 1$.
Therefore, from \eqref{in2}  we get 
$$
\limsup_{n\to\infty}\int_\Omega \left[\nu(\varphi)(\tau(x,\varepsilon(\mathbf{u}_n))-\tau(x,\varepsilon(\mathbf{u}))):\varepsilon(\mathbf{w}_n)\right]^\zeta dx\leq (\delta_j+C\epsilon_j)^\zeta\vert \Omega\vert^{1-\zeta}+ C\epsilon_j^{1-\zeta}.
$$
Being this valid for any $j\in\mathbb{N}$ and since $\lim_{j\to\infty}\epsilon_j= \lim_{j\to\infty} \delta_j=0$, we get \eqref{fond}. This ends the proof.
\end{proof}

}

\newpage

\bigskip
{\bf Acknowledgments}.
{The authors are grateful to the anonymous referee for the careful
reading of the manuscript as well as for the many valuable comments and suggestions which contributed to improve our results (in particular, the ones regarding the technical issues behind the proof of Theorem \ref{exist}).}
M.~Grasselli and A.~Poiatti are members of Gruppo Nazionale per
l'Analisi Ma\-te\-ma\-ti\-ca, la Probabilit\`{a} e le loro Applicazioni
(GNAMPA), Istituto Nazionale di Alta Matematica (INdAM). N.~Parolini and M.~Verani are members of Gruppo Nazionale per il Calcolo Scientifico (GNCS), Istituto Nazionale di Alta Matematica (INdAM). M.~Grasselli has been partially funded by MIUR PRIN research grant n. 2020F3NCPX.
M.~Verani has been partially funded by MIUR PRIN research grants n. 201744KLJL and n. 20204LN5N5. N.~Parolini has been partially funded by MIUR PRIN research grants n. 2017AXL54F and n. 20204LN5N5.

\end{document}